\documentclass[leqno]{amsart}

\usepackage{amssymb}
\usepackage{amsfonts}
\usepackage{amsthm}
\usepackage{amsmath}
\usepackage{bbm,bm}

\usepackage{xcolor,graphicx}
\usepackage[mathscr]{eucal}

\usepackage{enumitem}
\usepackage[unicode,bookmarks=false]{hyperref}
\usepackage{subcaption}

\newtheorem{theorem}{Theorem}[section]
\newtheorem{proposition}[theorem]{Proposition}

\newtheorem{lemma}[theorem]{Lemma}
\newtheorem{corollary}[theorem]{Corollary}

\theoremstyle{definition}

\newtheorem*{examples}{Examples}
\theoremstyle{remark}
\newtheorem{remark}[theorem]{Remark}

\numberwithin{equation}{section}

\newcommand{\bbone}{{\mathbbm{1}}}
\newcommand{\R}{\mathbb{R}}

\newcommand{\N}{\mathbb{N}}
\newcommand{\Z}{\mathbb{Z}}
\newcommand{\C}{\mathbb{C}}

\renewcommand{\hat}{\widehat}
\newcommand{\eps}{\varepsilon}

\newcommand{\scriptB}{\mathcal{B}}

\newcommand{\scriptF}{\mathcal{F}}

\newcommand{\scriptH}{\mathcal{H}}

\newcommand{\scriptO}{\mathcal{O}}

\newcommand{\scriptS}{\mathcal{S}}

\newcommand{\jp}[1]{\langle{#1}\rangle}

\DeclareMathOperator*{\supp}{supp}

\newcommand{\Be}{\begin{equation}}
\newcommand{\Ee}{\end{equation}}
\newcommand\lc{\lesssim}
\newcommand\gc{\gtrsim}
\newcommand\la{{\lambda}}

\newcommand\ka{{\kappa}}

\newcommand\Coi{{C^\infty_{\mathrm c}}}

\newcommand\bbC{{\mathbb {C}}}

\newcommand\bbH{{\mathbb {H}}}

\newcommand\bbN{{\mathbb {N}}}

\newcommand\bbR{{\mathbb {R}}}

\newcommand\cS{{\mathcal {S}}}

\newcommand\sK{{\mathscr {K}}}

\newcommand\ve{{\varepsilon}}
\newcommand\de{{\delta}}

\newcommand{\chr}{\bbone}
\newcommand{\g}{\mathfrak{g}}

\begin{document}

\title{Bochner--Riesz means on the Heisenberg group}

\author[D. Müller]{Detlef Müller}
\address{D. Müller, Mathematisches Seminar, C.A.-Universität Kiel, Heinrich-Hecht-Platz 6, 24118 Kiel, Germany}
\email{mueller@math.uni-kiel.de}

\author[L. Niedorf]{Lars Niedorf}
\address{L. Niedorf, Department of Mathematics, University of Wisconsin-Madison, 480 Lincoln Drive, Madison, WI 53706, USA}
\email{niedorf@wisc.edu}

\author[A. Seeger]{Andreas Seeger}
\address{A. Seeger, Department of Mathematics, University of Wisconsin-Madison, 480 Lincoln Drive, Madison, WI 53706, USA}
\email{seeger@math.wisc.edu}

\author[B. Stovall]{Betsy Stovall}
\address{B. Stovall, Department of Mathematics, University of Wisconsin-Madison, 480 Lincoln Drive, Madison, WI 53706, USA}
\email{stovall@math.wisc.edu}



\thanks{Research supported in part by  DFG grant MU 761/12-1 (D.M.), NSF grant 2348797 (A.S.) and Simons Foundation grant  SFI-MPS-SFM-00011865  (B.S.)}   

\begin{abstract} We prove new $L^p$ boundedness results for Bochner--Riesz means associated with the spectral decomposition of the sub-Laplacian on the  Heisenberg group $\mathbb H_n$. Our results hold for a range $1\le p\le p_n$ where $p_n\to 2$ as $n\to\infty$. As shown by the first named author in 1990 a Stein--Tomas type Fourier restriction theorem fails to hold on $\mathbb H_n$  and thus  previous  results  based on the approach by Fefferman and Stein from the Euclidean setting  only allowed to cover the cases $p=1$ and $p=\infty$. Our results on Bochner--Riesz means follow from a more general $p$-sensitive spectral multiplier theorem which is the main result of this article. This  is obtained as a  consequence of  $L^p$ estimates for square functions associated with the Heisenberg wave operator. 
\end{abstract}

\date{\today}
\maketitle

\tableofcontents

\section{Introduction}

Denote by  $\bbH_n$  the Heisenberg group of dimension $d=2n+1$. On $\bbH_n$ we are using  coordinates $(x,u)\in \bbR^{2n}\times \bbR$ and the group law 
\begin{equation}\label{eq:grouplaw} 
(x,u)\cdot  (x',u')= (x+x', u+u'+\tfrac 12 (Jx)^\intercal  x'),
\end{equation}
where $J=J_n$ denotes  the $2n\times 2n$ block-diagonal matrix $\sum_{k=1}^n (e_{2k} e_{2k-1}^\intercal - e_{2k-1} e_{2k}^\intercal)  $, i.e., the matrix with $n$  blocks  of the form $J_1=( \begin{smallmatrix} 0&-1\\ 1&0\end{smallmatrix})$. 
The purpose of this paper is to   establish  new spectral multiplier theorems for the \textit{sub--Laplacian} 
 \[L=-\sum_{j=1}^{2n} X_j^2\] 
on the Heisenberg group, one of the most basic subelliptic, non-elliptic examples of a  Hörmander sum of squares operator.
The vector fields $X_1,\dots,X_{2n}$ are given by
\[
X_{2k-1}=\partial_{x_{2k-1}}-\tfrac12 x_{2k}\partial_u,
\qquad
X_{2k}=\partial_{x_{2k}}+\tfrac12 x_{2k-1}\partial_u,
\qquad
k=1,\dots,n.
\]
The vector fields $X_1,\dots, X_{2n}$ and $U=\partial_u$ span the Lie algebra of $\bbH_n,$
 and the only non-trivial commutation relations among them are 
\[
 [X_{2k-1},X_{2k}]=U, \qquad k=1,\dots,n.
\]
In particular, they satisfy Hörmander's bracket condition for hypoellipticity of a sum of squares operator. 
Moreover, the sum of squares operator $L$ is essentially self-adjoint on $C^\infty_0(\bbH_n)$ (as is established  by the methods in \cite{NelsonStinespring}).  By the spectral theorem there is a functional calculus which defines the operator $m(\sqrt L)$  as a bounded operator on $L^2(\bbH_n)$ for any continuous bounded function $m$ on $[0,\infty)$. In $L^p$ spectral multiplier theory we care about conditions on $m$ under which  the operator $m(\sqrt L)$ extends to $L^p$ for $p\neq 2$.  Using the automorphic dilation structure \[ D_r: (x,u)\mapsto (rx, r^2 u)\] on the Heisenberg group the $L^p\to L^p$ operator norms of $m(r\sqrt L)$ for $r>0$ are independent of $r$ and therefore it is natural to formulate dilation invariant assumptions for multiplier transformations.

An effective $L^p$-result on spectral multipliers for the standard positive Laplacian $- \sum_{k=1}^d \partial_{x_k}^2$ in Euclidean space $\bbR^d,$ or for the Laplace--Beltrami operator  $-\Delta$ on a $d$-dimensional compact manifold,   is given by the classical Hörmander--Mikhlin multiplier theorem. One version of it states that  the multiplier transformations $m(r\sqrt{-\Delta}) $ extend  for any $r>0$  to $L^p$  for $1<p<\infty$,  with operator norm independent of $r$,  if the localized Sobolev conditions
\begin{equation}  \label{eq:L2Sob-loc} 
\sup_{t>0}  \| \eta \, m(t\cdot)\|_{L^2_\alpha} <\infty
\end{equation} 
are satisfied for $\alpha>d/2$; here $\eta\neq 0$ is a  $C^\infty$ function with compact support in $(0,\infty)$. The space of functions on $(0,\infty)$ defined by the finiteness of the norm in  \eqref{eq:L2Sob-loc} does  not depend on  the choice of $\eta$ (see e.g. \cite{CarberyGasperTrebels-JAT}).  By an interpolation argument  one can also replace in \eqref{eq:L2Sob-loc}  the space $L^2_{(d/2)+\eps}$ with $L^q_{(d/q)+\eps}$, where $\frac 1q=|\frac 1p-\frac 12|$ and $q>2$.

The above  Hörmander--Mikhlin-type result also holds for functions $m(\sqrt L)$ of the sub-Laplacian   on the Heisenberg group, with $d=2n+1$,  but this is  harder to prove. It was independently established  by the first author and Stein  \cite{MuellerStein-mult}   and by Hebisch  \cite{He93}, improving for  Heisenberg groups earlier results on general stratified groups by Christ \cite{Ch91} and  Mauceri--Meda \cite{MaMe90} that  involved in place of the Euclidean dimension $d$ the larger  homogeneous dimension $Q$ (e.g.  $d=2n+1$ for the group  $\bbH_n$,  but  $Q=2n+2$).

In radial Fourier multiplier theory on Euclidean spaces and spectral multiplier theory on compact  manifolds  sharper $p$-sensitive results are known for certain $p>1$ that cannot be obtained by interpolation between $p=1$  and $p=2$, and cover for  suitable $p$-ranges  essentially sharp $L^p$-bounds  on  Riesz means $(I-R^{-1} \sqrt{-\Delta})_+^\nu $  and on  the closely  related Bochner--Riesz means $(I+R^{-2} \Delta)_+^\nu$.
 On general compact manifolds such essentially optimal results for $m(\sqrt{-\Delta})$ are known for the  range  $1<p\le \frac{2(d+1)}{d+3}$ and its dual range, which corresponds to the range of the Stein--Tomas Fourier restriction theorem (\cite{tomas, SteinBeijing})  and its analogue by Sogge \cite{Sogge-spectral1988} on compact manifolds.
Specifically,  if $1<p<\frac{2(d+1)}{d+3}$ and \eqref{eq:L2Sob-loc} holds for $\alpha>d\,|\frac 1p-\frac 12|$ then $m(\sqrt{-\Delta})$ is bounded on $L^p$ (see \cite{SeegerSogge-mult}).  
 We remark that much better $p$-ranges beyond the  Stein--Tomas theory are known to hold in the Euclidean case. In particular,  the optimal multiplier results under condition \eqref{eq:L2Sob-loc} are directly  related via \cite{CarberyGasperTrebels} to results on a square function by Stein involving Bochner--Riesz means, for which the optimal results in two dimensions are due to Carbery \cite{Carbery1983}, and the current best results in higher dimensions are in  Gan--Oh--Wu \cite{GanOhWu}. The latter paper contains many more references. 

Turning to the analogous  question about  $p$-sensitive multiplier theorems for the sub-Laplacian  on the Heisenberg groups, no such optimal   improvements beyond interpolation have been  known for any $p\in (1,2).$ The reason for this state of affairs can be traced back to a result by the  first author in \cite{Mueller-Annals} which is  concerned with restriction and spectral projection  results for the sub-Laplacian. Specifically,   it was shown in \cite{Mueller-Annals} that  for large $M$
\begin{equation}\label{eq:Muller-equiv}
\big\| \bbone_{[1-M^{-1},1]}  (\sqrt L)\big\|_{L^p(\mathbb H_n) \to L^2(\mathbb H_n)} \sim M^{\frac 12-\frac 1p} , \qquad 1\le p\le 2,
\end{equation} 
while a Stein--Tomas type restriction theorem for $L$ would imply the stronger bound $O(M^{-\frac{1}{2}})$ for some $p>1$. 

In view of this outcome there cannot be, even for compactly supported multipliers,  an effective approach to sharp  $L^p$ multiplier theorems for $p>1$  based on spectral projection bounds \eqref{eq:Muller-equiv}, such as is suggested by the method established  by  Fefferman \cite{Fefferman-thesis}, with simplification by Stein \cite{FeffermanBR}. See also \cite{ChristProc, Seeger-Crelle1986} and then  \cite{SoggeBR, SeegerSogge-mult} for results on compact manifolds, with further extensions in \cite{ChenOuhabazSikoraYan}.   The first author used weaker  restriction theorems on the Heisenberg group to prove estimates for Riesz means in certain mixed norm spaces \cite{Mueller-Riesz-means}, but no sharp  results in $L^p$ spaces have been known for $p>1$. We note that the failure of the restriction estimate rules out a certain standard  method of proof and does not disprove  any   sharp Bochner--Riesz result itself  (contrary to  what might be  inferred from the title of  \cite{Tao-Duke99}),  but it shows that  a different approach is needed.

Here we present  such an approach and establish  a satisfactory multiplier result for functions of the sub--Laplacian $L$ in a range $1<p<p(n)$ where $p(n)\to 2$ as $n\to \infty$. As the operators $m(\sqrt{L})$ are self-adjoint  we also have $\smash{L^{p'}}$ boundedness for $p'=\frac{p}{p-1} $. In what follows, keep in mind that \[d=2n+1\] denotes the topological dimension of $\bbH_{n}$.  

We would like to point out that, at least by our approach, the $p$-range in our results below is determined more by the dimension $d_1=2n$ of the first, horizontal layer of the Heisenberg group than by the dimension $d$; see the discussion below. The $p$-range in  our next theorem would then read $1<p\le \frac{2(d_1+2)}{d_1+5}.$ We expect that an analogue of Theorem \ref{thm:Lp-mult} will hold  on groups of Heisenberg type in the same range (see the according discussion at the end of this section). 

\begin{theorem}\label{thm:Lp-mult}
Let $n\ge 1$, and let $p$ be in the range $1<p\le \frac{2(d+1)}{d+4}$, or its dual range $\frac{2(d+1)}{d-2}\le p<\infty$. Assume that $\alpha>d\,|\frac 1p-\frac12|$. 
Then, for all $f\in L^p(\bbH_n)$,
\[
\|m(\sqrt{L})f\|_{L^p(\bbH_n)}
 \le C_{p,\alpha,\eta}\,  \sup_{t>0} \|\eta \, m(t\cdot) \|_{L^2_\alpha} \, \|f\|_{L^p(\bbH_n)} .
\]   
\end{theorem}

Note that the functions $\tau\mapsto(1-|\tau|)_+^\nu$  belong to the Sobolev space $L^2_\alpha$ if $\alpha <\nu+\frac 12$.  
This, together with standard arguments can be used to deduce $L^p$ convergence for the Riesz means  
and the Bochner--Riesz means of an $L^p$ function. Martini \cite{Martini25} showed that $L^p$ boundedness for $(1-L)_+^\nu$ fails for $\nu<d\,|\frac 1p-\frac 12|-\frac 12$.

Thus, we get an essentially sharp result in the range $p\le \frac{2(d+1)}{d+4}$ (and, by duality, in the corresponding dual range):

\begin{corollary} \label{BR-cor}
Let $p$ be in the range $1\le p\le \frac{2(d+1)}{d+4}$, or its dual range, and let $\nu>d\,|\frac 1p-\frac 12|-\frac 12$.
Then the Bochner--Riesz means $(1-R^{-2} L)_+^\nu$ and the Riesz means $(1-R^{-1} \sqrt L)_+^\nu, \ R>0,$ are uniformly bounded on $L^p(\bbH_n)$.
\end{corollary} 
 
The  essential tool in the proof of Theorem~\ref{thm:Lp-mult} is a square function estimate for the half wave operator $e^{it\sqrt L} $. This approach is motivated by the equivalence of square function bounds for Bochner--Riesz means and for derivatives of spherical averages in the Euclidean case,  as demonstrated by Kaneko and Sunouchi \cite{KanekoSunouchi}.

\begin{theorem}\label{thm:main-square}
Assume that $I$ is a compact subset of $\bbR\setminus \{0\}$ and $p> \frac{2(d+1)}{d-2}$.
Then for  $\beta\in C^\infty_c(\bbR)$ and $\la\ge 1,$
\begin{equation}\label{eq:main-square}
\Big\| \Big( \int_I \big| e^{i t\sqrt{L}} \beta(\tfrac{t\sqrt{L}}{\la}) f \big|^2 \, dt \Big)^{\frac 12} \Big\|_{L^p(\bbH_n)} \le C_{p,\beta}  \la^{d(\frac 12-\frac 1p)-\frac 12} \|f\|_{L^p(\bbH_n)}
\end{equation}
for all $f\in L^p(\bbH_n)$. 
\end{theorem}

Note that the endpoint $p=\frac{2(d+1)}{d-2}$ is excluded from Theorem~\ref{thm:main-square}, but is included in Theorem~\ref{thm:Lp-mult} and Corollary~\ref{BR-cor}. Nevertheless, since the regularity conditions on $\alpha$ and $\nu$ above are given by strict inequalities, the corresponding endpoint estimates can be recovered by interpolation with the $L^2$ bounds, which follow directly from the spectral theorem.

The result in Theorem~\ref{thm:main-square} involves a better exponent than  the sharp fixed time estimate 
\begin{equation}\label{eq:fixedtime} \big\|e^{i t\sqrt{L}} \beta(\tfrac{t\sqrt{L}}{\la}) f\big\|_{L^p(\bbH_n)} \le C_{\beta,s} \, \la^{(d-1)(\frac 12-\frac 1p) } \|f\|_{L^p(\bbH_n)}
\end{equation}
which holds for $2\le p<\infty$; this follows from a result in \cite{MuellerSeeger2015}  (for previous non-endpoint results see also \cite{MuellerStein-wave}).   Analogs of \eqref{eq:main-square} are well known in the Euclidean case and follow from corresponding $L^2\to L^p$ inequalities for the square function. In contrast  the $L^2\to L^p$ versions for the sub-Laplacian  are not effective enough to yield Theorem~\ref{thm:main-square}. 
The reason is that the singularities for the convolution kernel of the wave kernel are far worse than in the Euclidean case, due to the complex underlying sub-Riemannian geometry associated with the sub-Laplacian. 
The  singular support of the convolution kernel 
is known to coincide with the geodesic sphere of radius $|t|$ with respect to the sub-Riemannian metric associated to $L$. This (together with asymptotics of the kernel in some regimes) was first described by Nachman \cite{Na82}.
The geodesic sphere  is  invariant under vertical reflections and horizontal rotations in $\bbR^{2n}$. Its $(|x|,|u|)$-profile for the case $t=1$ is shown in Figure \ref{fig1}.

\begin{figure}[ht]
 \centering
 \includegraphics[width = 0.38\linewidth]{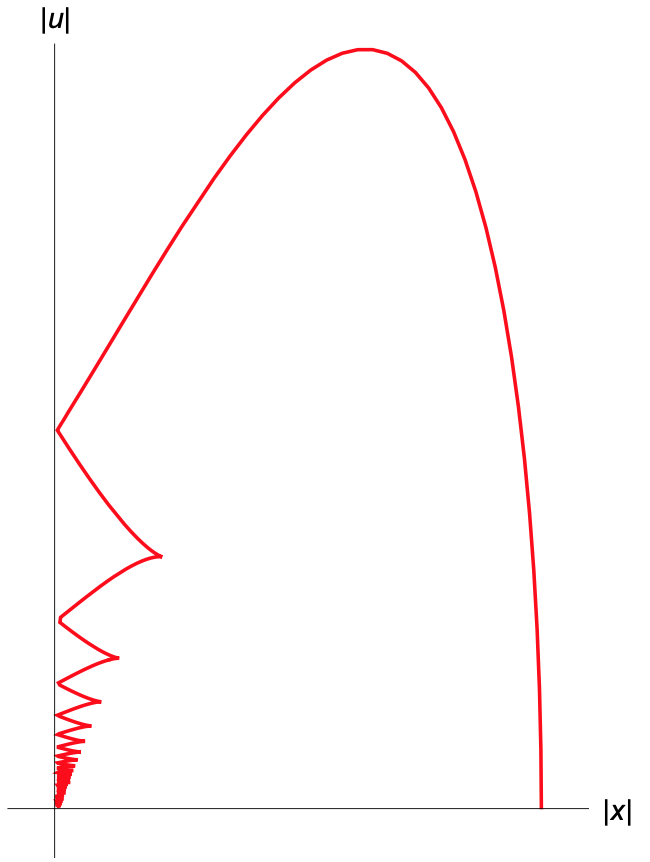}
 \caption{A plot of the singular support  $\Sigma_1$ of the convolution kernel of $e^{i\sqrt L}$. The curve admits infinitely many zigzags near the origin.}
 \label{fig1}
\end{figure}

Our analysis uses a more precise parametrix derived by the first and third authors in \cite{MuellerSeeger2015}; this involves a subordination argument using the Schrödinger type operators $e^{itL}$, and a decomposition in the joint spectrum of $L$ and $i\partial_u$. To begin with, we note that the proof of Theorem~\ref{thm:main-square} can easily be reduced to the case where $\beta$ is supported in $[\frac12,2]$ (cf. Section \ref{conseq}). Assuming this,  we have, for Schwartz functions $f\in \cS(\bbH_n)$,
\begin{equation}\label{eq:wavedec} e^{i t\sqrt{L}} \beta(\tfrac{t\sqrt{L}}{\la}) f= T^0_{\la,t}f +
\sum_{\substack {0<|k|\le 8\lambda} } \sum_{\ell\ge 1} T^{k,\ell}_{\la,t} f,
\end{equation}
where the  operators $T^0_{\la,t}$, $T^{k,\ell}_{\la,t}$ with convolution kernels $\sK^0_{\la,t}$, $\sK^{k,\ell}_{\la,t}$ are defined in  \eqref{Tklat-definition}, \eqref{Kkllat}.
The convolution kernels $\sK^0_{\la,t}$ are concentrated near the outer part in Figure \ref{fig1}. For the analysis of $T^{k,\ell}_{\la,t}$ we may assume that $k>0$, by symmetry considerations. 
For $k>0$,  the convolution kernels  $\sum_{\ell\ge 1} \sK_{\la,t}^{k,\ell}$  are concentrated  in a
$(Ck^{-1}, Ck^{-2} )$ box at distance $\sim k^{-1}$  from the horizontal plane; it corresponds to the $k$th zigzag. Moreover, for $\ell\ge 1$,  $\sK_{\la,t}^{k,\ell}$ is concentrated in a  part of this box at distance $\sim 2^{-\ell} k^{-1}$ from the $|u|$-axis (these are the boxes $D^{k,\ell}$ in Figure \ref{fig2} below). 
When acting on functions supported near the origin we may  consider $T^{k,\ell}_{\la,t}$ as a parabolically rescaled Fourier integral operator with non-homogeneous phase (see \eqref{eq:Klatkl-formula}), with the additional complication that the phase functions  become very singular for large values of $\ell$. The following theorem 
is  the key auxiliary result in the proof of Theorem~
\ref{thm:main-square}. 
\begin{theorem} \label{mainTklthm}
Assume $p>\frac{2d}{d-2}$.  Then 
\begin{equation} \label{eq:T0la-sq}
\Big\| \Big(\int_1^2 \big| T^{0}_{\la,t} f\big|^2 dt\Big)^{\frac 1 2} \Big\|_p \lc  \la^{d(\frac 12-\frac 1p)-\frac 12} \|f\|_p.
\end{equation}
Moreover, for $|k|\le 8\la$, $2^\ell|k|\lc \la$,
\begin{equation} \label{eq:Tkla-sq}
\Big\| \Big(\int_1^2 \big| T^{k,\ell}_{\la,t} f\big|^2 dt\Big)^{\frac 1 2} \Big\|_p \lc \big[ (2^\ell |k|)^{\frac{d+1}{p}-\frac d2} \la^{d(\frac 12-\frac 1p)-\frac 12}  +(2^\ell |k|^{-1}) \big] \, \|f\|_p.
\end{equation}
\end{theorem} 

For $2^\ell|k|\gtrsim \la$, there are also  more elementary  square function estimates which rely  on  $L^1$, or equivalently $L^\infty$, boundedness results from \cite{MuellerSeeger2015},
\begin{equation} \label{eq:Linftybounds} \|T^{k,\ell}_{\la,t} \|_{L^\infty\to L^\infty} \lc  (2^\ell|k|)^{-1} \la^{\frac{1}{2} }, 
\end{equation} and  additional cancellations for the $L^2$ square function bounds; see Proposition~\ref{prop:near}. We note that \eqref{eq:Linftybounds} leads to  favorable bounds in the range not covered by Theorem~\ref{mainTklthm}, namely $2^\ell |k|\gc \la$, and indeed even in a somewhat larger $p$-range. 

The estimates \eqref{eq:Tkla-sq} must be summed over $k,\ell$ in the relevant ranges. To arrive at the conclusion \eqref{eq:main-square} we then need the more restrictive condition $\smash{\frac{d+1}p-\frac d2<-1}$, i.e., $p>\frac{2(d+1)}{d-2}$, which is the  $p$-range assumed in Theorem~\ref{thm:main-square}.

\subsubsection*{On the Fourier integral methods used  in the proof of Theorem~\ref{mainTklthm}.} 

The estimates for the above mentioned Fourier integral operators rely on curvature conditions in the frequency variables formulated in \cite[Sec.\ 2]{Mockenhaupt-See-So93}  and \cite[Prop.\ 3.4]{Mockenhaupt-See-So93}; we refer to this  
as the \textit{rank $r$ cone condition} where $r\le d-1$. These are homogeneous versions of the Carleson-Sjölin condition (\cite{CarlesonSjolin, Hoermander1973, SteinBeijing}). For averages over curves in the plane the cone condition  corresponds to the \textit{cinematic curvature condition} which was proposed previously by Sogge in his  work on maximal functions for variable averages in two dimensions \cite{Sogge91} (for recent extensions see also   \cite{LeeLeeOh25, ChenGuoYang}). For wave equations associated with elliptic Laplacians,  the square function estimates of the form \eqref{eq:main-square} are consequences of an $L^2\to L^p(L^2)$  estimate \cite[Thm.\ 3.2]{Mockenhaupt-See-So93}, which  covers   classes of Fourier integral operators satisfying  the maximal rank $(d-1)$ cone condition;   in this case one has  estimates in    the Stein--Tomas range $p\ge \frac{2(d+1)}{d-1}$ associated to the space dimension $d$.

In contrast, for our  Fourier integral operators on the Heisenberg group the rank $(d-1)$ cone condition is not satisfied; instead only a rank $d_1-1=d-2$ cone condition holds (see the discussion in Appendix~\ref{cinecurv}), where $d_1$ denotes the dimension $d_1=2n$ of the first, horizontal layer of the Heisenberg group. While \cite{Mockenhaupt-See-So93} covers Strichartz $L^2\to L^p(L^p)$ estimates for  lower rank cone conditions it does not cover optimal  $L^2\to L^p(L^2)$ estimates in these situations. Indeed, the proof of square function results under the same  lower rank cone condition poses substantial technical difficulties  as  our proof of  \eqref{eq:T0la-sq}  for the best behaved family $\{T^0_{\la,t}\}$ already illustrates. Additional complications arise for the square functions associated with the operators $\smash{T^{k,\ell}_{\la,t}}$, for $k,\ell\ge 1$. Note that, for each fixed $\smash{T^{k,\ell}_{\la,t}}$, the square function estimate \eqref{eq:Tkla-sq} holds in the Stein--Tomas range $p>\frac{2(d_1+1)}{d_1-1}$ associated with the dimension $d_1$. However, the constants in these estimates grow with $k$ and $\ell$, leading to the slightly smaller $p$-range in Theorem~\ref{thm:main-square}.

In principle, after suitable rescalings, we adapt in all cases the $TT^*$ method from \cite{Kapitanski, Ginibre-Velo92, Mockenhaupt-See-So93}. This   uses  certain variants of $TT^*$ where we  freeze one of the space   variables. It turns out that in our problem, unlike in the classical applications, the standard canonical graph assumptions that lead  to optimal $L^2$ bounds for these frozen operators are not satisfied uniformly in $k$ and $\ell$. Consequently, the resulting constants in \eqref{eq:Tkla-sq} are  less favorable for large $k$ and $\ell$ than one might hope for. The lack of uniformity in $k$ requires more sophisticated $L^2$ estimates,  which make use of the special structure of the phase functions. 

\begin{figure}[ht]
  \centering
  \begin{subfigure}{0.49\textwidth}
    \centering
    \includegraphics[width=\linewidth]{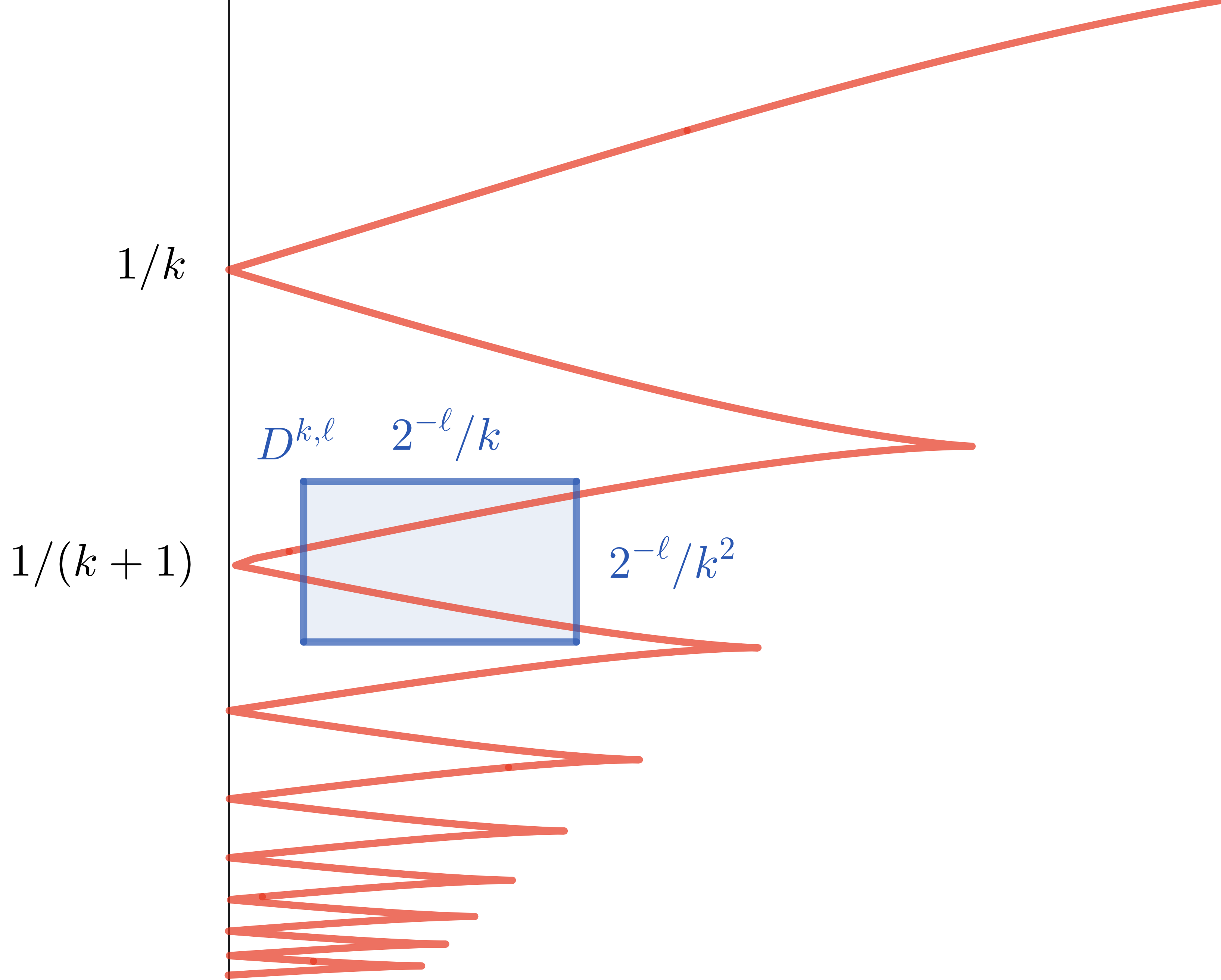}
    \caption*{Case I}
  \end{subfigure}
  \begin{subfigure}{0.49\textwidth}
    \centering
    \includegraphics[width=\linewidth]{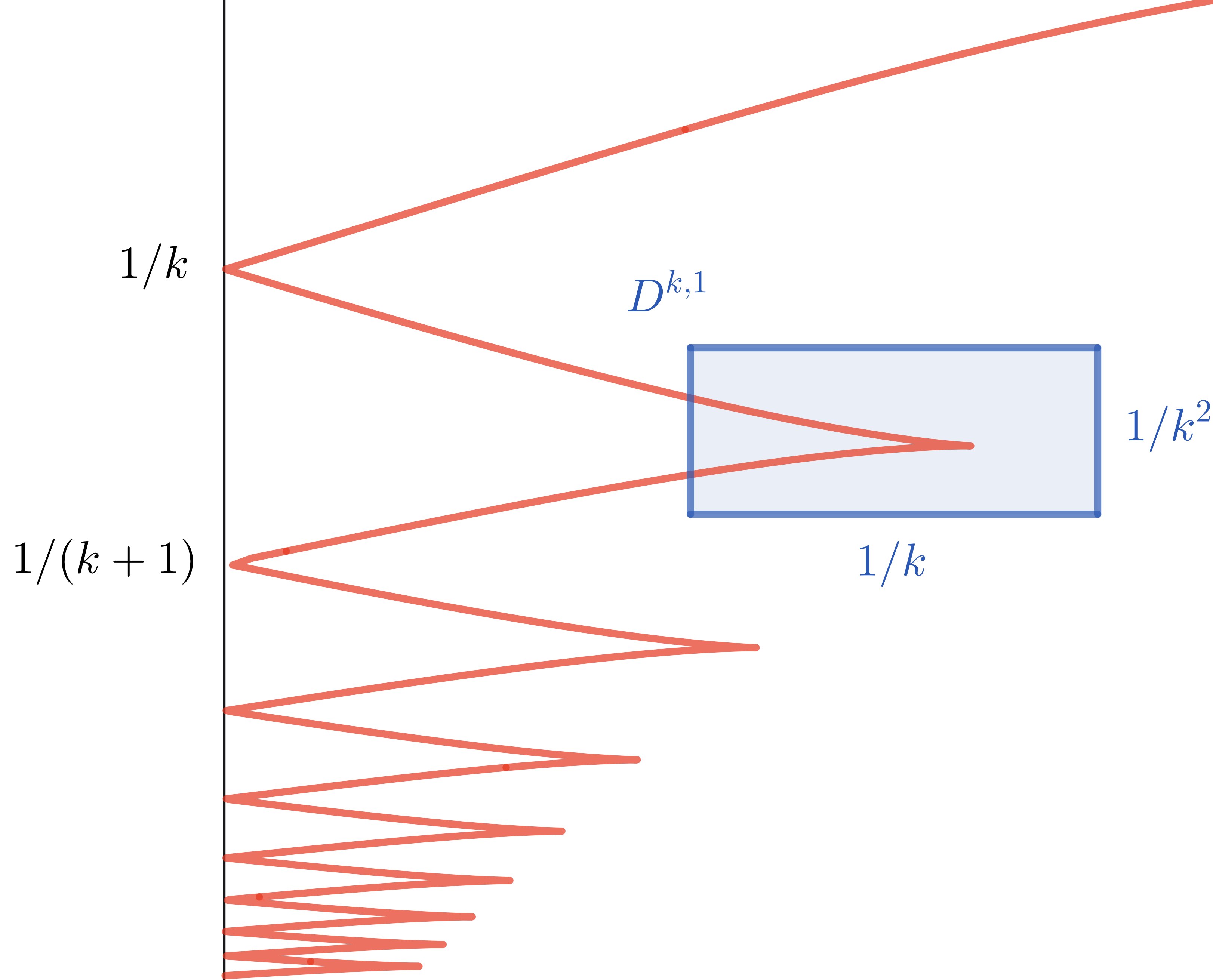}
    \caption*{Case II}
  \end{subfigure}
\caption{The localization regions $D^{k,\ell}$ associated with the $k$th zigzag of the $(|x|,|u|)$-profile of the singular support $\Sigma_1$. Case~I corresponds to $\ell\ge 2$, and Case~II to $\ell=1$.}
   \label{fig2}
\end{figure}

The choice and the further analysis  of such frozen operators  will be  different in the cases $\ell\ge 2$ and $\ell=1$  (corresponding to Case I and Case II in Figure \ref{fig2}, respectively). This phenomenon is  linked to the different ways in which  the above mentioned cone condition in \cite{Mockenhaupt-See-So93} can be interpreted as a curvature condition. In the  paper by Sogge \cite{Sogge91} on  averages over variable plane curves  this condition  was  given  in terms of a change of curvature  in time  (hence the use of the  descriptive terminology \textit{cinematic} curvature). However this interpretation breaks down at several points in Figure  \ref{fig3} where the graph has a horizontal tangent line. These  points correspond to the special  choices $\tau_j=j\pi+\pi/2$ in the  frequency variable $\tau$ parametrizing the oscillatory integrals in  \eqref{eq:Klat-formula}, \eqref{eq:Klatkl-formula} below  and are located at the top of the figure and  close to the cusps.  Thus there are  $(d_1-1)$-dimensional spheres  in the singular support where the tangent hyperplanes to the singular support are all horizontal. Since the  affine horizontal planes will remain horizontal under application of  the automorphic dilations   there is no longer  a change of curvature in time. Indeed,   the presence of these  horizontal tangent planes  suggests the absence of even  a lower rank cinematic phenomenon; nevertheless the rank $(d_1-1)$ cone condition still holds, due to  a  certain interplay between an  actual  curvature condition for the singular supports  and the rotational curvature \cite{PhongStein86} implicit   in the Heisenberg group law.  This insight has significant consequences for our parametrizations of the associated phase functions for these operators  and also for  our choice of frozen variables in the analysis of various  $TT^*$ arguments. We note that a  related phenomenon also occurs in the  paper \cite{RSS22} by Roos, Srivastava and the third author,   which studies certain families of $\ka$-dimensional submanifolds of Heisenberg type groups with $\kappa<d-1$,   where a cone condition of rank $d-1$  still holds.  
 
\begin{figure}[ht]
\centering
\includegraphics[width=0.9\linewidth]{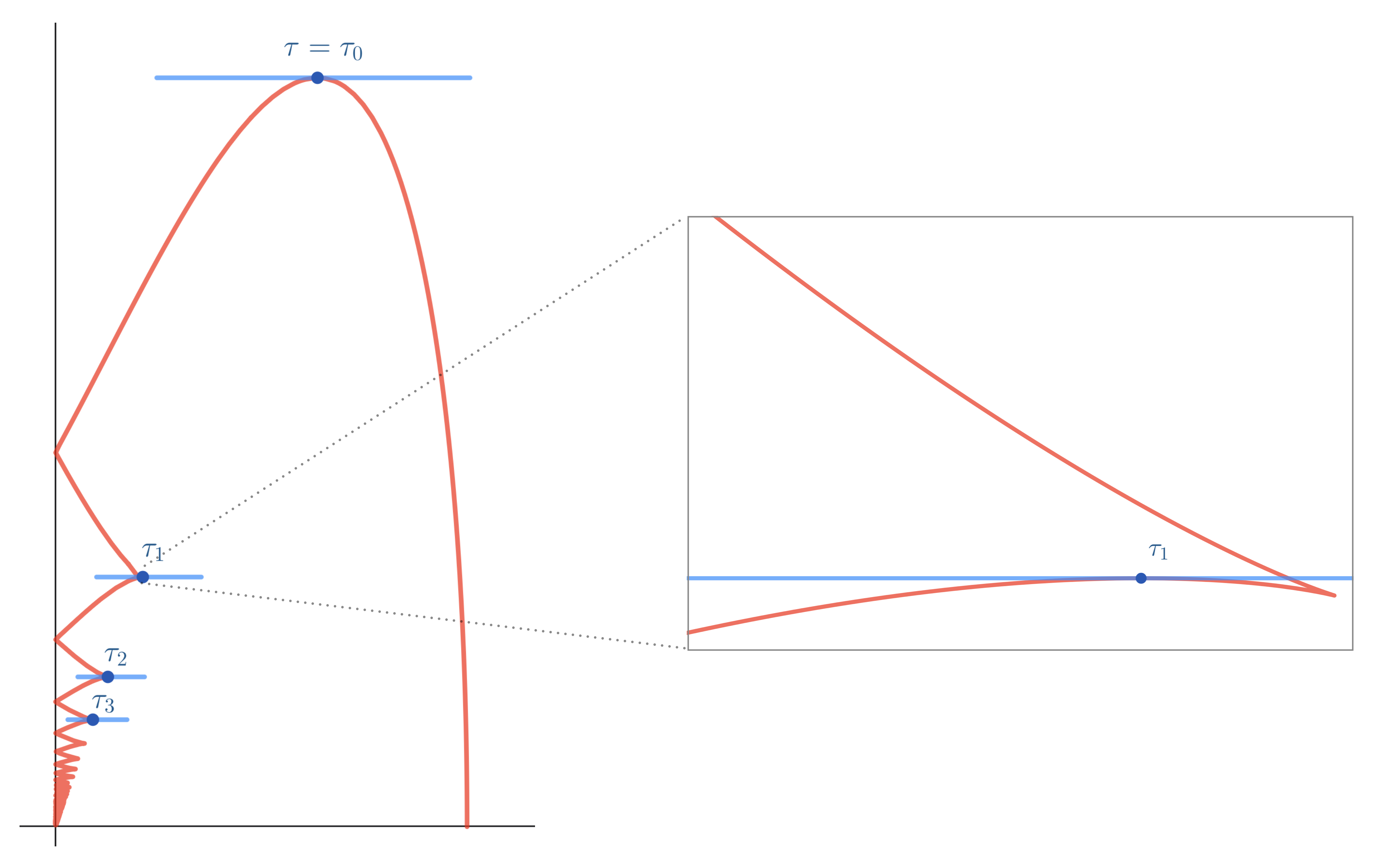}
\caption{A plot indicating horizontal points of the singular support $\Sigma_1$  of the convolution kernel of $\smash{e^{i\sqrt L}}$.  Note that these horizontal points are distinct from the cusp points visible in the plot.}
\label{fig3}
\end{figure}

\subsubsection*{A guide through the paper} 
In Section \ref{conseq}, we  show how the square function estimates from Theorem~\ref{thm:main-square} imply the   spectral multiplier results in Theorem~\ref{thm:Lp-mult} and Corollary~\ref{BR-cor}.
In Section \ref{sec:localization} we study the reduction of the proof of Theorem~\ref{thm:main-square} to  localizations based on an abstract localization result for left-invariant operators, on finite propagation speed for $k=0,$ and on further localizations which are made possible by means of  error estimates for the $T^{k,\ell}_{\la,t}$ for $k\ne 0$. The error estimates can be found in Appendix \ref{decayproof} (see also \cite{MuellerSeeger2015} for a different treatment).
For $k\neq 0$ these localizations force us to decompose the time intervals into smaller intervals of length comparable to the diameter of the $D^{k,\ell}$ in Figure \ref{fig2}.
In Section \ref{prelimwave} we present some preliminaries about the solution operator for the Heisenberg wave equation, including the decomposition \eqref{eq:wavedec} and further information. The proof of the square function bound \eqref{eq:T0la-sq} involving the operators $T^0_{\la,t}$ will be given in Sections \ref{sec:k=0-preliminaries}, \ref{sec:k=0-nonhor} and~\ref{sec:k=0-hor}. 
The proof of \eqref{eq:Tkla-sq} in the cases $k\ge 1$, $\ell\ge 1$ will be given in the remaining Sections~\ref{sec:kneq0-preliminaries}--\ref{CaseII}. Appendix~\ref{cinecurv} contains a calculation regarding the rank $(d_1-1)$ cone condition.

\subsection*{Further directions and open problems} 

\subsubsection*{Extensions to Heisenberg type groups}
We expect that the methods of the current paper also give improvements to known spectral multiplier results on a  Heisenberg type group $G$  with $d_2$-dimensional  center, where $d_2\ge 2$.  These exist for every $d_2$, but we note  that  if $d_1=d-d_2$ is the dimension of the first, horizontal layer of the group $G$, then  the Radon--Hurwitz formulas yield  $d_2\lc \log d_1$ (for more information see \cite{AdamsLaxPhillips, Kaplan1980}). The second author \cite{Ni24} proved an analogue of Theorem~\ref{thm:Lp-mult} in the range $p\le \frac{2(d_2+1)}{d_2+3}$ for $\alpha>d\,|\frac 1p-\frac 12|$ relying on Stein--Tomas  Fourier restriction type estimates (from Casarino--Ciatti \cite{CasarinoCiatti}, Liu--Wang \cite{LiuHeping-WangYingzhan} and Chen--Ouhabaz \cite{ChenOuhabaz}, see also \cite{Ni23-restr}). Note that this method does not yield multiplier results on the standard Heisenberg groups where $d_2=1$.  It is expected that the methods in our  paper can be further developed to substantially extend the  $p$-range for general Heisenberg type groups.

\subsubsection*{Further extensions to Métivier groups}
In \cite{Ni23-restr}, \cite{Ni25Studia} the second author also extended the  restriction and multiplier results to Métivier groups with higher dimensional center (again obtaining an analog of Theorem~\ref{thm:Lp-mult} in the range $1\le p\le  \frac{2(d_2+1)}{d_2+3}$, $d_2\ge 2$, at least in the case $(d_1,d_2)\notin\{(8,6), (8,7)\}$). For Métivier groups that are not of Heisenberg type the parametrix construction from \cite{MuellerSeeger2015} used in the current paper is  no longer available which makes further improvements beyond this range  very difficult. A new approach to describe solutions of the appropriate wave equation, with a partial parametrix,   was developed by Martini and Müller in \cite{MartiniMueller-Metivier}, relying on Fourier integral operators with complex phase.   Their approach yields fixed-time estimates of the kind \eqref{eq:fixedtime} (with a possible loss of $\la^\eps$).  It would be very interesting to prove new  improved  $p$-sensitive multiplier theorems in the  Métivier case. 

\subsubsection*{Endpoint bounds}
Sharp  estimates for wave multipliers on the Heisenberg group  can be found in \cite{MuellerSeeger2015}, and for other oscillating multipliers  in \cite{Bramati-etal}. It would be very interesting to find more general endpoint versions of Theorem~\ref{mainTklthm} (covering among other things the oscillatory multiplier results) which mirror corresponding   results 
for the Laplacian on  compact manifolds \cite{See91-Indiana, KimJongchon18}.

\subsubsection*{Improving the $p$-ranges}
The modern powerful methods in Euclidean Fourier analysis used in Fourier restriction and radial multiplier theory beg to be  used to extend the $p$-range in Theorem~\ref{thm:main-square},  and consequently the $p$-ranges in Theorem~\ref{thm:Lp-mult}.  This is an  interesting direction for further  research.

\subsubsection*{Local smoothing phenomena}
We conjecture that for suitable ranges of $p\gg 2$ the inequality  \eqref{eq:main-square}  can be strengthened  by replacing the  $L^2([1,2])$ norm with  an $L^p([1,2])$ norm. This is  analogous to Sogge's local smoothing conjecture \cite{Sogge91} for the wave equation associated with an elliptic Laplacian. Such an inequality  would also imply new $L^p$ spectral multiplier theorems beyond the spaces defined in \eqref{eq:L2Sob-loc}.

\section{Consequences of the square function estimate}\label{conseq}
We  first show that  Theorem~\ref{thm:main-square} implies Theorem~\ref{thm:Lp-mult} for multipliers that are compactly  supported in $(0,\infty)$, and then provide some references for the more general situation. The results and proofs in this section and  Section \ref{sec:localization}  are valid for any left-invariant homogeneous sub-Laplacian $L$ on any Carnot group $G$ of topological dimension $d$ in place of $\mathbb H_n.$

 If $T : \scriptS(G) \to  \scriptS'(G)$ is a bounded left-invariant linear operator, then the Schwartz kernel theorem ensures the existence of  a tempered distribution $k_T\in  \scriptS'(G)$ such that 
 \begin{equation}\label{convkernel}
T \phi = \phi * k_T, \qquad \phi\in \ \scriptS(G),
\end{equation}
where $*$ denotes the group convolution on $G.$ We can then write 
$$k_T=T\de_0$$ (and shall often do so).
In particular, by Hulanicki's result   \cite[Thm.\ 2.4]{Hu84}  we know that $m(L)\delta_0\in  \scriptS(G)$ for $m\in  \scriptS(\bbR)$.

\begin{proposition}\label{prop:compact} Let $p\ge 2$, $\alpha_0>0$, $0<\eps<1$, and
$\beta\in \Coi((0,\infty))$, with $\beta\neq 0$.
Assume that the $L^p$ inequality 
\begin{equation} \label{eq:sqfct-ass}
\Big\|\Big(\int_1^{1+\varepsilon}|e^{it\sqrt{L}}  \beta(\tfrac{t\sqrt{L}}{\la} ) f|^2\,dt\Big)^{\frac 1 2} \Big\|_p \le C_0(1+\la)^{\alpha_0-\frac 12} \|f\|_p
\end{equation} holds for all $\lambda\ge 1$ and all
$f\in \cS(G)$. 
Under these hypotheses, if $m:\R\to\C$ is supported in a compact subinterval $I$ of $(0,\infty)$  and $m \in L^2_\alpha(\R)$ for some $\alpha> \alpha_0$, then $\smash{m(\sqrt L)}$ extends to a bounded operator on $L^p(G)$ with
\begin{equation}\label{eq:compact}
\|m(\sqrt L)\|_{p\to p} 
\lesssim_{I,\alpha,\beta} C_0 \|m\|_{L^2_\alpha(\R)}.
\end{equation} 
Here $m(\sqrt L)$ on the  left-hand side can also be replaced by $m(L)$.
\end{proposition}

\begin{proof} 
Making $\eps>0$ in the assumption smaller if necessary we may assume that 
there is $r_\circ>0$ 
such that $\beta(r)\neq 0$ on the  interval $(r_\circ(1-\eps), r_\circ(1+\eps) )$. We first show \eqref{eq:compact} under the assumption
that $m$ is supported in $(1-\frac{1}{8N}, 1+\frac{1}{8N})$ where $N>\eps^{-1}$. Let $\chi\in \Coi$ with $\chi(r)=1$ for $|r-1|\le \frac{1}{8N}$ and $\chi$ supported in $(1-\frac{1}{6N}, 1+\frac{1}{6N})$. 

Let $\ell_0=\max\{\ell \in \Z: 2^\ell < r_\circ\}$.
Using the Fourier inversion formula (with $\widehat m(\rho) =\int_\R m(r)e^{-ir\rho} dr$) we write 
\[ m(\sqrt{L}) 
 = \frac 1{2\pi} \int_\R \, \hat{m}(\tau ) \, e^{i\tau  \sqrt{L}} \, \chi(\sqrt{L}) \, {d\tau  } = u_0(\sqrt L)+\sum_\pm \sum_{\ell > \ell_0} u_{\ell,\pm}  (\sqrt L), \]
  where, with $J_0 = [-2^{\ell_0+1}, 2^{\ell_0+1}]$, and,  $J_\ell^- = [-2^{\ell + 1}, -2^\ell] $, and 
  $J_\ell^+=  [2^\ell, 2^{\ell + 1}]$ for $\ell>\ell_0$, 
 \[ u_{\ell,\pm}  (\sqrt L)= \frac{1}{2\pi}
\int_{J_{\ell}^{ \pm } }\hat{m}(\tau ) \, \chi(\sqrt{L}) e^{i\tau \sqrt{L}} \, {d\tau } . 
\] The multiplier $u_0$ is analogously defined, replacing $J_\ell^\pm$ by $[-2^{\ell_0+1},2^{\ell_0+1}]$. Since $u_0$ is a smooth compactly supported function, the inequality
\[
\|u_0(\sqrt L) f\|_p\lc \|m\|_{L^2} \|f\|_p
\]
follows from Hulanicki's theorem. 

We shall now focus on the multipliers $u_{\ell,+}$ for $\ell>\ell_0$. 
We decompose
\[
J_\ell^+=\bigcup_{\nu=1}^{4N} J_{\ell,\nu},
\]
where $J_{\ell,\nu}=[R_{\ell,\nu}, R_{\ell,\nu+1}]$ with 
$R_{\ell,\nu} = 2^\ell (1+(\nu-1) \frac{1}{4N})$. Note that $R_{\ell,\nu}^{-1} J_{\ell,\nu} \subset [1, 1+\frac{1}{4N}]$.  
Use Cauchy--Schwarz and a change of variable in $t$ to estimate
\begin{align*}
& \Big| \int_{J_{\ell}^{+} }\hat{m}(\tau ) \, e^{i\tau \sqrt{L}} \, \chi(\sqrt{L}) \, f \, d\tau  \Big| \\
& \le \Big( \int_{J_{\ell}^{+}} |\hat{m}(\tau )|^2 \, d\tau  \Big)^{\frac 12} \Big( \int_{J_{\ell}^+ }\big| e^{i\tau  \sqrt{L}} \, \chi(\sqrt{L}) \, f \big|^2 \, d\tau  \Big)^{\frac 12} \\
& \lesssim 2^{- \ell \alpha} \, \| m \|_{L^2_ \alpha} 
\Big( \sum_{\nu=1}^{4N}\int_{J_{\ell,\nu} }\big| e^{i\tau  \sqrt{L}} \, \chi(\sqrt{L}) \, f \big|^2 \, d\tau  \Big)^{\frac 12} 
\\ &\lc 2^{-\ell \alpha}  \| m \|_{L^2_ \alpha} 
\Big( \sum_{\nu=1}^{4N}\int_{1}^{1+\frac{1}{4N}} \big| e^{iR_{\ell,\nu} t \sqrt{L}} \, \chi(\sqrt{L}) \, f \big|^2 \, R_{\ell,\nu}dt \Big)^{\frac 12} .
\end{align*} 
Let $\chi_1\in \Coi$ be supported in $(1-\frac{3}{4N}, 1+\frac{3}{4N})$ so that  $\chi_1(r)=1$ for $|r-1|\le \frac{1}{2N}$. If $|r-1|\le \frac{1}{6N}$ and $|t-1|\le \frac{1}{4N}$ then $\chi_1(tr)=1$;  hence $\chi_1(t\cdot )\chi=\chi$  for all $t\in [1, 1+\frac{1}{4N}]$. 
Since $\beta(r_\circ r)\neq 0$ for $|r-1|\le N^{-1}$  we see that $\chi_2(r)= \chi_1(r)/ \beta(r_\circ r)$ defines a function in $C^\infty_0$ and we get with $g=\chi(\sqrt L) f$
\[ \int_{1}^{1+\frac{1}{4N}} \big| e^{iR_{\ell,\nu} t \sqrt{L}} \, \chi(\sqrt{L}) \, f \big|^2 \, dt =  \int_{1}^{1+\frac{1}{4N} } \big| e^{iR_{\ell,\nu} t \sqrt{L}} \, \chi_2( t\sqrt L)\, \beta(r_\circ t\sqrt{ L}) \, g \big|^2 \, dt.
\]
Since $\chi_2(\sqrt L)\delta_0$ is a Schwartz function, we have by parabolic scaling with $D_t$ and Hulanicki's theorem (see also \cite{MaMe90})  
\[ \Big\| \Big(\int_{1}^{1+\frac{1}{4N}} \big| \chi_2( t\sqrt L)  F_t \big|^2 dt\Big)^{\frac 1 2}\Big\|_p \lc 
\Big\| \Big(\int_{1}^{1+\frac{1}{4N}} \big| F_t \big|^2 dt\Big)^{\frac 1 2}\Big\|_p,\quad 
1\le p\le \infty.\]
Hence, using that $R_{\ell,\nu} \sim 2^\ell$,
\begin{align*} 
&\|u_{\ell,+}(\sqrt L) f\|_p \lc 2^{\ell (\frac 12-\alpha)} \| m \|_{L^2_ \alpha}  \sum_{\nu=1}^{4N} \Big\| \Big(\int_1^{1+\frac{1}{4N}} |e^{iR_{\ell,\nu} t \sqrt { L}}  \beta (r_\circ t\sqrt {L} )\,g|^2 dt\Big)^{\frac 12} \Big\|_p 
\\
&\lc 2^{\ell (\frac 12-\alpha)} \| m \|_{L^2_ \alpha}  \sum_{\nu=1}^{4N}  
\Big\| \Big(\int_1^{1+\frac{1}{4N} }\big |D_{R_{\ell,\nu}^{-1} }\big[ e^{i t \sqrt {L}}  \beta (\tfrac{r_\circ}{R_{\ell,\nu}} t \sqrt  { L} ) \, D_{R_{\ell,\nu}} g\big]\big |^2  dt \Big)^{\frac 12} \Big\|_p 
\\
&\lc \sum_{\nu=1}^{4N}  C_02^{\ell (\frac 12-\alpha)} \| m \|_{L^2_ \alpha}  \Big( \frac{R_{\ell,\nu}}{r_\circ}\Big) ^{\alpha_0-\frac 12}   \|g\|_p \lc C_0 2^{\ell(\alpha_0-\alpha)} \| m \|_{L^2_ \alpha}  \|f\|_p. 
\end{align*}
Here we have used scaling by the dilations $D_{R_{\ell,\nu}}$, the assumption \eqref{eq:sqfct-ass} for $\la=R_{\ell,\nu} /r_\circ\sim 2^\ell$, and finally $\|g\|_p=\|\chi(\sqrt L) f\|_p\lc \|f\|_p$.

Since for real valued functions $f$ we can replace $e^{i t\sqrt L}f$ with $e^{-it\sqrt L}f$ the  above also yields $\|u_{\ell,-}(\sqrt L) \|_{p\to p} \lc C_0 2^{\ell (\alpha_0-\alpha)} \| m \|_{L^2_ \alpha}$ and the proposition  for $m$ supported in $(1-\frac 1{8N},  1+\frac{1}{8N}) $ follows after summing  in $\ell$. This support assumption can be removed by decomposing $m$ into a finite sum $\sum_{j} m_j$ where each multiplier is supported in an  interval $I_j$ of $\bbR\setminus\{0\}$ which is of length at most $\frac{1}{8N}$. Rescaling and observing that $\|m(R\cdot)\|_{L^2_\alpha}\lc (1+R)^\alpha\|m\|_{L^2_\alpha}$ for $\alpha\ge 0$, $R\ge 0$,  yields the general case of the proposition.  Finally, by using the proved result for $m_0(\sqrt L)$ with $m_0(r)= m(r^2)$, we get the final assertion for $m(L)$. 
\end{proof}

Corollary~\ref{BR-cor} and Theorem~\ref{thm:Lp-mult} can now  be derived  from Proposition~\ref{prop:compact} by known arguments, as outlined in the remainder of this section. Thus in this paper we will entirely focus on the proof of Theorem~\ref{thm:main-square}.

\subsubsection*{Riesz means}\label{onBR} 

To handle the Bochner-Riesz multiplier $(1-R^{-2}L)_+^\nu$, one decomposes $(1-L)_+^\nu= m_0(L)+ m_1(L),$ where $m_0$ is $C^\infty$ with compact support,  and $m_1$ is supported in $[\frac 12, 1]$ 
and 
belongs  to $L^2_\alpha$  when $\alpha< \nu+\frac 12$. Then $m_0(L)$ is $L^p$-bounded for $1\le p\le \infty$ by \cite{Hu84} and 
$m_1(L)$ is bounded on $L^p$ in the range of $p$ and $\nu$ prescribed in Corollary~\ref{BR-cor}, and the result for $R>0$ follows by scaling. 

The Riesz multiplier $(1-R^{-1} \sqrt L)_+^\nu$ can be similarly decomposed.  
Indeed  $\rho \mapsto m_1(\sqrt \rho)$    
belongs  to $L^2_\alpha$ when $\alpha< \nu+\frac 12$.
However, $m_0(\sqrt{L})=h(L),$ where the mapping $\rho\mapsto h(\rho):=m_0 (\sqrt \rho)$ is no longer smooth at the origin. 
To address this complication,  choose a smooth cut-off function $\chi_0$ which is identically 1 near the origin and such that $m_0=\chi_0 m_0.$ By a Taylor expansion around $0$, we may then decompose 
$$
m_0(\tau)=m_0(0)\chi_0(\tau)+\tau g(\tau),
$$
where  $g\in \Coi .$   Then $h(\rho)=m_0 (\sqrt \rho)=m_0(0)\widetilde{\chi_0}(\rho)+\psi(\rho),$ where $\widetilde{\chi_0}\in \Coi $ and and   $\psi(\rho)=\sqrt \rho g(\sqrt \rho)$. We note that if 
$\eta$ is a bump function in $(1/2,2)$, then $ \psi(2^{-\ell}\cdot)\eta$ and its derivatives  are $O(2^{-\ell/2})$. We again apply  Hulanicki's theorem to see that $m_0(0)\widetilde{\chi_0}(L)$ is $L^p$-bounded for $1\le p\le \infty$.
Moreover, $\eta(L)\psi(2^{-\ell} L)$ has $L^p\to L^p$ operator norm $O(2^{-\ell/2})$, and by dilation invariance the same holds true for  $\eta(2^{\ell}L )\psi(L)$.   Choosing $\eta $ appropriately we may sum in $\ell$ to conclude that  $\psi(L)$ is bounded on $L^p(G)$ for $1\le p\le \infty$. 

\subsubsection*{Conditions of Hörmander--Mikhlin type}
For multipliers which are not necessarily compactly supported (satisfying the assumptions of  Theorem~\ref{thm:Lp-mult})  one can use ideas  in \cite{carbery-revista} or \cite{Seeger-Crelle1986, See88} to combine a Calderón--Zygmund decomposition of $|f|^p$ with a crucial  application of Littlewood--Paley theory, together with singular integral estimates in work by Hulanicki and Stein (\cite{FoSt82, Hu84}) to deduce Theorem~\ref{thm:Lp-mult} from Proposition~\ref{prop:compact}. We shall not give the details. Alternatively one can rely on  the implementation of this argument in 
\cite[Thm. I.19] {ChenOuhabazSikoraYan} (putting $q=\infty$ in that result).

\section{First localizations} \label{sec:localization} 

We first observe that estimate \eqref{eq:main-square} in Theorem~\ref{thm:main-square} can easily be reduced to the case where $\beta$ is supported in $[\frac{1}{2},2]$. To this end, assume that $\beta$ is supported in $[-R, R].$ Choose next $\chi_0, \chi_1\in \Coi (\R)$ such that $\chi_0$ is supported in $[-1,1]$ and $\chi_1$ in $[\frac{1}{2}, 2],$ and so that 
$$\chi_0(\tau)+\sum_{j=0}^\infty \chi_1(2^{-j} \tau)=1$$
 for every $\tau\ge 0,$ and decompose
$$
\beta(\tau/\la)=\beta_0(\tau)+\sum_{1\le 2^j\le 2R\la} \beta_j(\tau/2^j),
$$
where $\beta_0(\tau)=\beta(\tau/\la)\chi_0(\tau)$ and $\beta_j(\tau)=\beta(2^{j}\tau/\la) \chi_1(\tau)$ for $1\le 2^j\le R\la.$ 

Then 
$$\beta(t\sqrt{L}/\la)=\psi (t^2L)+\sum_{1\le 2^j\le 2R\la} \beta_j(t\sqrt{L}/2^j),
$$
where $\psi(\rho)=\beta_0(\sqrt{\rho}).$

Arguing in a similar way as in the arguments for the Riesz means 
by means of (1.3) in \cite{MuellerStein-mult}, one can  show that for $1\le t\le 2,$
\[
\|e^{it \sqrt{L}} \beta_0(t\sqrt{L})  \|_{L^1\to L^1} \le C_\beta,
\]
with a bound  $C_\beta$ controlled by the $C^M$-norm of $\beta$ for some sufficiently large $M.$ 
Moreover, assuming that estimate \eqref{eq:main-square} holds for any  $\beta$ supported in $[\frac{1}{2},2],$ with
$C_{p,\alpha, \beta}$ depending only on $p,\alpha, R$ and the $C^M$-norm of $\beta$ for some sufficiently large $M,$ 
we find that 
$$\Big\| \Big( \int_1^2 \big| e^{it \sqrt{L}} \beta_j(\tfrac{t\sqrt{L}}{2^j}) f \big|^2 \, dt \Big)^{\frac 12} \Big\|_{L^p(\bbH_n)} \lesssim C_{p,\alpha,\beta} \, 2^{j\alpha} \, \|f\|_{L^p(\bbH_n)}.
$$
Summing over all $j$ such that $1\le 2^j\le 2R\la,$ we obtain  \eqref{eq:main-square}. 

\medskip

We next show that by exploiting the finite propagation speed of solutions to the wave equation associated to the sub-Laplacian $L,$  it suffices to prove inequality \eqref{eq:main-square}  for $L^p$ functions $f$ supported in a small neighborhood of $0$. 

\medskip

The following discussions will apply to arbitrary Carnot groups equipped with a left-invariant metric satisfying the local volume growth condition specified below. Suppose that $G$ is a Carnot group, and let $\mathfrak g$ be  its Lie algebra. As usual, we identify $G$ with its Lie algebra $\mathfrak{g}$ via exponential coordinates, so that the group product is given by the Baker--Campbell--Hausdorff formula. Moreover, let
\[
d:G\times G\to \R_{\ge 0}
\]
be a left-invariant metric on $G,$ i.e., $d(wz,wz')=d(z,z')$  for all $z,z',w\in G,$ and denote by $B(z,r)$ the corresponding  
ball of radius $r$ centered at $z,$ i.e.,
\[
B(z,r):=\{w\in G: d(w,z)<r\}.
\]
By $\overline B(z,r)$ we denote the corresponding closed ball, and by $$\widetilde{B(z,r)}:=B(z,2r)$$ the doubling of $B(z,r)$. In addition to the Carnot group structure, we assume that there exists a ``local homogeneous dimension'' $D>0$ such that
\begin{equation}\label{locdim}
|B(z,r)|\sim r^D \qquad \text{for all } \  r\in (0,r_0],
\end{equation}
where $r_0\gg 1$ is a fixed constant (note also that $|B(z,r)|=|B(0,r)|$).

\begin{examples}

(a) Assume that  $X_1,\dots, X_d$ is a basis of the Lie algebra $\g$ so that its first, horizontal layer $\g_1$ is spanned by $X_1,\dots,X_{d_1},$ and let $L=-\sum_{k=1}^{d_1}X_k^2$ be the associated sub-Laplacian (here, we identify as usually any vector $X\in \g$ with  the corresponding Lie derivative). For further background, we refer, e.g., to \cite{MaMueNi23}.
We then denote  by $d_{\mathrm{CC}}$  the (left-invariant) Carnot--Carathéodory or sub-Riemannian distance on $G$ associated 
to $L.$ We recall that for this metric $d=d_{\mathrm{CC}},$ the volume of $B(z,r)$ is given by  
\begin{equation}\label{ballvol}
|B(z,r)|=|B(0,r)|= |B(0,1)|\, r^Q,
\end{equation}
where  $Q\ge d$ is the homogeneous dimension of $G.$ 
\smallskip

(b) Let $d_G$ denote the left-invariant Riemannian metric defined by assuming that the left-invariant vector fields $X_1,\dots,X_d$ in Example a) form an orthonormal basis of the tangent space at any point of $G.$ Then $d_G$ is just  the sub-Riemannian metric associated to the full Laplacian 
$-\sum_{k=1}^{d}X_k^2,$ so that $d_G\le d_{\mathrm{CC}}.$ It is easily seen that the local homogeneous dimension is here given by $d.$

\end{examples}

We finally assume that $X$ is a complex Banach space, endowed with the norm $\|\cdot\|_X.$ 

\medskip

We begin with a general localization result.

\begin{proposition}\label{localize-a}
(i) Fix $r_0\gg1,$ and let $K\in L^\infty(G,X)$ be supported in a ball $\overline B(u_0,r)$, where  $u_0$ lies in the center of $G$ and $0<r\le r_0,$ and denote by $T_K$ the convolution operator $T_Kf=f*K,\ f\in \scriptS(G).$ Moreover, let $0<\ve <1$  and $1\le p\le\infty,$ and assume that 
\begin{equation}\label{locest}
\|T_Kf\|_{L^p(G,X)}\le A\, \|f\|_{L^p(G)} 
\end{equation}
for all $f\in L^p(G)$ which are supported in $\overline B(0,\ve r).$ 

Then there is a constant $C_{\ve}$ which, for a given group $G$ and metric $d$, depends only on  $\ve,$ such that 
\begin{equation}\label{globalest}
\|T_Kf\|_{L^p(G,X)}\le C_{\ve}^{1/p'} A\, \|f\|_{L^p(G)} 
\end{equation}
for all $f\in L^p(G).$ 
\smallskip

(ii) More generally, if $\alpha\in {\rm Aut}(G)$ is any automorphism of the Lie group $G,$ then the following analogous statement holds true:
\smallskip

If $\supp K\subset \alpha(\overline B(u_0,r)),$ and if \eqref{locest} holds for all $f\in L^p(G)$ which are supported in 
$\alpha(\overline B(0,\ve r)), $  then \eqref{globalest} holds  for all  $f\in L^p(G).$ 

\end{proposition}
\begin{proof} 
By Zorn's lemma, we can choose a maximal family  of pairwise disjoint  $\eps r/2$-balls $B_i=B(z_i,\eps r/2), i\in I, $ in $G.$ Then $I$ is countable. Indeed, given any  $R>1,$ if $I_R$ denotes the set of indices $i\in I$ such that $z_i\in B(0,R),$ then
\[
B_i\subset \widetilde{B(0,R)},\quad\text{hence}\quad  \sum_{i\in I_R} |B_i| \le |\widetilde{B(0,R)}|.
\]
Thus, by \eqref{locdim}, $(\sharp I_R) \, |B(0,\eps r/2)|\le |B(0,R)|,$ which shows that $I_R$ is finite.

We may  thus assume that $I=\bbN.$ Moreover, the maximality of the  family $(B_i)_{i\in \bbN}$ implies that the family of doublings $(\widetilde{B_i})_{i\in \bbN}$ covers $G.$ We can therefore decompose $G$ into the pairwise disjoint sets $D_i\subset B_i$ defined recursively by $D_0:=\widetilde{B_0}$ and $D_{i+1}:=\widetilde{B_{i+1}}\setminus 
( \widetilde{B_{0}} \cup\cdots \cup\widetilde{B_{i}}).$

Given $f\in L^p(G),$ we may accordingly decompose $f=\sum_{i\in\bbN} f_i,$ where $f_i:=f\chr_{D_i}$ satisfies $f|_{\widetilde {B_i}}=0$.
 
Next, note that for any fixed $z\in G,$ $q_z(f):=\|T_K f(z)\|_X$ defines a continuous semi-norm $q_z$  on $L^p(G),$ and thus 
$$
q_z(f)\le \sum\limits_{i\in\bbN} q_z(f_i).
$$
Moreover, 
$$T_K f_i(z)=\int_{\overline B(z_i,\eps r)} f_i(w)K(w^{-1}z) \,dw\ne 0$$ 
only if there is some $w\in \overline B(z_i,\eps r)$ such that $w^{-1}z\in \overline{B}(u_0,r).$ Then $z=wu_0y$ for some $y\in \overline{B}(0,r).$  Noting also that $u_0^{-1}z=wy,$ since $u_0$ is central, this implies  that
\begin{align*}
d(z_i,u_0^{-1}z) &\le d(z_i,w)+d(w,u_0^{-1}z)= d(z_i,w)+d(w,wy)\\
 &= d(z_i,w)+d(0,y)\le  \eps r+r< 2r.
\end{align*}

This implies that $\widetilde {B_i}\subset B(u_0^{-1}z,3r).$ Thus, if $I_z$ denotes the set of all $i\in \bbN$ such that $q_z(f_i)\ne 0,$ then 
 $\sum_{i\in I_z} |B_i| \le |B(u_0^{-1}z,3r)|=|B(0,3r)|,$ and arguing in a similar way as before and making use of \eqref{locdim}, we see that 
 $$
 \sharp I_z\le |B(0,3r)|/|B(0,\eps r/2)|\le C r^D/(\ve r)^D=C \ve^{-D}=:C_\eps.
 $$
 If $p<\infty,$ then by Hölder's inequality, we thus get
 $$
q_z(f)\le \sum\limits_{i\in I_z} q_z(f_i)\le  C_\eps^{1/p'} \big(\sum\limits_{i\in \bbN} q_z(f_i)^p\big)^{1/p},
$$
hence 
$$
\|T_Kf\|_{L^p(G,X)}\le C_\eps^{1/p'} \big(\sum\limits_{i\in \bbN} \|T_K f_i\|_{L^p(G,X)}^p\big)^{1/p}.
$$
But, $f_i$ is supported in $\overline B(z_i, \ve r),$ and since $T_K$ is left-invariant,  by means of the left-translation by $z_i^{-1},$ our assumption \eqref{locest} implies that $\|T_Kf_i\|_{L^p(G,X)}\le A\, \|f_i\|_{L^p(G)} ,$ and thus 
$$
\|T_Kf\|_{L^p(G,X)}\le C_\eps^{1/p'} A\,\big(\sum\limits_{i\in \bbN} \|f_i\|_{L^p(G)}^p\big)^{1/p}=C_\eps^{1/p'} A\,\|f\|_{L^p(G)},
$$
since the $f_i$'s have disjoint supports. This proves (i) for $p<\infty,$ and the proof for $p=\infty$  requires only the usual  modifications.

\medskip

To prove (ii), let $\alpha\in {\rm Aut}(G),$ and denote by $\alpha_*=D\alpha|_0\in {\rm Aut}(\mathfrak g)$ the associated automorphism of the Lie algebra of $G$ (see, for example, \cite{HiNe12}).

Then we recall that the  Haar measure $dw,$ which agrees with the Lebesgue measure on $\mathfrak g,$ transforms under $\alpha$ as 
$d(\alpha(w))=|\det \alpha_*| \,dw.$ Thus, if we put $K_\alpha(z)=|\det \alpha_*| \, K(\alpha(z)),$
then straight-forward computations show that 
\begin{equation}\label{auto}
T_{K_\alpha} (f\circ \alpha)=(T_K f)\circ \alpha.
\end{equation}

Moreover, if $\supp K\subset \alpha(\overline B(u_0,r)),$ then $\supp K_\alpha \subset \overline B(u_0,r).$ Thus, if 
 \eqref{locest} holds for all $f\in L^p(G)$ which are supported in $\alpha(\overline B(0,\ve r)), $  then \eqref{auto} implies that 
\eqref{locest} holds for $T_{K\alpha},$  for all $f\in L^p(G)$ which are supported in $\overline B(0,\ve r), $ and hence by part (i) 
the global estimate \eqref{globalest} will hold for $T_{K\alpha}.$  Again, by \eqref{auto}, then the estimate \eqref{globalest} will also hold for $T_{K}.$
\end{proof} 

Then we can exploit \cite[Proposition~3.1]{MartiniMueller-Metivier} to prove the following lemma, in which we choose as metric for simplicity the  Carnot--Carathéodory distance  $d_{\mathrm{CC}}$ (the metric $d_G$ would work as well).

\begin{lemma} \label{localize1}
Let $0<\eps<1$, $p\ge 2$  and $\alpha> 0,$ and suppose that   $\beta\in \Coi $ is supported in $[\frac{1}{2},2].$  Assume   that  for all $\la\ge 1$ the following variant 
\begin{equation}\label{eq:main-squareloc}
\Big\| \Big( \int_1 ^2 \big|  \chr_{\overline{B}(0,5)} e^{it \sqrt{L}} \beta(\tfrac{t\sqrt{L}}{\la}) f \big|^2 \, dt \Big)^{\frac 12} \Big\|_p \le C_{p,\alpha,\beta}  \la^\alpha \|f\|_p
\end{equation}
of  \eqref{eq:main-square}  holds true  for all  functions $f\in L^p(G)$ supported in the closed  ball $\overline B(0,\eps),$  where the constant $C_{p,\alpha,\beta}$ depends in $\beta$ only on the $C^M$-norm of $\beta$ for some sufficiently large $M\in\bbN.$
Then, for all $f\in L^p(G)$ and $\la\ge 1,$ 
\begin{equation}\label{eq:dualsqeps}
\Big\| \Big(\int_1^2 | e^{it\sqrt L} \beta (\tfrac{t\sqrt L} \la)f|^2   \,dt \Big)^{\frac{1}{2}} \Big\|_p \le C_{\eps, p,\alpha,\beta} \la^\alpha \|f\|_p,
\end{equation}
 where also the  constant $C_{\eps, p,\alpha,\beta}$ depends in $\beta$ only on the $C^M$-norm of $\beta$ for some sufficiently large $M\in\bbN.$
\end{lemma}

\begin{proof} 
Denote by $T_t$ the operator $e^{it\sqrt L}\beta (\tfrac{t\sqrt L} \la) .$ By Hulanicki's theorem, its convolution kernel 
$k_t=T_t \delta_0$ is a Schwartz function, so that  $T_tf$ is smooth for any $f\in L^p.$ Moreover, from Proposition~3.1 in  \cite{MartiniMueller-Metivier}, whose proof is based on the finite propagation speed 
for solutions of the wave equation (see \cite{Melrose86}), it follows that 
\[
k_t = \chr_{\overline{B}(0,4)} k_t+ R_{t,\lambda} \qquad \text{for}\ 1\le t\le 2,
\]
where
\begin{equation}\label{errorR}
\|R_{t,\lambda}\|_1 \lesssim_N \lambda^{-N}\qquad\text{for every } N\in\bbN.
\end{equation}
We decompose $T_t=T^1_t+T^2_t,$ where $T^1_tf=f\ast (\chr_{\overline{B}(0,4)} k_t)$ and $T^2_tf=f\ast R_{t,\lambda}.$
By means of Minkowski's inequality and \eqref{errorR}, we may bound 
\begin{equation}\label{erroresti}
\Big\| \Big(\int_1^2 |T^2_t f|^2   \,dt \Big)^{\frac{1}{2}} \Big\|_p \lesssim_N \la^{-N} \|f\|_p
\end{equation}
for every $N\in \bbN.$ 

As for $T^1_t,$ let us put $X=L^2([1,2])$  and $K(z)(t)=(\chr_{\overline{B}(0,4)} k_t)(z)$ for $z\in G$ and $t\in[1,2],$ so that $K$ is a convolution kernel taking values in $X.$ Then, with the notation of Proposition~\ref{localize-a},
$T^1_t f(z)=(f*K)(z)(t),$ so that
\[
\Big(\int_1^2 |T^1_t f(z)|^2   \,dt \Big)^{\frac{1}{2}}=\|T_K f(z)\|_X.
\]
Moreover, since for $f$ supported in $B(0,\ve),$ the function $f*K$ is supported in $\overline B(0,5)$, and since $T^1_t =T_t-T^2_t$, we see that the operator $T_K$ satisfies the assumptions  in Proposition~\ref{localize-a}, with $r=4$ and $u_0=0.$ By the proposition, this implies that,
for all $f\in L^p(G),$
\[
\Big\| \Big(\int_1^2 |T^1_t f|^2   \,dt \Big)^{\frac{1}{2}} \Big\|_p \le  C_{\ve, p,\alpha,\beta}  \la^\alpha\|f\|_p.
\]
In combination with \eqref{erroresti} this implies \eqref{eq:dualsqeps}.
\end{proof}

\section{Preliminaries on the wave operator on $\mathbb H_n$  and a spectral decomposition} \label{prelimwave}

Recall that  the Lie algebra $\g$ of  $\bbH_n$ admits the stratification  $\g=\g_1\oplus\g_2,$ where the first, horizontal layer $\g_1$ is spanned by the vector fields $X_1,\dots, X_{2n},$  and the second, central layer $\g_2$ by $U.$
Then $[\g_1,\g_1]=\g_2,$  where  $\g_1$ has dimension  $d_1=2n$ and $\g_2$ has dimension  $d_2=1.$  If one defines dilations $D_r, r>0,$ on $\g$ by setting $D_r|_{\g_j}=r^j\, {\mathrm{id}}_{\g_j}, j=1,2,$ and correspondingly on  $\mathbb H_n$ by setting $D_r(x,u):= (rx,r^2u), r>0,$ then these  dilations are automorphisms of the Heisenberg group.

The sub-Laplacian $L$ on  $\mathbb H_n$ is then homogeneous of degree $2$ under these dilations,    i.e., $L(f\circ D_r)=r^2 (L f)\circ D_r$ for any $f\in \mathcal S(\bbH_n).$ This implies that 
\begin{equation}\label{Lhomog}
m( L)(f\circ D_r)=(m(r^2L) f)\circ D_r.
\end{equation}
If $k_{m(L)}=m(L)\delta_0$ denotes the convolution kernel of $m(L)$ as in  \eqref{convkernel},  so that $m(L)f=f*k_{m(L)},$ then \eqref{Lhomog} easily implies that 
\begin{equation}\label{mrrL}
k_{m(r^2L)}=r^{-Q}k_{m(L)}\circ D_{r^{-1}}, \qquad r>0,
\end{equation}
where $Q=2n+2$ denotes the homogeneous dimension of $\bbH_n.$

\subsection{The singular support and horizontal points}

Nachman \cite{Na82} computed the singular support $\Sigma_1$  of the convolution  kernel $\sK_1$ of ${e^{i\sqrt L}}$ and established (non-uniform)  asymptotics at generic points of $\Sigma_1$.  In particular he found
\begin{equation}
\Sigma_1= \{ (x, u):  |x|= \rho(\tau) \text{ and }  |u|=\tfrac 14 \nu(\tau) \text{ for some } \tau\ge 0 \}, 
\end{equation}
where  the  functions $\rho$ and $\nu$ are given by 
\[
\rho(\tau) = \Big| \frac{\sin \tau}{\tau}\Big|,\quad\quad
\nu(\tau) = \frac 1{\tau} - \frac{\sin(2\tau)}{2\tau^2},
\]
where the values at $\tau=0$ are defined by continuous extension, so that $\rho(0)=1$ and $\nu(0)=0$. 

The singular support $\Sigma_1$ is symmetric under rotations in $\bbR^{2n}\times \{0\}$, reflections in $\{0\}\times \bbR,$  and in the $(|x|,|u|)$-plane it corresponds to the  zigzag curve approaching the origin as  illustrated in Figure \ref{fig1} in the introduction. 
As pointed out in the introduction the parameters where $\Sigma_1$ has horizontal tangent planes play a special role in our analysis.
For $k\in \mathbb{N}$ and  $k\pi<\tau<(k+1)\pi$,  we get 
$$ \rho'(\tau)=(-1)^k\frac {\tau\cos\tau-\sin \tau}{\tau^2},\qquad \nu'(\tau)=-2\frac {\cos \tau}{\tau^3}\big(\tau\cos\tau-\sin \tau\big).
$$ 
Hence the slope is given by $\nu'(\tau)/\rho'(\tau) = (-1)^{k+1}2\frac {\cos \tau}{\tau}$, and it vanishes if and only if 
\begin{equation}\label{hpoint}
\tau=\tau_k:=k\pi+\tfrac \pi 2.
\end{equation}
This shows that the points with horizontal tangent plane  are given by $|x|=\rho(\tau_k)$ and $|u|=\frac14\nu(\tau_k),$ for some $k\in \N$ (compare Figure \ref{fig3}).

\medskip

The paper \cite{MuellerSeeger2015} by the first and third author provides a parametrix for the wave operator by representing $e^{i\sqrt L}$  as a superposition of oscillatory integral operators which  will be described in the next subsection. An important tool is a subordination formula which expresses wave propagators in terms of the Schrödinger   propagators $e^{it L}, t>0,$ for which one has explicit expressions, up to the Fourier transform in the central variable $u$ (cf. \cite{Gaveau77, MuellerRicci-Inv90}). 
As a consequence this allowed  the authors to write down a decomposition of $\sK_1$   and give an effective 
analysis for the pieces associated with the zig-zag curves in Figure \ref{fig1}. 

In particular, for $\alpha(p)=(d-1)|\frac1p-\frac12|$ this led to 
sharp fixed time $L^p\to L^p$ bounds for  the operators  $e^{it\sqrt L} (I+t^2L)^{-\alpha(p)/2} $, and for the variants $\beta (\la^{-1}t\sqrt L) e^{it \sqrt L} $ considered here (first, for $t=1,$  and then for arbitrary $t$ because of \eqref{Lhomog}). 
\smallskip

In the next subsection,  we also recall preliminaries from the joint spectral calculus  of the operators $L$ and  $-iU$ which are needed to introduce this decomposition.

\subsection{Spectral calculus and a decomposition of the wave operator}

For $f\in \mathcal S(\mathbb H_n)$ and $\mu\in\R$, let $f^\mu$ denote the $\mu$-section of the partial Fourier transform along the central layer $\g_2$ given by
\[
f^\mu(x) = \int_{\R} f(x,u) e^{- 2\pi i \mu u} \, du,\quad x\in \R^{d_1}=\R^{2n}.
\]
We extend this notation to tempered distributions $\mathcal S'(\R^d)$. By this partial Fourier transform, the vector fields $X_j$ are transformed into some first order differential operators $X_j^\mu$ such that $(X_jf)^\mu = X_j^\mu f^\mu$, and the sub-Laplacian $L$ is transformed into the $\mu$-twisted Laplacian $L^\mu$ such that
\[
(Lf)^\mu = L^\mu f^\mu = -\big((X_1^\mu)^2 + \dots + (X_{d_1}^\mu)^2\big) f^\mu.
\]
The partial Fourier transform and the functional calculi of $L$ and $L^\mu$ are compatible, that is, for any bounded Borel measurable $m:\R\to\C$ and any function $f\in L^2(\mathbb H_n)$,
\[
(m(L) f)^\mu = m(L^\mu) f^\mu
\]
for almost all $\mu\in \R$. Moreover, the group convolution of two suitable functions $f,g$ on $\mathbb H_n$ is transformed into
\[
(f*g)^\mu = f^\mu *_\mu g^\mu,
\]
where $*_\mu$ is the $\mu$-twisted convolution on $\R^{d_1}$ given by
\[
\phi *_\mu \psi (x) = \int_{\R^{d_1}} \phi(x') \, \psi(x-x') \, e^{i \pi \mu\langle J x,x' \rangle} \, dx',
\quad x\in \R^{d_1}=\R^{2n}.
\]
Thus, at least formally, we have
\[
e^{itL^\mu}f^\mu=(e^{itL} f )^\mu=(f*(e^{itL}\delta_0))^\mu= f^\mu *_\mu (e^{itL}\delta_0)^\mu.
\]
The partial Fourier transform $(e^{itL}\delta_0)^\mu= e^{itL^\mu}\delta_0$ of the Schrödinger kernel $e^{itL }\delta_0$ can be computed explicitly, see \cite{MuellerRicci-Inv90}: 
For $f\in \mathcal S(\R^{2n})$, there is a distribution $\gamma_t^\mu \in \mathcal S'(\R^{2n})$ such that
\begin{equation}\label{eq:gamma}
e^{itL^\mu } f = f *_\mu \gamma_t^\mu.
\end{equation}
Moreover, for all $t\in \R$ and $\mu\in \R$ such that $\sin (2\pi t \mu)\neq 0$, $\gamma_t^\mu$ is given by
\begin{equation}
\label{eq:gammatmu}
\gamma_t^\mu (x) = \Big(\frac {\mu}{2\sin (2\pi t \mu )}\Big)^{ n} \, 
e^{-i\frac {\pi}2 \mu \cot(2\pi t \mu)|x|^2},\quad x\in \R^{2n}.
\end{equation}
This can also  be found in \cite[Proposition~3.2]{MuellerSeeger2015}, together with an alternative description near $\{\mu:\sin(2\pi t\mu)=0\} $.

\smallskip

The following subordination result is an immediate consequence of \cite[Proposition~3.2]{MuellerSeeger2015}. Note that in  \cite{MuellerSeeger2015} only the case $t=1$ is considered, but the  identity \eqref{eq:subordination} is a consequence of  a subordination formula for  multipliers and thus holds when we plug in $t^2L$ in place of $L$ as well. Moreover, the first identity in \eqref{eq:rhola-est} follows from \eqref{Lhomog}.

\begin{lemma}\label{lem:approx}(\cite[Proposition~4.1]{MuellerSeeger2015})
Let $\beta,\tilde \beta\in \Coi ((0,\infty))$ with $\supp \beta\subseteq (\frac 12,2)$ and ${\tilde \beta|_{[\frac 1 4 , 4 ]} = 1}$. Then there are functions $a_{\la,\circ},\rho_\lambda\in \Coi (\R)$ supported in $[\frac{1}{16}, 4 ]$ and $[\frac 1 4 , 4 ]$, respectively, such that
\begin{equation} \label{eq:subordination}
\beta\big(\tfrac{t\sqrt L}{\lambda}\big)e^{it\sqrt L} = \tilde\beta\big(\tfrac{t^2 L}{\lambda^2}\big) \, T_{\la,t} 
 + \, \rho_\lambda\big(\tfrac{t^2 L}{\lambda^2}\big)
\end{equation}
for all $\lambda\ge 1$ and $t\in (0,\infty)$, where for all $\lambda\ge 1$,
\begin{equation}\label{eq:Tlat} T_{\la,t}= \lambda^{\frac 12}\!\int_0^\infty e^{i\frac{\la }{ 4 \varsigma}} \, a_{\lambda,\circ}(\varsigma) \, 
e^{i \varsigma t^2 L/\la} \, 
d\varsigma ,
\end{equation}
\begin{equation}\label{eq:rhola-est} 
\| \rho_\lambda\big(\tfrac{t^2 L}{\lambda^2}\big) \|_{p\to p} = \| \rho_\lambda\big(\tfrac{L}{\lambda^2}\big) \|_{p\to p}\lesssim_N \lambda^{-N},
\end{equation}
and,\footnote{There is a typo in \cite[Prop 4.1]{MuellerSeeger2015}  where in (33) a decay of  order  $\la^{-N_2}$ is erroneously suggested  (but never used).}for all $N_1, N_2\in \N$, and all $\lambda\ge 1$, 
\begin{equation}\label{eq:ala}
 \sup_{\varsigma\in [\frac{1}{16}, 4 ]}\big|\partial_{\varsigma}^{N_1} \partial_\la^{N_2} a_{\lambda,\circ} (\varsigma) \big|\le C(N_1,N_2) \, \|\beta\|_{C^{2(N_1+N_2+1)}}. 
\end{equation}
\end{lemma}

Next, we recall that the self-adjoint operators  $L$ and $-iU$  commute and thus admit a joint functional calculus. For any Borel measurable function $m:\R\times \R\to\C$, we can define the (possibly unbounded) operator $m(L,-iU)$ on $L^2(\mathbb H_n)$. The partial Fourier transform and the joint functional calculus of $L$ and $-iU$ are compatible, that is, if $m$ is a bounded Borel function, then, for all $f\in L^2(\mathbb H_n)$ and almost all $\mu\in\g_2^*$,
\[
\left(m(L,-iU)f\right)^\mu = m(L^\mu,2\pi\mu) f^\mu.
\]
Note that since also $-iU$ is homogeneous of degree 2 with respect to the automorphic dilations $D_r,$ the identity \eqref{Lhomog}
generalizes naturally to the following one for the joint spectral calculus:
\begin{equation}\label{LUhomog}
m( L,-iU)(f\circ D_t)=(m(t^2L, t^2(-iU)) f)\circ D_t, \qquad t>0.
\end{equation}

We use the joint functional calculus to introduce a decomposition of the operator
$T_{\la,t}$ in \eqref{eq:Tlat} in the central Fourier variable $\mu,$ which is motivated by the periodicity properties of the trigonometric functions appearing in \eqref{eq:gammatmu}. 
 To this end, choose an even cut-off  $\eta_0\in C_c^{\infty}(\R)$ supported in ${(-\frac{5 \pi}{8}, \frac{5 \pi}{8})}$ with $\eta_0(\tau) = 1$ for ${\tau \in(-\frac{3 \pi}{8}, \frac{3 \pi}{8})}$ such that
\[
\sum_{k=-\infty}^\infty \eta_0(\tau-k\pi) = 1 \quad \text{for all }\tau\in \R ,
\]
and define \begin{equation} \label{Tklat-definition}
T_{\lambda,t}^{k}  = \lambda^{\frac 12} \int_0^\infty e^{i \frac{\lambda}{ 4 \varsigma}} \, a_{\lambda,\circ}(\varsigma) \, \eta_0\Big(\frac{\varsigma t^2}{\lambda}(-iU)-k \pi\Big) \, e^{i \varsigma t^2 L / \lambda} \, d\varsigma
\end{equation}
(note here that for $t=1$ this corresponds to the according definition in (55) in \cite{MuellerSeeger2015}, and for arbitrary $t>0$ this is motivated by \eqref{LUhomog}).

The operator $ T_{\la,t}$ then decomposes (compare (54) and (57) in  \cite{MuellerSeeger2015}) at least formally  as 
$$
 T_{\la,t} =\sum_{k\in\Z}  T^k_{\la,t}.
$$
Each operator $ T^{k} _{\la,t}$ is a convolution operator $T^{k}_{\la,t}f=f*\sK_{\la,t}^k,$ and 
by  \eqref{eq:gammatmu} (compare also \cite[(57)]{MuellerSeeger2015}\footnote{Note that in  contrast to \cite[(57)]{MuellerSeeger2015}, where each factor except for the last depends only on $|\mu|$, here we must allow for  negative $k\in\Z,$ and correspondingly  replace each instance of $|\mu|$ by $\mu$.} for the case $t=1$), its convolution kernel is given by 
\begin{multline}
\sK_{\la,t}^k(x,u)=\lambda^{\frac 12} \int_0^\infty \int_\R \,  e^{i \frac{\lambda}{ 4 \varsigma}}\,  a_{\lambda,\circ}(\varsigma) \, \eta_0\big(2\pi\tfrac{\varsigma t^2 }{\lambda}\mu-k \pi\big)\, \Big(\frac {\mu}{2\sin (2\pi \frac{\varsigma t^2}{\lambda}\mu)}\Big)^{n}\\
\times e^{-i\frac {\pi}2 \mu |x|^2\cot(2\pi \frac{\varsigma t^2}{\lambda} \mu) } 
  \, e^{2\pi i\mu u} \,d\mu\, d\varsigma. \label{Kklat}
\end{multline}
The concrete description of the joint spectrum of $L$ and $-iU$ also shows that 
for fixed $\la$ we  have
\begin{equation} \widetilde \beta\big(\tfrac{t^2L}{\la^2}) \, T^k_{\la,t} = 0 \text{ for }|k|\ge 8\la 
\end{equation} (for $t=1$ see \cite[(56)]{MuellerSeeger2015} and the discussion leading to it, and for general $t>0$ this follows then by \eqref{LUhomog}). Thus we only need to consider  $T^k_{\la,t}$ for $|k|<8\la.$
\smallskip

The analysis in \cite{MuellerSeeger2015} (using the previous formula for $\sK^k_{\la,t}$) links the singular support of $\sK^0_{\la,t}$ with the outer part in Figure \ref{fig1}, and for $k>0$  the singular support of  $\sK^{\pm k}_{\la,t}$ 
with the $k$th zigzag curve in Figure \ref{fig1}. The kernels $\sK^k_{\la,t}$  become more singular and less oscillatory near the vertical axis which necessitated another decomposition in \cite{MuellerSeeger2015}:
With $\eta_0$ as above and for $\ell\ge 1$,  let $\eta_\ell\in C_c^{\infty}(\R)$ be given  by 
\begin{equation}\label{eq:eta-ell-def}
 \eta_\ell(\tau) = \eta_{0}(2^{\ell-1} \tau)-\eta_{0}(2^{\ell} \tau),\qquad \tau\in \bbR.
\end{equation}
Then $\eta_{0}(\tau) = \sum_{\ell = 1}^{\infty} \eta_\ell(\tau)$ for all $\tau \in\R\setminus\{ 0\},$  and for $k\neq 0$  we  define for any $\ell\ge 1$
\begin{equation}\label{eq:Tlakl-def} 
T_{\lambda,t}^{k,\ell} = \lambda^{\frac 12} \int_0^\infty e^{i \frac{\lambda}{ 4 \varsigma}} \, a_{\lambda,\circ}(\varsigma) \, \eta_\ell\Big(\frac{\varsigma t^2}{\lambda}(-iU)-k \pi\Big) \, e^{i \varsigma t^2 L / \lambda} \, d\varsigma.
\end{equation}
This defines 
a convolution operator $f\mapsto T_{\lambda,t}^{k,\ell}f=f*\sK_{\la,t}^{k,\ell},$ whose convolution kernel is given  by 
\begin{multline}
\sK_{\la,t}^{k,\ell} (x,u)=\lambda^{\frac 12}  \int_0^\infty \int_\R \, e^{i \frac{\lambda}{ 4 \varsigma}}\,  a_{\lambda,\circ}(\varsigma) \, \eta_\ell\big(2\pi\tfrac{\varsigma t^2 }{\lambda}\mu-k \pi\big)\, \Big(\frac {\mu}{2\sin (2\pi \frac{\varsigma t^2}{\lambda}\mu)}\Big)^{n}\\
\times e^{-i\frac {\pi}2 \mu |x|^2\cot(2\pi \frac{\varsigma t^2}{\lambda} \mu) } 
  \, e^{2\pi i\mu u} \,d\mu\, d\varsigma,    \label{Kkllat}
\end{multline}
and we get the  decomposition 
$
T^k_{\la,t}f=\sum_{\ell=1}^\infty T_{\la,t}^{k,\ell}f,
$
with convergence at least in the sense of distributions (indeed by \cite{MuellerSeeger2015} also in $L^1$). 
\smallskip

To summarize,  we have decomposed 
\begin{equation}\label{opdecomp}
\beta\big(\tfrac{t\sqrt L}{\la}\big) e^{it\sqrt L }f=\widetilde \beta\big(\tfrac{t^2L}{\la^2}) \,  T^0_{\la,t} f+
\sum_{0<|k|<8\la} \sum_{\ell\ge 1} \widetilde \beta\big(\tfrac{t^2L}{\la^2}) \,  T^{k,\ell}_{\la,t}f+ \rho_\la\big(\tfrac{t^2L}{\la^2}\big)f,
\end{equation}
with $T^{0}_{\la,t}f=f*\sK^{0}_{\la,t}$ and $T^{k,\ell}_{\la,t}f=f*\sK^{k,\ell}_{\la,t}.$

Here, by \eqref{eq:rhola-est}, the error term $\rho_\lambda\big(\tfrac{t^2L}{\lambda^2}\big)$ can be ignored in the sequel. 
\smallskip

Finally, the  following proposition provides  concrete representations of the convolution kernels $\sK^0_{\la,t}$ and 
$\sK^{k,\ell}_{\la,t}.$  Let
\begin{equation} \label{eq:phasedef}
 \Phi(x,u,t,\sigma,\tau) =\sigma( t^2-|x|^2 g(\tau) +4u\tau), \text{ with } g(\tau)=\tau\cot \tau, 
  \end{equation} and 
 \begin{equation}\label{eq:ampldef}  a_\la(\sigma)=  (2\pi^{n+1})^{-1} \, a_{\la,\circ}(\tfrac 1{4\sigma}) \, \sigma^{n-1}. \end{equation}
 Note that $a_\la$ is a $\Coi $ function  supported in $(\frac 1{16},4)$. 

\begin{proposition}\label{lem:decomp} 
  Let $\Phi$ and  $a_\la$ be as in \eqref{eq:phasedef} and  \eqref{eq:ampldef}. Then for  $k=0$ 
\begin{equation}\label{eq:Klat-formula}
\sK_{\lambda,t}^{0}(x, u) = \lambda^{ n+\frac{3}{2}} t^{-2n} \int_\R \, \int_0^\infty   e^{i\lambda\Phi(x,u,t,\sigma,\tau)} \, a_\lambda(t^2\sigma) \, \eta_0( \tau) \, \Big(\frac{\tau}{\sin\tau}\Big)^{n}  d\sigma \, d\tau,
\end{equation}
and for $k\neq 0$, $\ell\ge 1$
\begin{equation}\label{eq:Klatkl-formula} 
\sK_{\lambda,t}^{k,\ell}(x, u) = \lambda^{ n+\frac{3}{2}} t^{-2n}  \int_\R \, \int_0^\infty  e^{i\lambda\Phi(x,u,t,\sigma,\tau)} \, a_\lambda(t^2\sigma) \, \eta_\ell( \tau -k \pi) \, \Big(\frac{\tau}{\sin\tau}\Big)^{n}  d\sigma \, d\tau.
\end{equation}
\end{proposition}

\begin{proof} 
To prove \eqref{eq:Klat-formula}, we  change variables in \eqref{Kklat} from $\varsigma$ to $\sigma$ via $1/(4\varsigma)=t^2\sigma$, so that $\varsigma = 1/(4t^2\sigma)$, and then from $\mu$ to $\tau$ by setting $2\pi\varsigma t^2\mu/\lambda=\tau,$ so that $\mu=2 \la  \sigma\tau/\pi$. The Jacobian of the total change of variables is $-\la/(2\pi\sigma t^2),$ and identity \eqref{eq:Klat-formula} follows by straight-forward computations from \eqref{Kklat}.  \eqref{eq:Klatkl-formula} is derived  in the same way from \eqref{Kkllat}.
\end{proof}

\begin{remark}\label{rem:k-negative}
In the following sections, we will typically suppress the case $k<0$ and state estimates only for $k>0$. The corresponding bounds for negative $k$ can be recovered directly from the positive case. Indeed, since $\chi_0$ was chosen to be even, each $\eta_\ell$ is even as well. Thus, replacing $\tau$ by $-\tau$ in \eqref{eq:Klatkl-formula} yields
\[
\sK_{\lambda,t}^{-k,\ell}(x, u) = \sK_{\lambda,t}^{k,\ell}(x, -u) \quad \text{for } k>0. 
\]
Moreover, the involution $\theta(x,u):=(x_1,-x_2,\dots,x_{2n-1},-x_{2n},-u)$ is an automorphism of $\mathbb H_n,$ and   \eqref{eq:phasedef} is radially symmetric in the $x$ variable.  In combination, this implies that 
\begin{align*}
T^{-k,\ell}_{\la,t} f  = \big( T^{k,\ell}_{\la,t}  (f\circ \theta)\big) \circ \theta.
\end{align*}
\end{remark}

\section{The case $k=0$: Preliminaries}\label{sec:k=0-preliminaries}

We begin by stating a localization result for $\bbH_n$ which is an almost immediate consequence of Lemma~\ref{localize1}, since the Euclidean topology and the sub-Riemannian topology  on $\bbH_n$ are the same, and since the Euclidean distance is even bounded by the sub-Riemannian distance. The Euclidean ball of radius $r$ centered at $z$ will be denoted by 
$B_r(z),$ i.e.,
\[
B_r(z)=\{w\in \bbH_n:|w-z|<r\}
\]
where $|\cdot|$ denotes the Euclidean norm.

We introduce a smooth bump function $\chi$ supported in the Euclidean ball $B_1(0)\subset \bbH_n$ such that $\chi\equiv 1$ on $B_{\frac{1}{2}}(0),$ and put $\chi_\eps(w):=\chi(w/\eps)$ for  $\eps>0.$ 

\begin{proposition} \label{prop:localization}
Let $0 < \eps < 1$, $p \geq 2$, $\alpha > 0$, and let $\beta$ be a smooth bump function supported in $[\frac12,2]$.  Suppose that for every  $(x^0,u^0) \in \bbH_n$ with $x^0 \in \bbR^{2n}$ of the form $x^0=(|x^0|,0,\ldots,0)$, with $|x^0| < 10$ and every $f \in \mathcal{S}(\bbH_n)$, we have the estimate
\begin{equation}\label{eq:main-square-L2}
\Big\| \Big( \int_1^2 \Big| \tilde \chi_{\varepsilon} \, e^{it\sqrt{L}} \, \beta\Big( \frac{t\sqrt{L}}{\lambda} \Big) \, ( \chi_{\varepsilon}f) \Big|^2 \, dt \Big)^{\frac{1}{2}} \Big\|_p
\le C(\eps, \alpha,p) \lambda^\alpha \, \|f\|_2,
\end{equation}
with $\tilde \chi_\eps(x,u):=\chi_\eps(x-x^0,u-u^0)$.  Then the square function estimate \eqref{eq:main-squareloc} holds.  
\end{proposition}

\begin{proof} Choose $\eps'>0$ so small that $B(0,\eps')\subset B_\eps(0).$ Then, for $f$ supported in $B(0,\eps'),$
$\|f\|_2\le C_{\eps' }\|f\|_p$ by Hölder's inequality, and thus \eqref{eq:main-square-L2} implies \eqref{eq:main-squareloc}, with $\eps'$ in place of $\eps.$ 

Finally, it suffices to consider only $x^0 = (|x^0|,0,\dots,0)$ because the unitary group $U(n)$, which acts transitively on $\bbC^n$, acts by automorphisms of $\bbH_n$ that leave the  center fixed. The sub-Laplacian $L$ is invariant under this action, and hence so is $\beta(\frac{t\sqrt{L}}{\lambda})$.  
\end{proof}

We will exploit this for $k=0$ from the decomposition of Proposition~\ref{lem:decomp}.

\smallskip

Now let $\eps \in (0,10^{-4})$ and $(x^0,u^0)$ with $|(x^0,u^0)| < 10$ and $x^0 = (|x^0|,0,\ldots,0)$ be fixed, and define $\chi_\eps,\tilde\chi_\eps$ as above.
The following proposition treats the case $k = 0$.

\begin{proposition}\label{L:T0lambda}
If $p\ge \frac{2(d_1+1)}{d_1-1}$  and $\alpha \ge d \, (\tfrac12-\tfrac1p)-\tfrac12$, then for $\ve>0$ sufficiently small, 
\begin{equation}\label{eq:T0lambda}
\Big\| \Big( \int_1^2 \big| \tilde \chi_{\varepsilon} \, T_{\lambda,t}^0 ( \chi_{\varepsilon} f) \big|^2 \, dt \Big)^{\frac{1}{2}} \Big\|_p
\le C(a_\lambda, \ve, \alpha, p)\lambda^\alpha \, \|f\|_2 \quad \text{for all } f\in \mathcal S(\R^d),
\end{equation}
where the constant  $C(a_\lambda, \ve,\alpha, p)$ depends in $a_\la$ only on the $C^M$-norm of $a_\lambda$ for some sufficiently large $M\in\N.$
\end{proposition}

To  prepare the proof of Proposition~\ref{L:T0lambda}, recall from Lemma~\ref{lem:decomp} that  the convolution kernel $\sK_{\lambda,s}^{0}$ of $T_{\lambda,t}^{0}$ is given by
\[
\sK_{\lambda,t}^{0}(x, u) = \lambda^{ n+\frac{3}{2}} t^{-2n} \int_\R \, \int_0^\infty e^{i\lambda\Phi(x,u,t,\sigma,\tau)} \, a_\lambda(t^2\sigma) \, \eta_0( \tau) \, \Big(\frac{\tau}{\sin\tau}\Big)^{n} d\sigma \, d\tau.
\]
Since the factor $t^{-2n}$ is irrelevant for the estimate \eqref{eq:T0lambda}, we drop it in the sequel. Moreover, it will be convenient to replace $t$ by $$s=t^2\in[1,4].$$
Moreover, since $\supp\eta_0\subset [-\tfrac 58 \pi, \tfrac 58 \pi],$ the factor $(\frac{\tau}{\sin\tau})^{n}$ is smooth and can be included into the amplitude by putting $\tilde \eta_0(\tau):=\eta_0(\tau)\, (\frac{\tau}{\sin\tau})^{n}.$ 

\smallskip

We shall therefore rather work with the convolution kernel 
\begin{equation}\label{eq:Klats-formula}
K_{\lambda,s}^{0}(x, u) = \lambda^{n+\frac32}  \int_\R \, \int_0^\infty e^{i\lambda\tilde \Phi(x,u,s,\sigma,\tau)} \, a_\lambda(s\sigma) \, \tilde \eta_0( \tau)  \, d\sigma \, d\tau,
\end{equation}
where 
\begin{equation} \label{eq:phasesdef}
\tilde \Phi(x,u,s,\sigma,\tau) =\sigma(s-|x|^2 g(\tau) +4u\tau), \quad\text{ with } g(\tau)=\tau\cot \tau,
\end{equation}
and denote the corresponding convolution operator by $\tilde T^0_{\la,s} f=f*K_{\lambda,s}^{0}.$

We then define $T^{0,\varepsilon}_\lambda :\mathcal S(\R^d) \to L^0(\R^{d + 1})$ as  the operator given by
\begin{equation}\label{eq:xut}
(T^{0,\varepsilon}_\lambda f)(x,u,s) : = \tilde\chi_\varepsilon(x,u) \big(\tilde T_{\lambda,s}^0(\chi_\varepsilon f)\big)(x,u),\quad f\in \mathcal S(\R^d).
\end{equation} 
We then need to prove estimates of the form 
\[
\Big\| \Big( \int_1^4 \big| T^{0,\varepsilon}_\lambda f \big|^2 \, ds \Big)^{\frac{1}{2}} \Big\|_p
\lesssim \lambda^\alpha \, \|f\|_2.
\]
By decomposing the $s$-interval $[1,4]$ into smaller intervals of length $\ve\in(0,1)$ and applying suitable automorphic 
scalings for each of them, we may actually  reduce to smaller $s$-intervals $[1,1+\ve].$  Accordingly, we always assume that  $s\in[1,1+\ve]$ in the definition of  $T^{0,\varepsilon}_\lambda f.$

Similarly, by means of a partition of unity  consisting of smooth cut-off functions supported in intervals of length $\ve$ we may further assume that $a_\la$ is supported in such a shorter interval. Thus, we may reduce to proving the following:

\medskip
{\it 
Suppose that  $p\ge \frac{2(d_1+1)}{d_1-1}$ and $\alpha \ge  d \, (\tfrac12-\tfrac1p)-\tfrac12$.
If the constant $\ve>0$ is  sufficiently small, and if $a_\la$ is supported in an interval of length $\ve$ contained in 
$(\frac 1{16},4),$ then we can bound 
\begin{equation}\label{eq:T0lambdas}
\Big\| \Big( \int_1^{1+\ve} \big| T^{0,\varepsilon}_\lambda f \big|^2 \, ds \Big)^{\frac{1}{2}} \Big\|_p
\le C(a_\lambda,\eps, \alpha,p)\lambda^\alpha \, \|f\|_2 \quad \text{for all } f\in \mathcal S(\R^d),
\end{equation}
where the constant  $C(a_\lambda,\eps, \alpha,p)$ depends in $a_\la$ only on the $C^M$-norm of $a_\lambda$ for some sufficiently large $M\in\N.$}
\smallskip

Recall that the point $(x^0,u^0)$ satisfies $|(x^0,u^0)|<10$. For $|x^0|\ll 1$, we obtain  rapid decay in $\lambda$ by exploiting that the phase of the associated integral kernel is non-stationary in at least one of the parameters $\sigma$ or $\tau$.

\begin{lemma}\label{lem:0-x0-small}
If $|x^0|\le \frac 1 {100}$, then 
\[
\Big\| \Big( \int_1^{1+\ve} \big| T^{0,\varepsilon}_\lambda f \big|^2 \, ds \Big)^{\frac{1}{2}} \Big\|_p
\lesssim_N \lambda^{-N} \, \|f\|_2 \quad \text{for all } f\in \mathcal S(\R^d).
\]
\end{lemma}

\begin{proof}
Recall that $\varepsilon\le \frac 1 {10^{4}}$.
Since $T^{0,\varepsilon}_\lambda f$ is supported on $\supp \tilde\chi_\varepsilon$, we have
\[
\Big\| \Big( \int_1^{1+\ve} \big| T^{0,\varepsilon}_\lambda f \big|^2 \, ds \Big)^{\frac{1}{2}} \Big\|_p
\lesssim 
\Big\| \Big( \int_1^{1+\ve} \big| T^{0,\varepsilon}_\lambda f \big|^2 \, ds \Big)^{\frac{1}{2}} \Big\|_\infty,
\]
so it suffices to show uniformly in $(x,u,s)$ that 
\begin{equation}\label{eq:0-x0-small-i}
|T^{0,\varepsilon}_\lambda f(x,u,s)|\lesssim_N \lambda^{-N}  \|f\|_2.
\end{equation}
By \eqref{eq:xut}, we have
\begin{align*}
(T^{0,\varepsilon}_\lambda f)(z,s)
 & = \tilde\chi_{\varepsilon}(z)\big((\chi_{\varepsilon} f)*K_{\lambda,s}^0\big) (z) \\
 & = \tilde\chi_{\varepsilon}(z)\int_{\bbH_n}\chi_{\varepsilon} (w) \,  K^0_{\lambda,s}\big(w^{-1}z\big) \, f(w)\, dw,
\end{align*}
where, by \eqref{eq:Klats-formula} and \eqref{eq:phasesdef}, when writing $z=(x,u)$ and $w=(y,v)$, the phase of the associated integral kernel is given by
\begin{align*}
\tilde \Phi(w^{-1}z,s,\sigma,\tau)
& =
\tilde \Phi(x-y,u-v+\tfrac 1 2 \langle Jx,y\rangle,s,\sigma,\tau)\\
&=\sigma\,\big(s-\left|x-y\right|^2 g(\tau) +4\, (u-v+\tfrac 1 2 \langle Jx,y\rangle)\,\tau\big),
\end{align*}
where $|(x-x^0,u-u^0)|<\eps$ on the support of $\tilde\chi_\eps$ and  $|(y,v)|\le \varepsilon$ for 
$(y,v)\in \supp \chi_\varepsilon.$

Assume that $|x^0|\le \frac 1 {100}$. We distinguish the cases $|(x,u)|<\frac 1 {10}$ and $|(x,u)|\ge \frac 1 {10}$.

Suppose that $|(x,u)|<\frac 1 {10}$.  Then since $|g(\tau)|=
|\tau\cot\tau| \le 1$ in the above integrals, we have
\[
\left||x-y|^2 \, g(\tau)\right| \le (\tfrac 1 {10} + \varepsilon)^2 < \tfrac 2 {100}
\]
and also
\[
|4\, (u-v+\tfrac 1 2 \langle Jx,y\rangle)|\, |\tau|
 \le (4 \,(\tfrac{1}{10} +  \varepsilon) + 2 \,\tfrac{1}{10} \, \varepsilon) \,\tfrac{5}{8}\pi < \tfrac{3}{10}\pi < \tfrac{95}{100}.
\]
Thus, since $s \in [1,1+\varepsilon]$,  we see that $|\partial_\sigma (\tilde \Phi(w^{-1}z,s,\sigma,\tau))| \ge \frac 3{100}$ and we can integrate by parts. Combining this with an application of Cauchy--Schwarz in the integral above shows \eqref{eq:0-x0-small-i} for $|(x,u)|<\frac 1 {10}$.

Next, suppose that $|(x,u)|\ge \frac 1 {10}$. We have $|x| <\varepsilon + \tfrac{1}{100}$ for $(x,u)\in \supp \tilde \chi_\varepsilon$, which then implies $|u| \ge |(x,u)| - |x| \ge \tfrac 1 {10} - |x| \ge \tfrac 1 {20}$. Hence,
\[
|4\, (u-v+\tfrac 1 2 \langle Jx,y\rangle)|
 \ge  \tfrac{1}{5} - 4\varepsilon - 2 \, (\varepsilon + \tfrac 1 {100} ) \, \varepsilon \ge \tfrac 1 {10}.
\]
Moreover, since $g'(\tau)= \cot \tau - \tau/(\sin \tau)^2=( \sin 2\tau -2\tau) /(1-\cos 2\tau) $ and $g''(\tau) = 2 \,(\tau  \cot \tau - 1)/  (\sin \tau)^2 \le 0$, we get for $|\tau|\le \frac 5 8 \pi$ that
\[
|g'(\tau)|\le g'(-\tfrac 5 8 \pi) = \frac{\sin \tfrac \pi 4 + \tfrac 5 4 \pi }{1+\cos \tfrac \pi 4} = \frac{\sqrt 2 + \frac 5 2 \pi}{2 +\sqrt 2} \le 3.
\]
Thus, we have $|x-y|^2\,|g'(\tau)|\le 3\, (2\varepsilon+\frac 1 {100})^2\le \frac{1}{20}$ for $(y,v)\in\supp \chi_\varepsilon$. Hence, we obtain $|\partial_\tau (\tilde \Phi(w^{-1}z,s,\sigma,\tau))| \ge \frac \sigma{20}$, and we can argue again via integration by parts and Cauchy--Schwarz.
\end{proof}

To prove \eqref{eq:T0lambdas}, we shall assume from now on that $|x^0|\ge \frac{1}{100}$, where we may assume that  $x^0=(|x^0|,0, \dots,0)$ since rotations in the $x$ variable are automorphisms of $\mathbb H_n$. In particular, if $|x-x^0|<\ve< 10^{-4}$ and $|x-y|<\ve,$ we see that $|x_1-y_1|\sim 1$ and $|x_j-y_j|\ll 1$ for $j=2,\dots, d_1$.

\subsection{Decomposition into the non-horizontal and horizontal part}\label{nh+h}

Recall from the discussion preceding \eqref{eq:Klats-formula} that
\[
\supp\tilde\eta_0\subset [-\tfrac{5}{8}\pi,\tfrac{5}{8}\pi],
\]
and that $\pm\frac{\pi}{2}\in\supp\tilde\eta_0$. The parameters $\tau=\pm\frac{\pi}{2}$ correspond to horizontal points on $\Sigma_s$, that is, points at which the tangent plane to $\Sigma_s$ is horizontal. The heuristic discussion in the Introduction therefore suggests decomposing
\[
\tilde\eta_0 = \eta_{0,\mathsf{nh}} +\eta_{0,\mathsf{h}} +\eta_{0,\mathsf{\check h}}
\]
into smooth bump functions $\eta_{0,\mathsf{nh}},\eta_{0,\mathsf{h}}$ and $\eta_{0,\mathsf{\check h}}$ such that 
$$
\supp\eta_{0,\mathsf{nh}}\subset [-\tfrac 3{8} \pi, \tfrac 3{8} \pi]\quad \text{and} \  \supp\eta_{0,\mathsf{h}}\subset [\tfrac 3{16} \pi, \tfrac 58 \pi],
$$
where $\eta_{0,\mathsf{\check h}}(\tau):=\eta_{0,\mathsf{h}}(-\tau).$ Accordingly, we decompose 
\begin{equation}\label{horidecomp}
T^{0,\varepsilon}_\lambda=T^{0,\mathsf{nh},\varepsilon}_\lambda+T^{0,\mathsf{h},\varepsilon}_\lambda+T^{0,\mathsf{\check h},\varepsilon}_\lambda,
\end{equation}
where  the non-horizontal part $T^{0,\mathsf{nh},\varepsilon}_\lambda$ is defined like $T^{0,\varepsilon}_\lambda,$ only with $\tilde \eta_0$ in \eqref{eq:Klats-formula} replaced by $\eta_{0,\mathsf{nh}}$, and the other two summands are defined analogously.

By symmetry, estimates for $T^{0,\mathsf{\check h},\varepsilon}_\lambda$ can be reduced to estimates for $T^{0,\mathsf{h},\varepsilon}_\lambda,$ so that it suffices to work with $T^{0,\mathsf{nh},\varepsilon}_\lambda$ and $T^{0,\mathsf{h},\varepsilon}_\lambda$.

\subsection{The $TT^*$ argument for the operators $T^{0,\varepsilon}_\lambda$}\label{TT*0}

To prove that
\[
T^{0,\varepsilon}_\lambda:
L^2_{(x,u)}\to L_{(x,u)}^{p}(L_s^2)
\]
is bounded, we first decompose it as in \eqref{horidecomp}. We then apply the Stein--Tomas $TT^*$ argument separately to the non-horizontal and horizontal parts, proving
\[
L_{(x,u)}^{p'}(L_s^2)\to L_{(x,u)}^{p}(L_s^2)
\]
estimates for the operators $T^{0,\mathsf{nh},\varepsilon}_\lambda
(T^{0,\mathsf{nh},\varepsilon}_\lambda)^*$ and $T^{0,\mathsf{h},\varepsilon}_\lambda
(T^{0,\mathsf{h},\varepsilon}_\lambda)^*$.
We first record the corresponding kernel formula for the former operator.

\begin{lemma}\label{lem:0-TTstar}
Assume that $\ve>0$ is sufficiently small. Then, for $s,s'\in[1,1+\ve]$, the integral kernel $Q^{0,\mathsf{nh},\varepsilon}_\lambda$ of $T^{0,\mathsf{nh},\varepsilon}_\lambda (T^{0,\mathsf{nh},\varepsilon}_\lambda)^*$ is given by
\begin{multline}\label{eq:Q0}
Q^{0,\mathsf{nh},\varepsilon}_\lambda(x,u,s,x',u',s')
 = \lambda^{d + 1} \, \chi_{\varepsilon}(x-x^0,u-u^0) \, \chi_{\varepsilon}(x'-x^0,u'-u^0) \\
\times \int_\R \, \int_0^\infty \! \int_0^\infty \! \int_{\R^{d_1}}e^{i\lambda\Psi}  \, a_\lambda(s\sigma)  \,\overline{a_\lambda(s'\sigma')} \,\eta_{0,\mathsf{nh}}(\tau')  \, 
\chi_{\varepsilon}(y) \, h_\lambda \, dy \, d\sigma \, d\sigma' \, d\tau',
\end{multline}
where $h_\lambda=h_\lambda(\sigma,\sigma',\tau',x,y,u)$ is  smooth and  such that each partial derivative of $h_\lambda$ is uniformly bounded, if $\sigma,\sigma' \sim 1$, $ |\tau'|,|y| \lesssim 1$ and $\lambda\gg 1,$  and where the phase function $\Psi$  is given by 
\begin{align}
\Psi = \sigma s - \sigma's' & - \left|x-y\right|^2\sigma \, g(\tfrac{\sigma'}{\sigma}\tau') \notag \\ 
& + \left|x'-y\right|^2\sigma' g(\tau') + 4 \sigma'\tau'(u-u' + \tfrac12\langle J(x-x'),y \rangle), \label{eq:Psi}
\end{align}
where we recall that $g(\tau) = \tau\cot\tau$.
\end{lemma}

\begin{remark}\label{rem:h-variant}
Correspondingly, the integral kernel $Q^{0,\mathsf{h},\varepsilon}_\lambda$ of $T^{0,\mathsf{h}, \varepsilon}_\lambda (T^{0,\mathsf{h},\varepsilon}_\lambda)^*$ is given by the analogous formula, with each occurrence of $\eta_{0,\mathsf{nh}}$ replaced by $\eta_{0,\mathsf{h}}$.
\end{remark}

\begin{proof}
Recall that $a_\lambda$ in \eqref{eq:Klats-formula} is supported in an interval of length $\ve$ contained in $(\frac1{16},4)$ and that $\eta_{0,\mathsf{nh}}$ is supported in $[-\frac38\pi,\frac38\pi]$.
We put $z^0=(x^0,u^0),$ and let $z=(x,u), z'=(x',u')\in\R^d.$ Then, for $\varphi\in\cS(\bbH_n),$
\begin{align*}
(T^{0,\mathsf{nh},\varepsilon}_\lambda \varphi)(z,s)
 & = \tilde\chi_{\varepsilon}(z)\big((\chi_{\varepsilon} \varphi)*K_{\lambda,s}^{0,\mathsf{nh}}\big) (z) \\
 & = \tilde\chi_{\varepsilon}(z)\int_{\bbH_n}\chi_{\varepsilon} (w) \,  K^{0,\mathsf{nh}}_{\lambda,s}\big((w)^{-1}z\big) \, \varphi(w)\, dw,\\
\big((T^{0,\mathsf{nh},\varepsilon}_\la)^*f\big) (w)
&= \chi_{\varepsilon} (w) \int_1^{1+\ve} \int_{\bbH_n} \tilde \chi_\ve(z')\, \overline{ K_{\lambda,s'}^{0,\mathsf{nh}}}\big(w^{-1}z'\big)f(z',s')\, dz' \,ds',
\end{align*}
where $K_{\lambda,s}^{0,\mathsf{nh}}$ is defined as $K_{\lambda,s}^{0}$ in \eqref{eq:Klats-formula}, but with $\tilde \eta_{0}$ replaced by $\eta_{0,\mathsf{nh}}$.
We obtain
\[
\big(T^{0,\mathsf{nh},\varepsilon}_\la (T^{0,\mathsf{nh},\varepsilon}_\la)^*f\big)(z,s)
 = \int_1^{1+\ve} \int_{\bbH_n} \, Q^{0,\mathsf{nh},\varepsilon}_\lambda(z,s,z',s') \,  f(z',s')\, dz'\, ds',
\]
where
\[
 Q^{0,\mathsf{nh},\varepsilon}_\lambda(z,s,z',s') =\tilde \chi_{\varepsilon}(z) \, \tilde\chi_{\varepsilon}(z') \,
  \int_{\bbH_n} K_{\lambda,s}^{0,\mathsf{nh}}\big(w^{-1}z\big) \, \overline{ K_{\lambda,s'}^{0,\mathsf{nh}}}\big(w^{-1}z'\big) \, \chi_{\varepsilon}^2(w) \, dw. 
\]
Moreover, by \eqref{eq:Klats-formula} and \eqref{eq:phasesdef},
\begin{align*}
& K_{\lambda,s}^{0,\mathsf{nh}}\big((y,v)^{-1}(x,u)\big) \, \overline{ K_{\lambda,s'}^{0,\mathsf{nh}}}\big((y,v)^{-1}(x',u')\big)\\
&= K_{\lambda,s}^{0,\mathsf{nh}}\big(x-y,u-v + \tfrac12\langle Jx,y\rangle\big) \, \overline{ K_{\lambda,s'}^{0,\mathsf{nh}}}\big(x'-y,u'-v + \tfrac12\langle Jx',y\rangle)\big) \\ 
&= \lambda^{d+2} \int_\R \, \int_0^\infty \int_\R \, \int_0^\infty e^{i\lambda\Psi_0} \, a_\lambda(s\sigma) \, \eta_{0,\mathsf{nh}}(\tau)  \,\overline{a_\lambda(s'\sigma')} \,\eta_{0,\mathsf{nh}}(\tau')  d\sigma \, d\tau \, d\sigma ' \, d\tau',
\end{align*}
and
\begin{align*}
\Psi_0 & = \sigma\big(s-\left|x-y\right|^2g(\tau) + 4 \tau(u-v + \tfrac12 \langle Jx,y\rangle)\big)\\
 &\quad -\sigma'\big(s'-\left|x'-y\right|^2g(\tau') + 4 \tau' (u'-v + \tfrac12 \langle Jx',y\rangle)\big).
\end{align*}

Recall that $w=(y,v)$.  We may assume that $\chi(w)$ equals a product of a smooth bump in $y$ and a smooth bump in $v$.  To help keep expressions more compact, we incorporate this assumption by a slight abuse of notation, writing $\chi^2_\ve(w)=\chi_\ve(y)\chi_\ve(v)$. With this assumption, the only term depending on $v$ in the previous identity for $Q^{0,\mathsf{nh}\varepsilon}_\lambda(z,s,z',s')$ is $e^{-i\la 4 (\sigma\tau-\sigma'\tau') v} \chi_\ve(v).$ Since  
\[
\int_{\R} \, e^{-i\la 4 (\sigma\tau-\sigma'\tau') v} \chi_\ve(v)\, dv=\widehat{\chi_\ve}\big(4\la(\sigma\tau-\sigma'\tau')\big),
\] 
we may simplify the previous expression for $ Q^{0,\mathsf{nh},\varepsilon}_\lambda$ to 
\begin{multline*}
 Q^{0,\mathsf{nh}\varepsilon}_\lambda(z,s,z',s') = \tilde \chi_{\varepsilon}(z) \, \tilde\chi_{\varepsilon}(z') \, \lambda^{d + 2} \\
 \times \int_{\R^{d_1}} \int_\R \, \int_0^\infty \int_\R \, \int_0^\infty e^{i\lambda\Psi_1} \chi_{\varepsilon}(y) \,\widehat{\chi_\ve}\big( 4\lambda(\sigma\tau-\sigma'\tau')\big) \\
  \times a_\lambda(s\sigma) \, \eta_{0,\mathsf{nh}}(\tau)  \,\overline{a_\lambda(s'\sigma')} \,\eta_{0,\mathsf{nh}}(\tau')\,  d\sigma \, d\tau \, d\sigma ' \, d\tau'\, d y, \label{eq:Q0-v}
\end{multline*}
where
\begin{align*}
\Psi_1 & = \sigma\big(s-\left|x-y\right|^2g(\tau) + 4 \tau (u + \tfrac12 \langle Jx,y\rangle)\big)\\
&\quad-\sigma'\big(s'-\left|x'-y\right|^2g(\tau') + 4 \tau'(u' + \tfrac12 \langle Jx',y\rangle)\big).
\end{align*}
We change variables from $\tau$ to $\vartheta$, defined by 
$$
\tau = \tfrac{\sigma'}\sigma \tau' + \vartheta, \quad\text{i.e., } \  \sigma\tau-\sigma'\tau' = \sigma\vartheta.
$$
Next, we decompose $\Psi_1$ as $\Psi_1 = \Psi_2 + \Psi_3$,
where
\[
\Psi_2 = \sigma \left|x-y\right|^2\big(g(\tfrac{\sigma'}\sigma \tau')-g(\tfrac{\sigma'}\sigma \tau' + \vartheta)\big) + 
4 \sigma\vartheta (u + \tfrac12\langle J x,y\rangle),
\]
and
\begin{align*}
\Psi_3 & = \sigma\big(s-\left|x-y\right|^2g(\tfrac{\sigma'}\sigma \tau') + 4 \tfrac{\sigma'}{\sigma}\tau'(u + \tfrac12\langle J x,y\rangle)\big)\\
&\quad -\sigma'\big(s'-\left|x'-y\right|^2g(\tau') + 4 \tau'(u' + \tfrac12\langle J x',y\rangle\big)).
\end{align*}
Note that if we choose $\ve>0$ sufficiently small in \eqref{eq:T0lambdas}, then, due to our support assumption on $a_\la$, we may assume that $\sigma'/\sigma$ is very close to $1,$ so that the argument $\tfrac{\sigma'}\sigma \tau'$ will still stay away from singularities of $g$. 

Let $h_\lambda$ be given by
\[
h_\lambda(\sigma,\sigma',\tau',x,y,u) = \lambda \int_{\R} \, e^{i\lambda\Psi_2} \, \eta_{0,\mathsf{nh}}(\tfrac{\sigma'}\sigma \tau' + \vartheta) \, \widehat{\chi_\ve}(4\lambda\sigma \vartheta) \, d\vartheta.
\]
Then $h_\lambda$ is smooth, and  each partial derivative of $h_\lambda$ is uniformly bounded, if $\sigma,\sigma' \sim 1$, $ |\tau'|,|y| \lesssim 1$ and $\lambda\gg 1$.  Indeed, for $|\vartheta|\gg 1,$ this smoothness is clear since $\widehat{\chi_\ve}\in \cS$, and for $|\vartheta|\lesssim 1$, this can be seen by writing $g(\tfrac{\sigma'}\sigma \tau')-g(\tfrac{\sigma'}\sigma \tau' + \vartheta)=\vartheta \,G(\sigma,\sigma',\tau',\vartheta)$ and scaling in $\vartheta$ by $\la^{-1}.$ 

Thus, the previous expression for  $Q^{0,\mathsf{nh}\varepsilon}_\lambda(z,s,z',s')$ becomes
\begin{multline*}
 Q^{0,\mathsf{nh},\varepsilon}_\lambda(z,s,z',s') =
\tilde \chi_{\varepsilon}(z) \, \tilde\chi_{\varepsilon}(z') \,\lambda^{d+1} \\
\times \int_\R \, \int_0^\infty \! \int_0^\infty \! \int_{\R^{d_1}}e^{i\lambda\Psi_3} \, a_\lambda(s\sigma)  \,\overline{a_\lambda(s'\sigma')} \,\eta_{0,\mathsf{nh}}(\tau')  \, 
\chi_{\varepsilon}(y) \, h_\lambda \, dy \, d\sigma \, d\sigma' \, d\tau'.
\end{multline*}
This finishes the proof of Lemma~\ref{lem:0-TTstar}, since $\Psi=\Psi_3.$
\end{proof}

\section{The non-horizontal part} \label{sec:k=0-nonhor}

In this section, we prove Proposition~\ref{L:T0lambda} with $T^{0,\mathsf{nh}}_{\lambda,t}$ in place of $T^0_{\lambda,t}$.

\subsection{The key estimates for the non-horizontal part}

Given $x,x'\in \R^{d_1}$, we split coordinates as $x = (x_1,\underline x)$ and $x' = (x_1',\underline x')$. For fixed $x_1\in \R$, let $T_{\lambda,x_1}^{0,\mathsf{nh},\varepsilon}:\scriptS(\R^{d})\to \scriptS(\R^{d})$ be the operator given by
\[
(T_{\lambda,x_1}^{0,\mathsf{nh}, \varepsilon} \varphi)(\underline x,u,s) = (T_{\lambda}^{0,\mathsf{nh},\varepsilon} \varphi)(x_1,\underline x,u,s).
\]
If we write $f(x',u',s') = f_{x_1'}(\underline x',u',s')$, then
\begin{equation}\label{eq:x_1-x_1'}
\big(T^{0,\mathsf{nh},\varepsilon}_\la (T^{0,\mathsf{nh},\varepsilon}_\la)^*f\big)(x,u,s) = \int_\R \, \big( T_{\lambda,x_1}^{0,\mathsf{nh},\varepsilon} (T_{\lambda,x_1'}^{0,\mathsf{nh},\varepsilon} )^* \, f_{x_1'}\big) (\underline x,u,s) \, dx_1'.
\end{equation}
Let $Q_{\lambda,x_1,x_1'}^{0,\mathsf{nh},\varepsilon}$ denote the integral kernel of the operator $T_{\lambda,x_1}^{0,\mathsf{nh},\varepsilon} (T_{\lambda,x_1'}^{0,\mathsf{nh},\varepsilon} )^*$. Note that
\begin{equation}\label{eq:frozen-Q}
Q_{\lambda,x_1,x_1'}^{0,\mathsf{nh},\varepsilon}(\underline x,u,s,\underline x',u',s')
= Q_\lambda^{0,\mathsf{nh},\varepsilon}(x,u,s,x',u',s').    
\end{equation}

In the subsequent discussions, we shall always assume  the constant $\ve>0$ to be  sufficiently small, usually without further mentioning.

\begin{proposition}\label{lem:0-L1} 
We have
\[
\|T_{\lambda,x_1}^{0,\mathsf{nh},\varepsilon} (T_{\lambda,x_1'}^{0,\mathsf{nh},\varepsilon} )^* \|_{L_{(\underline x,u)}^{1}(L^2_s)\to L_{(\underline x,u)}^\infty(L^2_s)}
\lesssim \lambda^{d-1} \, \langle\lambda\left|x_1'-x_1\right|\rangle^{-\frac{d-2}{2}} ,
\]
with control of the implicit constant in this estimate as in  Proposition~\ref{L:T0lambda}.
\end{proposition}

\begin{proposition}\label{lem:0-L2}
We have
\[
\|T_{\lambda,x_1}^{0,\mathsf{nh},\varepsilon} (T_{\lambda,x_1'}^{0,\mathsf{nh},\varepsilon} )^*\|_{L^2(\R^{d})\to L^2(\R^{d})} \lesssim 1,
\]
with control of the implicit constant in this estimate as in Proposition~\ref{L:T0lambda}.
\end{proposition}

We postpone the proofs of Propositions \ref{lem:0-L1} and \ref{lem:0-L2} for the moment and show how they imply the estimates in Proposition~\ref{L:T0lambda} for  
$T_\lambda^{0,\mathsf{nh},\varepsilon}$.

\begin{proof}[Proof of Proposition~\ref{L:T0lambda} for $T_\lambda^{0,\mathsf{nh},\varepsilon}$]
Since $p> 2,$ using \eqref{eq:x_1-x_1'} and  Minkowski's inequality, we obtain
\[
\|T_\lambda^{0,\mathsf{nh},\varepsilon} (T_\lambda^{0,\mathsf{nh},\varepsilon})^* f\|_{L_{(x,u)}^p(L_s^2)}
 \le \Big\| \int_\R \, \big\| T_{\lambda,x_1}^{0,\mathsf{nh},\varepsilon} (T_{\lambda,x_1'}^{0,\mathsf{nh},\varepsilon} )^* \, f_{x_1'} \big\|_{L_{(\underline x,u)}^p(L_s^2)} \, dx_1' \Big \|_{L_{x_1}^p}.
\]
Interpolating the estimates of Propositions \ref{lem:0-L1} and \ref{lem:0-L2}, we get
\[
\|T_{\lambda,x_1}^{0,\mathsf{nh},\varepsilon} (T_{\lambda,x_1'}^{0,\mathsf{nh},\varepsilon} )^* f_{x_1'}\|_{L_{(\underline x,u)}^p(L^2_s)}
\lesssim \lambda^{(d-1)(1-\frac 2 p)} \, \langle\lambda\left|x_1'-x_1\right|\rangle^{-\gamma_p} \, \|f_{x_1'}\|_{L_{(\underline x,u)}^{p'}(L^2_s)},
\]
where $\gamma_p = \frac{d-2}2 (1-\frac 2 p)$. 

Assume first that $p = \frac{2d}{d-2}.$ Then $(d-1)(1-\frac 2 p)-\gamma_p =\frac d2(1-\frac 2 p)=1,$  and if we choose $q$ so that  $1+\frac 1p=\frac 1q+\frac 1{p'},$ then $q=\frac p 2>1,$ and so $q\gamma_p =\frac{d-2}2 (\frac p2-1)=1.$ Moreover,
$2\big(d(\frac 12 -\frac 1p)-\frac 12\big)=1.$

Thus, applying Hardy--Littlewood--Sobolev's inequality, we get
\begin{multline*}
\lambda^{(d-1)(1-\frac 2 p)} \, \Big\| \int_\R \, \langle\lambda\left|x_1'-x_1\right|\rangle^{-\gamma_p} \, \|f_{x_1'}\|_{L_{(\underline x,u)}^{p'}(L^2_s)} \, dx_1' \Big \|_{L_{x_1}^p} \\
 \lesssim \lambda \|f\|_{L_{(x,u)}^{p'}(L^2_s)}=\lambda^{2\big(d(\frac 12 -\frac 1p)-\frac 12\big)} \|f\|_{L_{(x,u)}^{p'}(L^2_s)},
\end{multline*}
hence
\[
\| T_\lambda^{0,\mathsf{nh},\varepsilon} \|_{L_{(x,u)}^2 \to L_{(x,u)}^p(L^2_s) } \lesssim \lambda^{d(\frac 12 -\frac 1p)-\frac 12},
\]
which proves \eqref{eq:T0lambdas} for this critical Lebesgue exponent $p$, with $\alpha=d(\frac 12 -\frac 1p)-\frac 12.$ 

For $p=\infty,$ since we may assume that $|x|,|x'|\lesssim 1,$ and since $\langle\lambda\left|x_1'-x_1\right|\rangle^{-\gamma_p}\le 1$, we can trivially estimate 
$$\lambda^{(d-1)(1-\frac 2 p)} \, \Big\| \int_\R \, \langle\lambda\left|x_1'-x_1\right|\rangle^{-\gamma_p} \, \|f_{x_1'}\|_{L_{(\underline x,u)}^{p'}(L^2_s)} \, dx_1' \Big \|_{L_{x_1}^p}\lesssim \lambda^{d-1} \|f\|_{L_{(x,u)}^{1}(L^2_s)},
$$
so that \eqref{eq:T0lambdas}  holds as well. 

Then \eqref{eq:T0lambdas} follows by interpolation, which completes the proof of Proposition~\ref{L:T0lambda} with $T_\lambda^{0,\mathsf{nh},\varepsilon}$ in place of $T_\lambda^{0,\varepsilon}$.  
\end{proof}

\subsection{Pointwise estimates} \label{pointwise0}

In this section, we prove Proposition~\ref{lem:0-L1}. Recall from Lemma~\ref{lem:0-TTstar} that the integral kernel $Q^{0,\mathsf{nh},\varepsilon}_\lambda$ of $T^{0,\mathsf{nh},\varepsilon}_\lambda (T^{0,\mathsf{nh},\varepsilon}_\lambda)^*$ is given by
\begin{multline*}
Q^{0,\mathsf{nh},\varepsilon}_\lambda(x,u,s,x',u',s')
 = \lambda^{d + 1} \, \chi_{\varepsilon}(x-x^0,u-u^0) \, \chi_{\varepsilon}(x'-x^0,u'-u^0) \\
\times \int_\R \, \int_0^\infty \! \int_0^\infty \! \int_{\R^{d_1}}e^{i\lambda\Psi} 
\,a_\lambda(s\sigma)  \,\overline{a_\lambda(s'\sigma')}  \, 
\chi_{\varepsilon}(y) \, h_\lambda \, dy \, d\sigma \, d\sigma' \, \eta_{0,\mathsf{nh}}( \tau')\,  d\tau',
\end{multline*}
where $\eta_{0,\mathsf{nh}}$ is supported in $[-\frac 3{8} \pi, \frac 3{8} \pi]$\, with phase
\begin{align*}
\Psi = \sigma s - \sigma's' & - \left|x-y\right|^2\sigma'\tau' \cot(\tfrac{\sigma'}\sigma \tau') \\ 
& + \left|x'-y\right|^2\sigma' \tau' \cot \tau' + 4\sigma'\tau'(u-u' + \tfrac12\langle J(x-x'),y \rangle).
\end{align*}

The localizations given by the cut-offs $\chi_\ve$ allow us to assume in the subsequent discussions that  $|(x,u)-(x^0,u^0)|\le \ve$, $|(x',u')-(x^0,u^0)|\le \ve$, and $|y|\le \ve.$ We recall that $|(x^0,u^0)|\le 10$,   $|x^0|\ge\frac{1}{100},$ and  $x^0=(|x^0|,0, \dots,0)$.

\medskip

Proposition~\ref{lem:0-L1} will follow immediately from the next proposition   by means of Schur's test, since $|\Delta x| \geq |\Delta x_1|$.
\begin{proposition}\label{lem:pointwiseQ0} 
If the constant $\ve>0$ is sufficiently small, then for all $s\in[1,1+\ve],$
\begin{equation}\label{sprimeintest}
\int_1^{1+\ve}| Q^{0,\mathsf{nh},\varepsilon}_\lambda(x,u,s,x',u',s')|\, ds'  
\le C(\ve,a_\la)\, \lambda^{d-1}
\, \langle\lambda\left|\Delta x\right|\rangle^{-\frac{d_1-1}2} ,
\end{equation}
and for all $s'\in[1,1+\ve],$
\begin{equation}\label{sintest}
\int_1^{1+\ve}| Q^{0,\mathsf{nh},\varepsilon}_\lambda(x,u,s,x',u',s')|\, ds 
\le C(\ve,a_\la)\, \lambda^{d-1}
\, \langle\lambda\left|\Delta x\right|\rangle^{-\frac{d_1-1}2} ,
\end{equation}
uniformly in $(x,u)$ and $(x',u'),$ 
where the  constant $C(\ve,a_\la)$ can be chosen to be independent of  $(x^0,u^0)$ and to depend only on the $C^M$-norm of $a_\la$ for some sufficiently large $M$.
\end{proposition}

We will prove Proposition~\ref{lem:pointwiseQ0} by showing the following.

\begin{lemma}\label{lem:pointwiseQ0-precise} 
We can decompose $Q^{0,\mathsf{nh},\varepsilon}_\lambda=Q^{\varepsilon}_{1,\lambda}+E^\ve_\la,$ where the main term 
$Q^{\varepsilon}_{1,\lambda}$ can be estimated pointwise by 
\begin{multline}\label{eq:pointwiseQ0}
|Q^{\varepsilon}_{1,\lambda}(x,u,s,x',u',s')| \le C_N(\ve,a_\la)\, \lambda^d
\, \langle\lambda\left|\Delta x\right|\rangle^{-\frac{d_1-1}2} \\
\times \int_\R \big\langle \lambda\, (\Delta s - q(s,\Delta x,\Delta w,\tau') ) \big\rangle^{-N} \, \eta_{0,\mathsf{nh}}(\tau')\, d\tau'
\end{multline}
for any $N\in\N,$ where 
$$
\Delta x = x-x',\, \Delta w = u-u' +\frac 12\langle J\Delta x,x\rangle \text{ and }\Delta s=s-s',
$$ 
and where $q$ is a smooth function satisfying estimates of the form $\partial^j_s q=\mathcal O(\ve)$ for every $j\in\N.$
Here, the  constant $C_N(\ve,a_\la)$ can be chosen to be independent of  $\tau'$ and to depend, for a given $N\in\N,$  in $a_\la$ only on the $C^M$-norm of $a_\la$ for some sufficiently large $M$ (which may depend on $N.$) 

For every $N\in\N$, for all $s\in[1,1+\ve]$,
\[
\int_1^{1+\ve}| E^\ve_\la(x,u,s,x',u',s')|\, ds'  
\le C_N(\ve,a_\la)\, \lambda^d
\, \langle\lambda\left|\Delta x\right|\rangle^{-N} ,
\]
and for all $s'\in[1,1+\ve]$,
\[
\int_1^{1+\ve}| E^\ve_\la(x,u,s,x',u',s')|\, ds 
\le C_N(\ve,a_\la)\, \lambda^d
\, \langle\lambda\left|\Delta x\right|\rangle^{-N} .
\]
\end{lemma}

Before the prove this lemma, we show how Proposition~\ref{lem:pointwiseQ0} is derived from it.

\begin{proof}[Proof of Proposition~\ref{lem:pointwiseQ0}]
We need to show that the  main term $Q^{\varepsilon}_{1,\lambda}$ will satisfy estimates of the form \eqref{sprimeintest} and \eqref{sintest}. For $N\ge 2,$ we clearly have 
$$
\int_1^{1+\ve}  \big\langle\lambda\, (\Delta s - q(s,\Delta x,\Delta w,\tau') ) \big\rangle^{-N} \,ds'\lesssim \la^{-1}.
$$
Noting  that the properties of the function $q$ imply that
$$
\partial_{s}\big(\Delta s - q(s,\Delta x,\Delta w,\tau') \big)=1+\mathcal O(\ve)\sim 1,
$$
the  inverse function theorem allows to see that the  estimate 
$$
\int_1^{1+\ve}  \big\langle\lambda\, (\Delta s - q(s,\Delta x,\Delta w,\tau') ) \big\rangle^{-N} \,ds\lesssim \la^{-1}
$$
holds as well. The estimates \eqref{sprimeintest} and \eqref{sintest} are immediate consequences.
\end{proof}

The rest of this subsection will be  devoted to the proof of Lemma~\ref{lem:pointwiseQ0-precise}.

\subsubsection{Freezing $\tau'$}
We freeze $\tau'\in \supp \eta_{0,\mathsf{nh}},$  and consider the oscillatory integral
\begin{multline}\label{eq:Q0-v-2}
Q^{0,\varepsilon}_{\lambda,\tau'}(x,u,s,x',u',s') := \lambda^{d + 1} \int_{\sigma' \sim 1} \int_{\sigma \sim 1} \int_{\R^{d_1}} e^{i\lambda\Psi} \,a_\lambda(s\sigma)  \,\overline{a_\lambda(s'\sigma')}  \\ 
\times \chi_{\varepsilon}(y) \, h_\lambda \, dy \, d\sigma \, d\sigma'
\end{multline}
(by writing $\sigma\sim 1,$ etc., we wish to remind that $a_\la$ is supported in an interval of length $\ve$ contained in 
$(\frac 1{16},4)$). Note also that 
\begin{multline*}
Q^{0,\mathsf{nh},\varepsilon}_\lambda(x,u,s,x',u',s')
 =  \chi_{\varepsilon}(x-x^0,u-u^0) \, \chi_{\varepsilon}(x'-x^0,u'-u^0)\\
 \times\int_\R Q^{0,\varepsilon}_{\lambda,\tau'}(x,u,s,x',u',s') \, \eta_{0,\mathsf{nh}}( \tau')\,  d\tau'.
\end{multline*}

\subsubsection{Changes of  variables and  stationary phase for a spherical integration}

First, by  a change of variables from $y$ to $y'=y- x$, \eqref{eq:Q0-v-2} becomes
\begin{equation}\label{eq:Q0-y}
\lambda^{d + 1} \int_{\sigma' \sim 1} \int_{\sigma \sim 1} \int_{\R^{d_1}} e^{i\lambda\Psi_1} \,a_\lambda(s\sigma)  \,\overline{a_\lambda(s'\sigma')}  \, 
\chi_{\varepsilon}(y'+x) \,h_{1,\lambda} \, dy' \, d\sigma \, d\sigma',
\end{equation}
where  $ h_{1,\lambda}$ is a  slightly modified version of $h_\la$ sharing the essential  properties with $h_\la,$ with  phase function
\begin{align*}
\Psi_1 = \sigma s - \sigma's'& - \left|y'\right|^2\sigma'\tau' \cot(\tfrac{\sigma'}\sigma \tau') \\
& + \left|y' + \Delta x\right|^2\sigma' \tau' \cot \tau' + 4\sigma'\tau'(\Delta w + \tfrac12\langle J\Delta x,y' \rangle).
\end{align*}
Next, we use polar coordinates in $y'$, that is, $y' = \rho \, \omega$, where $\rho \in (0,\infty)$ satisfies $\rho \sim 1$, since 
$|y'+x^0|<2\ve,$ and where $\omega$ lies in the unit sphere $S^{d_1-1}$ of $\R^{d_1}.$  Denoting by $d\omega$ the surface measure of $S^{d_1-1}$,  by this change of coordinates \eqref{eq:Q0-y} equals
\[
\lambda^{d + 1} \int_{\sigma' \sim 1} \int_{\sigma \sim 1}  \int_{\rho \sim 1}  \int_{S^{d_1-1}} e^{i\lambda \Psi_2} \, \,a_\lambda(s\sigma)  \,\overline{a_\lambda(s'\sigma')}  \, \chi_{\varepsilon}(\rho\omega + x) \, h_{2,\lambda}\, d\omega \, d\rho  \, d\sigma \, d\sigma' ,
\]
with a function $ h_{2,\lambda}$  having analogous properties in these new coordinates to those of $h_{1,\la}, $ and 
where
\begin{multline*}
\Psi_2 = \sigma s - \sigma's'  + \rho^2\big(\sigma' \tau' \cot \tau' -\sigma'\tau' \cot(\tfrac{\sigma'}\sigma \tau')\big) 
 + 2\rho \, \big\langle \omega, \Delta x \, \sigma' \tau' \cot \tau' +\sigma'\tau' J\Delta x\big\rangle \\
 + |\Delta x|^2\sigma' \tau' \cot \tau' + 4 \sigma'\tau' \Delta w.
\end{multline*}
Note that the factor $\chi_{\varepsilon}(\rho\omega + x)$ localizes the integration in $\omega$ to an $\mathcal{O} (\ve)$- neighborhood of the point $-x^0/|x^0|$ in $S^{d_1-1}.$

Finally, we change variables from $\sigma$ to $\beta$ such that
\[
\frac{\sigma'}\sigma = 1 + |\Delta x| \, \beta.
\]
Note that since $a_\la$ is supported in an interval of length $\ve$ within $(\frac{1}{16}, 4)$ and $s,s'\in[1,1+\ve],$ we may assume that $|\Delta x| \, |\beta|\lesssim \ve,$ which will allow us to add a cut-off  factor $\chi_\ve(\left|\Delta x\right|\beta)$ to the amplitude.

We can then write \eqref{eq:Q0-y} as 
\begin{multline}\label{eq:Q0-polar}
\lambda^d (\la |\Delta x|) \int_{\sigma' \sim 1} \int  \int_{\rho \sim 1}  \int_{S^{d_1-1}} e^{i\lambda \Psi_3}  
 \, \overline{a_\lambda(s'\sigma')}  \\ \times \chi_{\varepsilon}(\rho\omega + x) \, \chi_\ve(\left|\Delta x\right|\beta)\,\tilde a_\la\, h_{3,\lambda}\,d\omega \, d\rho  \, d\beta \, d\sigma' ,
\end{multline}
where $\tilde a_\la$ stands for $a_\lambda(s\frac{\sigma'}{1 + |\Delta x|  \beta}),$
with $h_{3,\lambda}$ sharing the essential  properties with $h_\la,$
and with phase 
\begin{align*}
\Psi_ 3 &=  \sigma'\Big[\Delta s - \frac{s\left|\Delta x\right| \beta }{1 + \left|\Delta x\right|\beta} +\rho^2 \psi_2+2\rho \, \big\langle \omega, \Delta x \,\tau' \cot \tau' +\tau' J\Delta x\big\rangle\\
 & \hskip6cm + \left|\Delta x\right|^2 \tau' \cot \tau' + 4 \tau' \Delta w\Big]\\
 &= \sigma'\Big[\Delta s-q(s,\Delta x,\Delta w,\tau',\omega,\rho,|\Delta x|\beta)\Big].
\end{align*}
Here, we have set
\[
\psi_2  = - \tau' \big(\cot\big((1 + \left|\Delta x\right|\beta) \tau' \big)-\cot \tau' \big).
\]
Note that
\[
\psi_2 \sim \Big(\frac{ \tau' }{\sin \tau' }\Big)^2 \left|\Delta x\right|\beta \sim \left|\Delta x\right|\beta\quad\text{and}\quad
|\big\langle \omega, \Delta x \, \tau' \cot \tau' +\tau' J\Delta x\big\rangle|\lesssim |\Delta x|.
\]
Let us denote by $E^{\varepsilon,2}_{\lambda,\tau'}$ the contribution of the region where $|\beta|\gg 1$ to \eqref{eq:Q0-polar}. For $\beta$ in this region, we see that we can integrate by parts in  $\rho$ in order to gain negative powers  $\jp{\lambda|\Delta x||\beta|}^{-N}.$ Applying subsequently integrations by parts in $\sigma',$ we finally see that we may write 
\begin{align*}
E^{\varepsilon,2}_{\lambda,\tau'}
= \la^d \,\int_{\sigma' \sim 1} \int_{|\beta|\gg 1}  \int_{\rho \sim 1}  \int_{S^{d_1-1}} 
& \big\langle \lambda\, (\Delta s - q(s,\Delta x,\Delta w,\tau',\omega,\rho,|\Delta x|\beta) ) \big\rangle^{-N} \\
& \times  (\la |\Delta x|)\, \jp{\lambda|\Delta x||\beta|}^{-N}\,
e^{i\lambda \Psi_3}  \,\overline{a_\lambda(s'\sigma')}  \\
& \times \chi_{\varepsilon}(\rho\omega + x) \,  \chi_\ve(\left|\Delta x\right|\beta)\, \tilde a_\la \,h_{N,\lambda}\,d\omega \, d\rho  \, d\beta \, d\sigma' 
\end{align*}

Observing  that  also here $\partial^j_s q=\mathcal O(\ve)$ for every $j\in\N,$ we can then essentially argue  as before for the main term  $Q^\ve_{1,\la}$ to prove that the contribution by the error term $E^{\varepsilon,2}_{\lambda,\tau'}$ even  satisfies  the stronger error term estimates,  with the right-hand sides in \eqref{sprimeintest} and \eqref{sintest} replaced by 
$C_N(\ve,a_\la)\, \lambda^d\, \langle\lambda\left|\Delta x\right|\rangle^{-N}.$ Note to this end that  for $N\ge2$
$$
\int(\la |\Delta x|)\jp{\lambda|\Delta x||\beta|}^{-N}\, d\beta\lesssim 1,
$$
 so that we obtain these estimates when  $\la |\Delta x|\le 1,$ and when $\la |\Delta x|>1,$ then  
$$
(\la |\Delta x|)\jp{\lambda|\Delta x||\beta|}^{-N}\lesssim (\la |\Delta x|)^{1-N} \,|\beta|^{-N},
$$ 
so that the integral in $\beta\gg 1$ is finite, and we  also gain any factor  $\langle\la |\Delta x|\rangle^{1-N}.$
\smallskip

We are thus reduced to considering the contribution by the  region where $|\beta|\lesssim 1,$ i.e., 
\begin{multline*}
Q^{\varepsilon,2}_{\lambda,\tau'}=\lambda^d (\la |\Delta x|) \int_{\sigma' \sim 1} \int_{|\beta|\lesssim 1}  \int_{\rho \sim 1}  \int_{S^{d_1-1}} e^{i\lambda \Psi_3}  \,\overline{a_\lambda(s'\sigma')} \\ 
 \times \chi_{\varepsilon}(\rho\omega + x) 
\, \chi_\ve(\left|\Delta x\right|\beta)\,\tilde a_\la \, h_{3,\lambda}\,d\omega \, d\rho  \, d\beta \, d\sigma' .
\end{multline*}

We  examine the term in the phase $\Psi_3$ involving the inner product with $\omega$. Using that $J$ is skew-symmetric with $ J^2 = - \mathrm{id}_{\R^{d_1}}$, we get
\begin{align*}
\bigl| \Delta x \, \sigma' \tau' \cot \tau' + \sigma' \tau' J\Delta x\bigr|& = \sqrt{\left( \left|\Delta x\right| \sigma' \tau' \cot \tau' \right)^2 + \left| \sigma'\tau'\right|^2\left|\Delta x\right|^2} \\
& = \sigma'\left|\Delta x\right| \frac{\tau'}{ \sin \tau' }\sim |\Delta x|.
\end{align*}
Thus, performing stationary phase in $\omega$ when $|\underline{\Delta x}|\lesssim \ve |\Delta x_1|,$   and $N$ integrations by parts in $\omega$ along $S^{d_1-1}$ when $|\underline{\Delta x}|\gg  \ve |\Delta x_1|,$ we find that  \eqref{eq:Q0-polar}
is the sum of two corresponding terms
\begin{equation}\label{eq:Q0-stat}
Q^{\varepsilon,2}_{\lambda,\tau'}=Q^{\varepsilon,3}_{\lambda,\tau'}+E^{\varepsilon,3}_{\lambda,\tau'}.
\end{equation}
Here, the first term is given by 
\begin{equation*}
Q^{\varepsilon,3}_{\lambda,\tau'}=\lambda^d (\la |\Delta x|) \int_{\sigma' \sim 1} \int_{|\beta|\lesssim 1} \int_{\rho \sim 1} \jp{\lambda\Delta x}^{-\frac{d_1-1}2} \, e^{i\lambda \Psi_4^\pm} \, \tilde a_\la  \,\overline{a_\lambda(s'\sigma')}\, h_{4,\lambda}\, d\rho \, d\beta \, d\sigma',
\end{equation*}
where the function $h_{4,\lambda}$ has similar properties as its predecessors and in addition localizes to the region where 
$|\underline{\Delta x}|\lesssim \ve |\Delta x_1|,$ and the  phase is given by
\[
\Psi_ 4 ^\pm = \sigma'\Big[\Delta s - \frac{s\left|\Delta x\right| \beta }{1 + \left|\Delta x\right|\beta} +\psi^\pm
 + \left|\Delta x\right|^2 \tau' \cot \tau' + 4 \tau' \Delta w\Big],
\]
where  $\psi^\pm=\rho^2\, \psi_2\pm 2\rho\, \psi_1$
collects all terms in the phase depending on $\rho$, with
\[
\psi_2  = - \tau' \big(\cot\big((1 + \left|\Delta x\right|\beta) \tau' \big)-\cot \tau' \big)
 \quad\text{and}\quad
\psi_1=  \frac{ \tau' }{\sin \tau' } \left|\Delta x\right|.
\]
The sign in this phase  depends on whether $\Delta x_1$ is positive, or negative. Moreover, with a slight abuse of notation, the factor ${\jp{\lambda\Delta x}^{-\frac{d_1-1}2}}$ actually denotes a symbol of order ${-\frac{d_1-1}{2}}$ in $\rho(\Delta x \, \sigma' \tau' \cot \tau' + \sigma' \tau' J \Delta x)$.

Note that
\[
\psi_2 \sim \Big(\frac{ \tau' }{\sin \tau' }\Big)^2 \left|\Delta x\right|\beta \sim \left|\Delta x\right|\beta.
\]

The second term in \eqref{eq:Q0-stat} is again an error term, which for any $N\in\N$ can be written as 
\begin{multline*}
E^{\varepsilon,3}_{\lambda,\tau'}
=\lambda^d (\la |\Delta x|) \int_{\sigma' \sim 1} \int_{|\beta|\lesssim 1} \int_{\rho \sim 1}  \int_{S^{d_1-1}} \jp{\lambda\Delta x}^{-N}\,e^{i\lambda \Psi_3}  \,\overline{a_\lambda(s'\sigma')}  \, \chi_{\varepsilon}(\rho\omega + x) \\
 \times \,\tilde a_\la\, h_{N,\lambda}\,d\omega \, d\rho  \, d\beta \, d\sigma' ,
\end{multline*}
where  $h_{N,\lambda}$ localizes to where $|\underline{\Delta x}|\gg \ve |\Delta x_1|.$ After integrations by parts in $\sigma'$, its contribution can be handled very much in the same way as we handled the one by  $E^{\varepsilon,2}_{\lambda,\tau'}$ before.

We shall thus continue with the main term $Q^{\varepsilon,3}_{\lambda,\tau'}.$

Let us first examine again  the term $\psi^\pm$  in the phase function $\Psi_ 4 ^\pm$ involving $\rho.$  If $|\beta|\ll 1,$ then
\[
|\partial_\rho\psi^{\pm}| = |2  \rho\psi_2 \pm  2\psi_1|\sim |\Delta x|. 
\]
  Thus, again by integration by parts in $\rho,$ we may again gain  any factors $ \jp{\lambda\Delta x}^{-N},$ and so also the contributions by the region where  $|\beta|\ll 1$ can be handled as an error term in a similar way as before.
 
 Thus, we are left with 
 \begin{equation*}
Q^{\varepsilon,4}_{\lambda,\tau'}=\lambda^d (\la |\Delta x|) \int_{\sigma' \sim 1} \int_{|\beta|\sim 1} \int_{\rho \sim 1} \jp{\lambda\Delta x}^{-\frac{d_1-1}2} \, e^{i\lambda \Psi_4^\pm} \, \tilde a_\la \,\overline{a_\lambda(s'\sigma')}  \, h_{4,\lambda}\, d\rho \, d\beta \, d\sigma'.
\end{equation*}

\subsubsection{Stationary phase for the integration in $\rho$}
Since the critical point of the  phase $\psi^\pm=\rho^2\, \psi_2\pm 2\rho\, \psi_1$ in $\rho$ is given by  $\rho_c = \mp \psi_1/\psi_2$,  observing that $\rho_c \sim \mp \beta^{-1}$, where $|\beta|\sim 1,$ we see that we may assume that $\psi^\pm$ has a critical point if and only if $\mp\beta \sim 1$.  

For the following discussion, we shall consider the phase  $\Psi_4^-$ (the cases of the phase  $\Psi_4^+$ can be treated analogously). 

If  $-\beta\sim 1,$ we can then again integrate by parts in $\rho$ and see that the corresponding contribution can be treated as an error term as before. Thus, we shall assume in the sequel that $\beta\sim 1,$ and that $\psi^-$ does have a critical point for any such $\beta.$ 

We can then  insert a suitable additional smooth cut-off functions $\chi_1(\beta)$  which localizes to  the region where 
$\beta\sim 1$ that we have devised, i.e., we may assume that 
 \begin{multline*}
Q^{\varepsilon,4}_{\lambda,\tau'}=\lambda^d (\la |\Delta x|) \int_{\sigma' \sim 1} \int_{\beta\sim 1} \int_{\rho \sim 1} \jp{\lambda\Delta x}^{-\frac{d_1-1}2} \, e^{i\lambda \Psi_4^-} \, \tilde a_\la \,\overline{a_\lambda(s'\sigma')}  \\ \times h_{4,\lambda}\, \chi_1(\beta) d\rho \, d\beta \, d\sigma'.
\end{multline*}
Using various trigonometric identities, we get
\begin{align*}
 \psi^-(\rho) |_{\rho = \rho_c}
& = - \frac{\psi_1^2}{\psi_2} = \frac{\left|\Delta x\right|^2 \tau' }{(\sin^2 \tau') \left(\cot\big((1 + \left|\Delta x\right|\beta) \tau' \big)-\cot \tau' \right)}\\
 & = - \left|\Delta x\right|^2  \tau' (\cot \tau' + \cot(\left|\Delta x\right|\beta \tau' )).
\end{align*}
 Since $\partial_\rho^2 \psi^-=2\psi_2\sim |\Delta x|,$ applying stationary phase in $\rho$ we get
\begin{equation}
Q^{\varepsilon,4}_{\lambda,\tau'}=\lambda^d (\la |\Delta x|) \int_{\sigma' \sim 1} \int_{|\beta|\sim 1} \jp{\lambda\Delta x}^{-\frac{d_1}2} \, e^{i\lambda \Psi_5} \, \tilde a_\la \,\overline{a_\lambda(s'\sigma')}  \, h_{5,\lambda}\,\chi_1(\beta)\, d\beta \, d\sigma',\label{Q5}
\end{equation}
where
\[
\Psi_5 = \sigma'\Big[\Delta s -\frac{s\left|\Delta x\right| \beta }{1 + \left|\Delta x\right|\beta} 
 - \left|\Delta x\right|^2 \tau' \cot(\left|\Delta x\right| \beta \tau' ) + 4\tau' \Delta w\Big].
\]

\subsubsection{Stationary phase for the integration in $\beta$}

Consider the function
\[
F(s,\beta,a,b) = \frac{s \beta }{1 + a\beta} + \frac {g(b \, \beta)}\beta,
\]
for $|a|,|b|\lesssim \ve\ll 1$ and $\beta \sim 1$, where  again $g(\tau) = \tau \cot \tau$. Then
\begin{equation}\label{Psi_5}
\Psi_5 = \sigma'\Big[\Delta s - \left|\Delta x\right|F\big(s,\beta,|\Delta x|,|\Delta x| \tau'\big) + 4 \tau' \Delta w\Big] .
\end{equation}
Note that for $a = |\Delta x|$ and $b = |\Delta x| \tau'$ the conditions $|a|,|b|\lesssim \ve\ll 1$  are indeed satisfied, since we are assuming that $|\Delta x|\lesssim \ve.$ 
If $a = b = 0$ and $s=1,$ then $F(1,\beta,0,0)=\beta+\frac 1\beta,$ so $F_\beta(1,\beta,0,0)=1-\frac 1{\beta^2}$ and $F_{\beta \beta}(1,\beta,0,0)=2\frac 1{\beta^3}.$ Then  $\beta_c(1,0,0):=1$ is the unique  critical point. Thus, by the implicit function theorem,  for $a,b$ and $s-1$  sufficiently small, $F$ has a unique critical point $\beta_c(s,a,b)$ which depends smoothly on $s,a,b.$ 

Let us  put $h(s,a,b)=F(s,\beta_c(s,a,b),a,b).$ Then $h$ is smooth, and we see  by \eqref{Psi_5} that if we let $a = |\Delta x|$ and $b = |\Delta x| \tau'$ and assume $\ve$ to be sufficiently small,  then
\begin{equation}
\Psi_6=\Psi_5\big|_{\beta = \beta_c(|\Delta x|,|\Delta x| \tau')}
 =\sigma'\Big[\Delta s-q(s,\Delta x,\Delta w,\tau')\Big], \label{eq:Psi5}
  \end{equation}
  where 
$$
  q(s,\Delta x,\Delta w,\tau')=|\Delta x|\, h(s,|\Delta x|,\,|\Delta x| \tau')- 4 \tau' \Delta w.
 $$ 
  Note that $\partial_{s}^j q = \mathcal O(\ve)$ for every $j\in\N,$ since $|\Delta x|,|\Delta w|\lesssim \ve.$
Hence, returning to the oscillatory integral in \eqref{Q5}, after applying the method of stationary phase in $\beta$, we obtain

\begin{equation*}
Q^{\varepsilon,4}_{\lambda,\tau'}=\lambda^d (\la |\Delta x|) \int_{\sigma' \sim 1} \jp{\lambda\Delta x}^{-\frac{d_1+1}2} \, e^{i\lambda \Psi_6} \, \tilde a_\la \,\overline{a_\lambda(s'\sigma')}  \, h_{6,\lambda}\, d\sigma'.\label{Q6}
\end{equation*}
Integrating $N$-times by parts in $\sigma',$ we see that the contribution by  $Q^{\varepsilon,4}_{\lambda,\tau'}$ gives our main term $Q^{\ve}_{1,\la}$ in \eqref{eq:pointwiseQ0}.  Note  also that the  asserted control of the constants $C(\ve,a_\la)$ in Proposition~\ref{lem:pointwiseQ0} is evident from our proof.

This completes the proof of Lemma~\ref{lem:pointwiseQ0-precise}, hence also that of Proposition~\ref{lem:pointwiseQ0} and Proposition~\ref{lem:0-L1}.
\qed

\subsection{$L^2$ estimates}
We proceed by proving Proposition~\ref{lem:0-L2}, that is, by showing  that the operator
\[
{T_{\lambda,x_1}^{0,\mathsf{nh},\varepsilon} (T_{\lambda,x_1'}^{0,\mathsf{nh},\varepsilon} )^*}:L^2\to L^2
\]
is  uniformly  bounded. Recalling \eqref{eq:frozen-Q}, Proposition~\ref{lem:0-L2} will then be a direct consequence of the following (see \cite{Hoermander-Acta1971},  \cite{GreenleafSeeger-Escorial}).

\begin{lemma}\label{lem:0-det}
The phase function $\Psi$ in \eqref{eq:Psi} satisfies
\[
\left|
\, \det\begin{pmatrix}
\Psi_{(\underline x, u,s),(\underline x', u',s')} & \Psi_{(\underline x, u,s),(\sigma,\sigma',\tau',y)}\\
\Psi_{(\sigma,\sigma',\tau',y),(\underline x', u',s')} & \Psi_{(\sigma,\sigma',\tau',y),(\sigma,\sigma',\tau',y)}\\
\end{pmatrix} 
\right|
\sim 1
\]
on the support of the oscillatory integral in \eqref{eq:Q0}.
\end{lemma}

\begin{proof}

Recall that the phase function $\Psi$ of Lemma~\ref{lem:0-TTstar} is given by
\begin{align*}
\Psi = \sigma s - \sigma's' & - \left|x-y\right|^2\sigma'\tau' \cot(\tfrac{\sigma'}\sigma \tau') \\ 
& + \left|x'-y\right|^2\sigma' \tau' \cot \tau' + 4 \sigma'\tau'(u-u' + \tfrac12\langle J(x-x'),y \rangle).
\end{align*}
Thus, the Monge-Ampère matrix
\[
\begin{pmatrix}
\Psi_{(\underline x, u,s),(\underline x', u',s')} & \Psi_{(\underline x, u,s),(\sigma,\sigma',\tau',y)}\\
\Psi_{(\sigma,\sigma',\tau',y),(\underline x', u',s')} & \Psi_{(\sigma,\sigma',\tau',y),(\sigma,\sigma',\tau',y)}\\
\end{pmatrix}
\]
is given by
\[
\begin{pmatrix}
0 & 0 & 0 & \Psi_{\underline x\sigma} & \Psi_{\underline x\sigma'} & \Psi_{\underline x\tau'} & \Psi_{\underline xy} \\
0 & 0 & 0 & 0 & 4 \tau' & 4 \sigma' & 0 \\
0 & 0 & 0 & 1 & 0 & 0 & 0 \\
0 & 0 & 0 & \Psi_{\sigma\sigma} & \Psi_{\sigma\sigma'} & \Psi_{\sigma\tau'} & \Psi_{\sigma y} \\
\Psi_{\sigma'\underline x'} & - 4 \tau' & -1 & \Psi_{\sigma'\sigma} & \Psi_{\sigma'\sigma'} & \Psi_{\sigma'\tau'} & \Psi_{\sigma' y} \\
\Psi_{\tau'\underline x'} & - 4 \sigma' & 0 & \Psi_{\tau'\sigma} & \Psi_{\tau'\sigma'} & \Psi_{\tau'\tau'} & \Psi_{\tau' y} \\
\Psi_{y\underline x'} & 0 & 0 & \Psi_{y\sigma} & \Psi_{y\sigma'} & \Psi_{y\tau'} & \Psi_{yy} 
\end{pmatrix}.
\]
Using first the $s$-row and the $s'$-column and then the $u'$-column for cofactor expansion, we see that the modulus of the determinant of the above matrix equals $4 \sigma'$ times the modulus of
\[
\begin{vmatrix}
0 & \Psi_{\underline x\sigma'} & \Psi_{\underline x\tau'} & \Psi_{\underline xy} \\
0 & 4 \tau' & 4 \sigma' & 0 \\
0 & \Psi_{\sigma\sigma'} & \Psi_{\sigma \tau'} & \Psi_{\sigma y} \\
\Psi_{y\underline x'} & \Psi_{y\sigma'} & \Psi_{y \tau'} & \Psi_{y y}
\end{vmatrix}
\]
Note that $\sigma' \, \Psi_{\underline x\sigma'} = 
\Psi_{\underline x\tau'} \, \tau'$ and $\sigma' \, \Psi_{\sigma\sigma'} = 
\Psi_{\sigma\tau'} \, \tau'$.
Thus, multiplying first the $\sigma'$-column by $\sigma'$, and subtracting then $\tau'$ times the $\tau'$-column from the $\sigma'$-column, we get that the determinant above is comparable to
\[
\begin{vmatrix}
0 & 0 & \Psi_{\underline x\tau'} & \Psi_{\underline xy} \\
0 & 0 & 4 \sigma' & 0 \\
0 & 0 & \Psi_{\sigma \tau'} & \Psi_{\sigma y} \\
\Psi_{y\underline x'} & \sigma' \, \Psi_{y\sigma'} -\Psi_{y \tau'} \, \tau'& \Psi_{y \tau'} & \Psi_{y y}
\end{vmatrix}.
\]
Using the $u$-row for cofactor expansion, the modulus of this determinant equals $ 4 \sigma'$ times the modulus of
\[
\begin{vmatrix}
0 & 0 & \Psi_{\underline xy} \\
0 & 0 & \Psi_{\sigma y} \\
\Psi_{y\underline x'} & \sigma' \, \Psi_{y\sigma'} -\Psi_{y \tau'} \, \tau' & \Psi_{y y}
\end{vmatrix}.
\]
 Let 
 \begin{equation}\label{htau}
h(\tau') = g(\tau')-\tau' g'(\tau')=(\tau')^2/(\sin \tau')^2,
\end{equation}  
and note that $h(\tau')\sim 1$ for $|\tau'|\le 5\pi/8.$   Then
\[
\sigma' \, \Psi_{y\sigma'} -\Psi_{y \tau'} \, \tau'
 = -2\sigma'\left(x'-y\right)h(\tau').
\]
Recall that $|x'_1-y_1|\sim 1$ and $|\underline x'-\underline y|\ll 1$. Thus, subtracting $\tfrac{x'_j-y_j}{x'_1-y_1}$ times the $y_1$-row from the $y_j$-row for all $j\in \{2,\dots,d_1\}$, the above determinant equals
\[
\begin{vmatrix}
0 & 0 & \Psi_{\underline xy} \\
0 & 0 & \Psi_{\sigma y} \\
\Psi_{y_1\underline x'} & -2\sigma'\left(x'_1-y_1\right)h(\tau') & \Psi_{y_1 y} \\
\big(\Psi_{y_j\underline x'} -\tfrac{x'_j-y_j}{x'_1-y_1}\Psi_{y_1\underline x'}\big)_{j\ge 2} & 0 & \big(\Psi_{y_j y} -\tfrac{x'_j-y_j}{x'_1-y_1}\Psi_{y_1 y}\big)_{j\ge 2}
\end{vmatrix}.
\]
Using the  middle column in this block matrix for cofactor expansion, the modulus of this determinant equals $2\sigma'\left|x'_1-y_1\right|h(\tau')$ times the modulus of 
\[
\begin{vmatrix}
0 & \Psi_{\underline xy} \\
0 & \Psi_{\sigma y} \\
\big(\Psi_{y_j\underline x'} -\tfrac{x'_j-y_j}{x'_1-y_1}\Psi_{y_1\underline x'}\big)_{j\ge 2} & \big(\Psi_{y_j y} -\tfrac{x'_j-y_j}{x'_1-y_1}\Psi_{y_1 y}\big)_{j\ge 2}
\end{vmatrix}.
\]
Note that $\Psi_{y_j y_k} = 0$ for $j\neq k$. Thus, the above determinant equals
\[
\begin{vmatrix}
0 & \Psi_{\underline xy_1} & \big(\Psi_{\underline xy_k}\big)_{k\ge 2}^\intercal \\
0 & \Psi_{\sigma y_1}& \big(\Psi_{\sigma y_k}\big)_{k\ge 2}^\intercal \\
\big(\Psi_{y_j\underline x'} -\tfrac{x'_j-y_j}{x'_1-y_1}\Psi_{y_1\underline x'}\big)_{j\ge 2} & \big(-\tfrac{x'_j-y_j}{x'_1-y_1}\Psi_{y_1 y_1}\big)_{j\ge 2} & \big(\Psi_{y_j y_k}\big)_{j,k\ge 2} 
\end{vmatrix}.
\]
since $\Psi_{y_1 y_k}=0$ for $k\ge2.$
Moreover, 
\[
\Psi_{\sigma y} = 2 \left(x-y\right)h(\tfrac{\sigma'}{\sigma}\tau').
\]
 
Thus, using that $|x_1-y_1|\sim 1$ and $|\underline x-\underline y|\ll 1$, subtracting $\tfrac{x_k-y_k}{x_1-y_1}$ times the $y_1$-column from the $y_k$-column for all $k\in \{2,\dots,d_1\}$, and using then cofactor expansion in the $\sigma$-row, the above determinant equals $2 \left|x_1-y_1\right|h(\tfrac{\sigma'}{\sigma}\tau')$ times
 \[
\begin{vmatrix}
0 & \big(\Psi_{\underline x y_k} - \tfrac{x_k-y_k}{x_1-y_1} \Psi_{\underline xy_1}\big)_{k\ge 2}^\intercal \\
\big(\Psi_{y_j\underline x'} -\tfrac{x'_j-y_j}{x'_1-y_1}\Psi_{y_1\underline x'}\big)_{j\ge 2} & \big(\Psi_{y_j y_k} + \tfrac{x'_j-y_j}{x'_1-y_1}\tfrac{x_k-y_k}{x_1-y_1}\Psi_{y_1 y_1}\big)_{j,k\ge 2}
\end{vmatrix}.
\]
Note that all these matrices are $(d_1-1)\times (d_1-1)$ matrices. Let $A$ be the matrix in the lower left corner and $B$ upper right corner.

Then the modulus of the determinant of the matrix above equals $\left|\det A\right|\left|\det B\right|$. 
Note that
\[
\Psi_{\underline y \underline x'} = -2\sigma' \big(g(\tau')I_{d_1-1} + \tau'\underline{J}\big)
\quad\text{and}\quad
\Psi_{\underline x \underline y}
=2\sigma g\big(\tfrac{\sigma'}{\sigma}\tau'\big)I_{d_1-1}
-2\sigma'\tau'\underline{J},
\]
where the matrix $\underline{J}$ is skew-symmetric. Since $|\tau'|\le \frac 3 8 \pi$ on the support of $\eta_{0,\mathsf{nh}}$ and $\sigma'/\sigma$ is sufficiently close to $1$, we have $g(\tau')\sim 1$ and $g(\tfrac{\sigma'}{\sigma}\tau')\sim 1$. Thus,
\[
| \Psi_{\underline y \underline x'} v| \sim |v| \quad\text{and}\quad 
| \Psi_{\underline x \underline y} v| \sim |v| \quad\text{for all } v\in\R^{d_1-1}. 
\]
Since $|\frac{x'_j-y_j}{x'_1-y_1}|\ll 1$ and $|\frac{x_j-y_j}{x_1-y_1}|\ll1$ for $j\ge2$, we get
\[
|Av|\sim |v| \quad\text{and}\quad |Bv|\sim |v| \quad \text{for all } v\in\R^{d_1-1}. 
\]
In particular, both $A$ and $B$ are invertible with determinant comparable to 1.
\end{proof}

\begin{remark}\label{breakdown}
If $\tau'=\frac{\pi}{2},$ then $g(\tau')=0.$ Moreover, the matrix 
$
\underline J=(\begin{smallmatrix}
0 & 0 \\
0 &  \underline{\underline J} \\
\end{smallmatrix})
$
(where $\smash{\underline{\underline J}}=J_{n-1}$) is degenerate.  Thus the last step of the proof would break down if we were to apply the same arguments to the horizontal part of $T^{0,\ve}_\la$. 

In analogy to what we shall do later  for the horizontal part of the operator $\breve T_{\lambda}^{k,1,\varepsilon} $ in Proposition~\ref{lem:k1-L2}, we should  not freeze the variable $x_1$ here,  but $x_2.$
\end{remark}

\section{The horizontal part}\label{sec:k=0-hor}

Our goal in this section is to prove Proposition~\ref{L:T0lambda} with $T^{0,\mathsf{h}}_{\lambda,t}$ in place of $T^0_{\lambda,t}$.  In conjunction with Section \ref{sec:k=0-nonhor}, this completes the proof of Proposition~\ref{L:T0lambda}.

For the horizontal part $T_{\lambda}^{0,\mathsf{h},\varepsilon},$ we have to replace the kernel $K_{\lambda,s}^{0}$ in \eqref{eq:Klats-formula}
 by 
 \[
K_{\lambda,s}^{0,\mathsf{h}}(x, u) := \lambda^{n+\frac32}  \int_\R \, \int_0^\infty e^{i\lambda\tilde\Phi(x,u,s,\sigma,\tau)} \, a_\lambda(s\sigma) \, \eta_{0,\mathsf{h}}( \tau)  \, d\sigma \, d\tau,
\]
where $\supp\eta_{0,\mathsf{h}}\subset [\tfrac 3{16} \pi, \tfrac 58 \pi],$   with phase as in \eqref{eq:phasesdef}, that is,
\[
\tilde \Phi(x,u,s,\sigma,\tau) =\sigma \big(s-\tau (|x|^2 \cot(\tau) - 4u)\big).
\]

For the pointwise estimates, one easily checks that the arguments in the proof of Proposition~\ref{lem:pointwiseQ0} remain valid for $|\tau|,|\tau'|\le\tfrac{5}{8}\pi$, so the proposition also holds for the kernel $Q^{0,\mathsf{h},\varepsilon}_\lambda$ in place of $Q^{0,\mathsf{nh},\varepsilon}_\lambda$.

\smallskip

On the other hand, for the $L^2$ estimates, shifting  coordinates $\tau=\tilde \tau +\tfrac \pi 2,$ so that $\tilde \tau\in [-\tfrac 5{16} \pi, \tfrac 18 \pi]\subset[-\tfrac 38 \pi,\tfrac 38\pi],$ we arrive at the phase
\begin{equation}\label{eq:brevePhi-0}
\breve\Phi
 = \sigma\big(s+(\tfrac \pi 2+\tilde\tau )(|x|^2 \tan(\tilde\tau) + 4u)\big).
\end{equation}
Up to a translation and a constant rescaling in the $u$-variable, this phase agrees with the phase
\[
\sigma \, \tilde \Phi^{-} = \sigma\Big[\delta^{-1}-1+s +\left(k\pi -\tfrac12\pi+\tilde\tau\right) \Big(\delta |x|^2 \tan(\tilde\tau)+\frac {u}{k} +4\delta^{-1} u_c(\tau_{k,1})\Big)\Big]
\]
in the oscillatory integral \eqref{eq:kernel-kpm}, if we set $k=1$ and hence $\delta=1$.
Thus, freezing the variable $x_2$, the $L^2$ estimates from Proposition~\ref{lem:k1-L2} include the estimate for $T_{\lambda,x_2}^{0,\mathsf{h},\varepsilon} (T_{\lambda,x_2'}^{0,\mathsf{h},\varepsilon} )^*$ as a special case when $\delta=1$.

Alternatively, for the $L^2$ estimate, we may also follow the arguments of the proof of Lemma~\ref{lem:0-det}. In what follows we set $\tilde x=(x_1,x_3,\dots, x_{d_1})$ and $\tilde x'=(x_1',x_3',\dots, x_{d_1}')$, while keeping the previous notation for $\underline y=(y_2,\dots, y_{d_1})$. For the $L^2$ estimate for $T_{\lambda,x_2}^{0,\mathsf{h},\varepsilon} (T_{\lambda,x_2'}^{0,\mathsf{h},\varepsilon})^*$ we need to consider the phase function $\Psi$ in \eqref{eq:Psi}, but with $x_2$, $x_2'$ frozen, that is,
\begin{multline*}
\Psi(\tilde x,u,s, \tilde x',u',s', \sigma,\sigma',\tau',y)  = \sigma s - \sigma's'  - \left|x-y\right|^2\sigma'\tau' \cot(\tfrac{\sigma'}\sigma \tau') \\ 
 + \left|x'-y\right|^2\sigma' \tau' \cot \tau' + 4\sigma'\tau'(u-u' + \tfrac12\langle J(x-x'),y \rangle).
\end{multline*}
We need to verify 
\[
\left| \,
\det
\begin{pmatrix}
\Psi_{(\tilde x, u,s),(\tilde x', u',s')} & \Psi_{(\tilde x, u,s),(\sigma,\sigma',\tau',y)}\\
\Psi_{(\sigma,\sigma',\tau',y),(\tilde x', u',s')} & \Psi_{(\sigma,\sigma',\tau',y),(\sigma,\sigma',\tau',y)}\\
\end{pmatrix}
\right|
\sim 1
\]
The same matrix and determinant manipulations as in the previous section, with $\underline x, \underline x'$ replaced by $\tilde x, \tilde x'$, reduce matters to verifying that, for $\tau'\in\supp\eta_{0,\mathsf{h}}$, the modulus of the determinant
\[
\begin{vmatrix}
0 & \big(\Psi_{\tilde x y_k} - \tfrac{x_k-y_k}{x_1-y_1} \Psi_{\tilde x y_1}\big)_{k\ge 2}^\intercal \\
\big(\Psi_{y_j\tilde x'} -\tfrac{x'_j-y_j}{x'_1-y_1}\Psi_{y_1\tilde x'}\big)_{j\ge 2} & \big(\Psi_{y_j y_k} + \tfrac{x'_j-y_j}{x'_1-y_1}\tfrac{x_k-y_k}{x_1-y_1}\Psi_{y_1 y_1}\big)_{j,k\ge 2}
\end{vmatrix}
\]
is bounded below. 

Since, as before, $|x_k-y_k|\ll|x_1-y_1|$ and $|x_k'-y_k|\ll|x_1'-y_1|$ for $k\geq 2$, it suffices to verify that the $(d_1-1)\times (d_1-1)$ matrices $\Psi_{\underline y\tilde x'}$ and $\Psi_{\tilde x \underline y}$ are uniformly invertible.
We compute 
\[
\Psi_{\underline y\tilde x'} = -2\sigma'\big( g(\tau') A + \tau' B\big),
\]
where
\[
A = \begin{pmatrix} 0&0\\0&I_{d_1-2}\end{pmatrix}
\quad\text{and}\quad
B= \begin{pmatrix} 1&0\\0& \tilde J \end{pmatrix},
\]
and where $\tilde J=J_{n-1}$ is a $(d_1-2)\times (d_1-2)$ invertible skew-symmetric matrix. Clearly $B$ is invertible and since in our present case $|\sigma'g(\tau')|\ll 1$ we get that $\Psi_{\underline y\tilde x'} $ is also invertible. The same argument shows that $\Psi_{ \tilde x \underline y }$ is also invertible. 

Altogether, in either way, we see that the ``frozen'' operators $T_{\lambda,x_2}^{0,\mathsf{h},\varepsilon}$ satisfy the analogous estimates of Propositions \ref{lem:0-L1} and \ref{lem:0-L2}, only with $x_1, x'_1$ replaced by  $x_2, x'_2$. Following from thereon the previous arguments for the non-horizontal part, we thus complete the proof of Proposition~\ref{L:T0lambda}.

\section{The case $k\ge 1$: Preliminary reductions}\label{sec:kneq0-preliminaries}

In this section, we first treat the range $2^\ell k\gg\lambda$ using estimates from \cite{MuellerSeeger2015}. In remaining the range $2^\ell k\lesssim \lambda$, we reduce the proof of Theorem~\ref{thm:main-square} to the square function estimate \eqref{kl-est} in Proposition~\ref{prop:conditional}.

Recall from \eqref{opdecomp} that we have the decomposition 
\[
\beta\big(\tfrac{t\sqrt L}{\la}\big) e^{it\sqrt L }f
= \widetilde \beta\big(\tfrac{t^2L}{\la^2}) \, T^0_{\la,t} f
+ \sum_{0<|k|<8\la } \sum_{\ell\ge 1} \widetilde \beta\big(\tfrac{t^2L}{\la^2}) \, T^{k,\ell}_{\la,t}f+ \rho_\la\big(\tfrac{t^2L}{\la^2}\big)f.
\]
We shall now turn to  estimating the contributions  of the operators  $T^{k,\ell}_{\la,t}f=f*\sK^{k,\ell}_{\la,t}$ to \eqref{eq:main-square} in Theorem~\ref{thm:main-square}. In view of Remark \ref{rem:k-negative}, we may and shall assume that $k>0$ in what follows. 

\subsection{The contributions by the  $T^{k,\ell}_{\la,t}$ with $2^\ell k \gg\lambda$}

We consider estimates for the case  $2^\ell k\gg\lambda$, with $k\le 8\la,$ for which we are not expected to use curvature and oscillation in a significant way.

In this subsection, we establish estimates dual to the square function estimates in Theorem~\ref{thm:main-square}. Note that the conjugate exponent of $\frac{2d_1}{d_1+1}$ is $\frac{2d_1}{d_1-1}$.

\begin{proposition}\label{prop:near}
For $1\le p < \frac{2d_1}{d_1 + 1}$ and all $\ell\ge 1$,
\[
\Big\| \sum_{k\in\mathcal I_{\lambda,\ell}} 
\int_1^2 T^{k,\ell}_{\la,t}f_t \, dt \Big\|_p\\
\lc  2^{-\frac{\ell}4} \la^{d(\frac 1p-\frac 12)-\frac 12}
\Big\| \Big(\int_1^2 | f_t|^2 \, dt\Big)^{\frac{1}{2}} \Big\|_p,
\]
where $ \mathcal I_{\lambda,\ell} : = \{k > 0 : 10^5 \lambda \le 2^\ell k \text{ and }k\le 8\lambda\}$.
\end{proposition}

\begin{proof}
We interpolate between $L^1$ and $L^2$. For the $L^1$ bounds, we use Lemmas 8.5 and 8.6 of \cite{MuellerSeeger2015} (see also Eq.\ (100) there), which yield
 \[\|\la^{-\frac{d-1}{2}} T_{\lambda,t}^{k,\ell} \|_{L^1\to L^1} \lc \la^{-\frac{d_1-1}{2} } (2^\ell k)^{-1} \quad\text{for }k\in
 \mathcal I_{\lambda,\ell}.
 \]
Thus, for $k\in\mathcal I_{\lambda,\ell}$, 
 \[ \Big\|\int_1^2 T^{k,\ell}_{\la,t} f_t \, dt\Big\|_1\lc \la ^{\frac{1}{2}} (2^\ell k)^{-1} \Big\|\int_1^2 |f_t| \, dt\Big\|_1,\]
 and therefore, for $K\le \la$,
 \[
 \Big\| \sum_{k\sim K} \Big|\int_1^2 T^{k,\ell}_{\la,t} f_t \, dt\Big| \, \Big\|_1 \lc \la^{\frac{1}{2}} 2^{-\ell} \, 
 \Big\|\int_1^2 |f_t| \, dt\Big\|_1.\]
 Note that $10^5 2^{-\ell} \lambda \le k \le 8\lambda$ for $k\in\mathcal I_{\lambda,\ell}$. 
 Hence
 \begin{equation}\label{L1}
 \Big\| \sum_{k\in\mathcal I_{\lambda,\ell}} \Big|\int_1^2 T^{k,\ell}_{\la,t} f_t \, dt\Big| \, \Big\|_1 \lc 
 C(\lambda,\ell) \, 
 \Big\|\int_1^2 |f_t| \, dt\Big\|_1 ,
\end{equation}
with constant
\[
C(\lambda,\ell) = \begin{cases}\la^{\frac{1}{2}} 2^{-\ell} (\ell + 1)
 &\text{ if } 2^\ell\le \la,
 \\
 \la^{\frac{1}{2}} \log(\la + 1) 2^{-\ell} 
 &\text{ if } 2^\ell\ge \la.
\end{cases}
\]
We proceed with the $L^2$ estimates. Note that 
\begin{align*}T^{k,\ell}_{\la,t}& = \la^{\frac{1}{2}}\int_0^\infty  e^{i\frac{\la}{ 4 \varsigma} } \, a_{\lambda,\circ}(\varsigma) \, \eta_\ell\Big(\frac{\varsigma t^2}{\la} ( -iU )-k\pi\Big) \, e^{i\varsigma t^2 L/\la} \, d\varsigma
\\& = \la^{\frac{1}{2}} \int_0^\infty e^{i\frac{\la t^2}{ 4 s}} \, a_{ \la, \circ}(st^{-2}) \, \eta_\ell(\la^{-1} s(-iU)-k\pi) \, e^{isL/\la} \, t^{-2} \, ds.
\label{s-integral}
\end{align*} 
Thus,
\[ 
\int_1^2 T^{k,\ell}_{\la,t}f_t \, dt = \int\eta_\ell(\la^{-1} s(-iU)-k\pi) \, e^{isL/\la} \, G_s \, ds,
\]
where
\[G_s = 
\la^{\frac{1}{2}} \int e^{i\frac{\la t^2}{ 4 s}} \, a_{\lambda,\circ}(st^{-2}) \, t^{-2} \, f_t \, dt.\]
Note that $|\frac{\partial^2}{\partial s\partial t} (t^2s^{-1}/ 4 ) |\sim 1$ for $s,t\sim 1$. Thus, by standard oscillatory $L^2$ integral estimates, we have
\begin{equation}\label{hoerm} \Big(\int_J | G_s(x,u)|^2 \, ds\Big)^{\frac{1}{2}} \lc \Big(\int_1 ^2 |f_t(x,u)|^2 \, dt\Big)^{\frac{1}{2}}. \end{equation}
Note that the $s$ integral is extended over a compact interval $J\subset [\frac 1 8, 32]$.
Using Plancherel  in the central variable $u,$ we get
\begin{align*}
&\Big\| \sum_{k\in\mathcal I_{\lambda,\ell}} 
\int_1^2 T^{k,\ell}_{\la,t}f_t \, dt \Big\|_2 
 =\Big( \int_{\R^{d_1}}\int_{\R} \Big| \sum_{k\in\mathcal I_{\lambda,\ell}} \int_1^2 (T^{k,\ell}_{\la,t}f_t)^\mu(x) \, dt \Big|^2 d\mu \, dx\Big)^{\frac{1}{2}} .
\end{align*} 
Note that for any $x\in \R^{d_1},$ 
\[
\int_1^2 (T^{k,\ell}_{\la,t}f_t)^\mu(x) \, dt = \int_J \eta_\ell( 2\pi s\tfrac{\mu}{\lambda}-k\pi) (e^{isL/\la} \, G_s)^\mu(x) \, ds.
\]
Since $2^\ell\ge 10^5 \la /k\ge 10^5/8\gg1,$ we have $2^{-\ell}\ll 1.$ For fixed $\mu,$ the condition $\eta_\ell(2\pi s\tfrac{\mu}{\lambda}-k\pi) \neq 0$ implies that $|2\pi s\mu-\la k\pi|\le \la 2^{1-\ell},$ and since $1/8\le s\le 32,$ this implies that 
$$
1\le k\sim \frac{|\mu|}\la.
$$
Moreover, since  $|s-\frac {\la} {\mu} \frac k 2 |\le \frac\la{|\mu|} 2^{-\ell},$ where $2^{-\ell}\ll \frac{1}{2},$ we see that for fixed $\mu$ and $x,$ the functions 
$$
g^{\mu,x}_k(s) =\eta_\ell(2\pi s\tfrac{\mu}{\lambda}-k\pi) (e^{isL/\la} \, G_s)^\mu(x)
$$
are  supported in pairwise disjoint intervals of length $\frac\la{|\mu|} 2^{-\ell}.$ And, since $k\sim |\mu|/\la,$ we see that
$\sum_{k\in\mathcal I_{\lambda,\ell}} g^{\mu,x}_k$ is supported in a set of total volume bounded by a constant times 
$(\frac{|\mu|}{\la})\, \frac\la{|\mu|} 2^{-\ell}\sim 2^{-\ell}.$ Thus, by Cauchy--Schwarz,
\begin{multline*}
 \big|\sum_{k\in\mathcal I_{\lambda,\ell}}\int_1^2(T^{k,\ell}_{\la,t}f_t)^\mu(x) \, dt\Big|
 =\Big| \int_J\sum_{k\in\mathcal I_{\lambda,\ell}}g^{\mu,x}_k(s) \, ds\Big|\\
 \lesssim 2^{-\frac \ell 2}\Big(\int_J \big|\sum_{k\in\mathcal I_{\lambda,\ell}}g^{\mu,x}_k(s)\big|^2 \, ds\Big)^{\frac 12}
 \lesssim 2^{-\frac \ell 2}\Big(\int_J \sum_{k\in\mathcal I_{\lambda,\ell}}\big| g^{\mu,x}_k(s)\big|^2 \, ds\Big)^{\frac 12}.
\end{multline*}
We thus obtain
\begin{align*}
&\Big\| \sum_{k\in\mathcal I_{\lambda,\ell}} 
\int_1^2 T^{k,\ell}_{\la,t}f_t \, dt \Big\|_2 \\
&\lesssim 2^{-\frac \ell 2}
\Big( \int_{\R^{d_1}}\int_{\R} \sum_{k\in\mathcal I_{\lambda,\ell}}  \int_J \Big|\eta_\ell(2\pi s\tfrac{\mu}{\lambda}-k\pi) (e^{isL/\la} \, G_s)^\mu(x) \Big|^2 \, ds \,d\mu \, dx\Big)^{\frac{1}{2}}\\
&\lesssim 2^{-\frac \ell 2} \Big( \int_{\R^{d_1}}\int_{\R}   \int_J \Big|(e^{isL/\la} \, G_s)^\mu(x) \Big|^2 \, ds \,d\mu \, dx\Big)^{\frac{1}{2}},
\end{align*} 
since $\sum_{k\in\mathcal I_{\lambda,\ell}} |\eta_\ell(2\pi s\tfrac{\mu}{\lambda}-k\pi)|^2 \lesssim 1.$
Hence, again by Plancherel, since $e^{isL/\la}$ is a unitary operator on $L^2,$ 
\begin{equation}\label{L2} 
\Big\| \sum_{k\in\mathcal I_{\lambda,\ell}} 
\int_1^2 T^{k,\ell}_{\la,t}f_t \, dt \Big\|_2 \lesssim 2^{-\frac \ell 2}\Big\|\Big(\int_J |G_s|^2 \,ds\Big)^{\frac 12}\Big\|_2
\lesssim 2^{-\frac \ell 2}\Big\|\Big(\int_1^2 |f_t|^2 \,dt\Big)^{\frac 12}\Big\|_2,
\end{equation} 
where we have applied \eqref{hoerm} in the last inequality.

Finally, for $1\le p\le 2$, by interpolation between \eqref{L1} and \eqref{L2}, we get
\begin{align*}\Big\| \sum_{k\in\mathcal I_{\lambda,\ell}} 
\int_1^2 T^{k,\ell}_{\la,t}f_t \, dt \Big\|_p
& \lc \la^{\frac 1p-\frac 12} \, 2^{-\frac \ell p} 
 \\
& \times \max \big\{(\ell + 1)^{\frac{1}{p}} , (\log \la)^{\frac{1}{p}}\big\} \, \Big\| \Big(\int_1^2 | f_t|^p \, dt\Big)^{\frac{1}{p}} \Big\|_p .
\end{align*}
Now 
\[
\Big(\int_1^2 | f_t|^p \, dt\Big)^{\frac{1}{p}}\le
\Big(\int_1^2 | f_t|^2 \, dt\Big)^{\frac{1}{2}}
\]
and since 
$ \tfrac 1p-\tfrac 12 < d\,(\frac 1p-\frac 12)-\tfrac 12$ {if and only if} $p< \tfrac{2d_1}{d_1+1}$
we have established  Proposition~\ref{prop:near}. 
\end{proof}

Note that the estimates in Proposition~\ref{prop:near} can be summed over all $\ell>0$, so that the contributions of these $k$ and $\ell$ to the left-hand side of \eqref{eq:main-square} in Theorem~\ref{thm:main-square} are fine.

\subsection{Support reduction of the convolution kernel of $T_{\lambda,t}^{k,\ell}$}\label{decay}

We are left with the contributions of the operators $T_{\lambda,t}^{k,\ell}$ with $0<k\le 8\la$ and $2^\ell k\lesssim \la.$ By Proposition~\ref{lem:decomp}, for $k\neq 0$ and $\ell \ge 1$, the convolution kernel of $T_{\lambda,t}^{k,\ell}$ is given by
\[
\sK_{\lambda,t}^{k,\ell}(x, u) = \lambda^{ n + \frac{3}{2}} \, t^{-2n} \int_0^\infty \int_0^\infty  e^{i\lambda\Phi(x,u,t,\sigma,\tau)} \, a_\lambda(t^2\sigma) \, \eta_\ell( \tau -k \pi)\left(\frac{\tau}{\sin\tau}\right)^{n} \, d\sigma \, d\tau,
\]
where $a_\lambda$ is supported in $(\frac{1}{16}, 4 )$, and
\[
\Phi(x,u,t,\sigma,\tau) = \sigma( t^2-|x|^2 g(\tau) + 4u\tau), \text{ with } g(\tau) = \tau\cot \tau.
\]
As in the case $k=0$, in order to simplify the notation slightly, we will drop the factor $t^{-2n}$, replace $t$ by $$s := t^2\in[1,4],$$ and  work from here on with the convolution kernel
\begin{equation}\label{eq:K-kl-alt}
K_{\lambda,s}^{k,\ell}(x, u) = \lambda^{n + \frac 32} \int_0^\infty \int_0^\infty e^{i\lambda\tilde \Phi(x,u,s,\sigma,\tau)} \, a_\lambda(s\sigma) \, \eta_\ell(\tau -k \pi)\Big(\frac{\tau}{\sin\tau}\Big)^n d\sigma \, d\tau
\end{equation}
in place of $\sK_{\lambda,t}^{k,\ell}$,
where, as in \eqref{eq:phasesdef},
\begin{equation}\label{eq:phase-kl-alt}
\tilde \Phi(x,u,s,\sigma,\tau) = \sigma(s-|x|^2 g(\tau) + 4u\tau).
\end{equation}
Correspondingly,  we shall estimate the operator $\tilde T^{k,\ell}_{\la,s}$ defined by 
$$\tilde T^{k,\ell}_{\la,s} f = f*K_{\lambda,s}^{k,\ell}$$
 in place of $T_{\lambda,t}^{k,\ell}.$
Note that 
\begin{equation}\label{TtildeT}
\Big\| \Big( \int_1^2 \big| T_{\lambda,t}^{k,\ell} f \big|^2 \, dt \Big)^{\frac{1}{2}} \Big\|_p
\sim \Big\| \Big( \int_1^4 \big| \tilde{T}_{\lambda,s}^{k,\ell} f \big|^2 \, ds \Big)^{\frac{1}{2}} \Big\|_p.
\end{equation}

Assume now that $k\ge 1$ and $\ell\ge 1.$ In order to prepare our decay estimates for fixed $k$ and $\ell$, we compute the critical points of the phase function $\tilde\Phi$. We have
\[
\nabla_{\sigma,\tau}^{\mathsf T} \tilde \Phi (x,u,s,\sigma,\tau) = \binom{s-|x|^2 g(\tau) + 4 \tau u}{\sigma (-|x|^2 g'(\tau) + 4 u)}
 = A_{\sigma,\tau}
\begin{pmatrix}
 |x|^2 \\ u
\end{pmatrix}
 + \begin{pmatrix}
 s \\ 0
\end{pmatrix},
\]
where
\[
A_{\sigma,\tau} : = \begin{pmatrix}
 -g(\tau) & 4 \tau \\
 -\sigma g'(\tau) & 4 \sigma
\end{pmatrix}.
\]
Note that
\begin{equation}\label{Ainvers}
A_{\sigma,\tau}^{-1}=
\frac 1 {\sigma\tau^2}
\begin{pmatrix}
 -\sigma \sin^2\tau & \tau \sin^2\tau \\
 \frac{1}{4}\sigma(\tau-\sin\tau\cos\tau) &  \frac{1}{4}\tau\sin\tau\cos\tau
\end{pmatrix}.
\end{equation}
We have $g'(\tau) = \cot\tau - \tau/\sin^2\tau$. Thus, $\nabla_{\sigma,\tau} \tilde \Phi(x,u,s,\sigma,\tau) = 0$ if and only if
\begin{align*}
\begin{pmatrix}
 |x|^2 \\ u
\end{pmatrix}
 = 
A_{\sigma,\tau}^{-1}
\begin{pmatrix}
 -s \\ 0
\end{pmatrix}
 = 
s\begin{pmatrix}
 \sin^2\tau/\tau^2 \\ 
 - \frac{1}{4}(\tau-\sin\tau\cos\tau)/\tau^2
\end{pmatrix}.
\end{align*}
Recall that the factor $\eta_\ell(\tau-k \pi)$ localizes to the union  $J_{k,\ell}$ of two intervals on which $|\tau-k \pi| \sim 2^{-\ell}$. 
Given $\tau\in J_{k,\ell}$, we write
\begin{equation}\label{eq:formulas-sing}
|x|_c(s,\tau) := \sqrt{s} \, \Big|\frac{\sin\tau}{\tau}\Big| \quad\text{and}\quad
  u_c(s,\tau) := - s \, \frac{\tau-\sin\tau\cos\tau}{ 4 \tau^2}.
\end{equation}
Moreover, we write
\[
\Delta x^2 (s,\tau): = |x|^2-|x|_{c}^2 (s,\tau) \quad\text{and}\quad \Delta u(s,\tau): = u-u_c(s,\tau).
\]
Then
\begin{align}\notag
 \nabla_{\sigma,\tau}^{\mathsf T} \tilde \Phi (x,u,s,\sigma,\tau)
& =  A_{\sigma, \tau}\binom{|x|^2-|x|_{c}^2 (s,\tau)}{u-u_c(s,\tau)} \\
& =  A_{\sigma, \tau}\binom{\Delta x^2 (s,\tau)}{\Delta u(s,\tau)}. \label{nablaPhi}
\end{align}
We fix any point $\tau_{k,\ell} \in J_{k,\ell}$. For $\tau \in J_{k,\ell}$, we have 
\[
|x|_{c}(s,\tau) \sim 2^{-\ell}/k \quad\text{and}\quad |u_c(s,\tau)| \sim 1/k.
\]
Moreover,
\[
 \left|u_c(s,\tau)-u_c (s,\tau_{k,\ell}) \right| \lesssim \frac{2^{-\ell}}{k^2} \quad\text{and}\quad 
\left| |x|_c^2 (s,\tau)-|x|_c^2 (s,\tau_{k,\ell}) \right| \lesssim \Big(\frac{2^{-\ell}}{k}\Big)^2.
\]
This suggest the following to hold:  $K_{\lambda,s}^{k,\ell}$  is essentially supported  in a ``hollow cylinder''
of the form 
$$
||x|-|x|_c(s,\tau_{k,\ell})|  \lesssim \frac{2^{-\ell}}{k} \quad \text{and}\quad |u-u_c (s,\tau_{k,\ell})| \lesssim \frac{2^{-\ell}}{k^2}.
$$
 We will make this more precise in the next proposition, whose proof will be given in  Appendix \ref{decayproof}.

\begin{proposition}\label{prop:decay-away}
There is some constant $R_0 > 0$ such that for any $s \in[1,4]$ the following holds:
Let $k\ge 1$ and $\ell\ge 1$, and let $\tau_{k,\ell}$ be in the support of $\eta_\ell(\tau-k \pi).$ Denote by $E_s^{k,\ell}$ the set of all $(x,u)\in \R^{d_1}\times \R$ such that 
\[
|x| \leq 10^{-1}\frac{2^{-\ell}}{k}, \quad\text{or}\quad 
|x| \ge \sqrt{10 }\, \frac{2^{-\ell}}{k}, \quad\text{or}\quad 
|u-u_c (s,\tau_{k,\ell})|  \ge R_0  \frac{2^{-\ell}}{k^2}.
\]
Then, for any $N\in\N,$ 
\[
\int_{E_s^{k,\ell}} |K_{\lambda,s}^{k,\ell}(x,u)| \, d(x,u) \le C_N
\Big(\frac{\lambda}{2^\ell k}\Big)^{-N}\Big(\frac{2^{\ell}}{k}\Big)^{\frac{1}{2}}.
\]
\end{proposition}
Let us put  $u_c(\tau)= u_c(1,\tau),$ so that $u_c(s,\tau)=s\,u_c(\tau).$

\begin{corollary}\label{decaysmalls}
There is some constant $R_0 > 0$ such that for any $s \in[1,1+\frac{2^{-\ell}}k]$ the following holds: If 
$E^{k,\ell}$ denotes the set of all $(x,u)\in \R^{d_1}\times \R$ such that 
\[
|x| \leq 10^{-1}\frac{2^{-\ell}}{k}, \quad\text{or}\quad 
|x| \ge \sqrt{10 }\, \frac{2^{-\ell}}{k}, \quad\text{or}\quad 
|u-u_c (\tau_{k,\ell})|  \ge R_0  \frac{2^{-\ell}}{k^2},
\]
then, for any $N\in\N,$ 
\[
\int_{E^{k,\ell}} |K_{\lambda,s}^{k,\ell}(x,u)| \, d(x,u) \le C_N
\Big(\frac{\lambda}{2^\ell k}\Big)^{-N}\Big(\frac{2^{\ell}}{k}\Big)^{\frac{1}{2}}.
\]
\end{corollary}
\begin{proof} 
Since $|u_c(\tau)| \sim 1/k,$ we see that for  $s \in[1,1+\frac{2^{-\ell}}k]$ we have  $|u_c(s,\tau)- u_c(\tau)|\lesssim 2^{-\ell}k^{-1}  k^{-1}= 2^{-\ell}k^{-2},$ so that, for some constant $C>0, $
$$
|u-u_c (s,\tau_{k,\ell})|\ge |u-u_c (\tau_{k,\ell})|-C\frac { 2^{-\ell}}{k^2}.
$$
The asserted  estimate follows thus from the one in Proposition~\ref{prop:decay-away}, if we increase the constant $R_0$ in a suitable way.
\end{proof}

\subsection{Reduction  to shorter  time intervals}\label{shorts}

Assume again that  $k, \ell\ge 1$ are such that $ k\le 8\la$ and $2^\ell k\lesssim \la.$
For $k$ and  $\ell$ fixed,   we shall use  in the sequel the abbreviation
\[
\delta = \delta(k,\ell) := 2^{1-\ell}/k.
\]

In view of Corollary~\ref{decaysmalls}, we shall next reduce the  square function estimates for  $\tilde T^{k,\ell}_{\la,s}$ to estimates over shorter $s$-intervals of length at most $\delta.$ 

\begin{proposition}\label{prop:conditional}
Assume there is some $\ve\in (0,1)$ such that for all
 $p\ge \frac{2(d_1+1)}{d_1-1}$  and all $k,\ell$ as above the following estimate 
\begin{multline}
\Big\| \Big( \int_1^{1+\ve\delta(k,\ell)} \big| \tilde{T}_{\lambda,s}^{k,\ell} f \big|^2 \, ds \Big)^{\frac{1}{2}} \Big\|_p
\le C(\ve,a_\la,p)\, \delta(k,\ell)^{\frac 12}\\
\times\Big[\lambda^{d(\frac 12 - \frac 1p )-\frac 12} \, (2^\ell k)^{-d(\frac 12 - \frac 1p ) + \frac1p } + \Big(\frac{2^{\ell}}{k}\Big)^{\frac{1}{2}}
\Big]\,
\|f\|_p \label{kl-est}
\end{multline}
holds for all $f\in \mathcal S(\R^d),$ where the constant  $C(\ve,a_\lambda, p)$ depends in $a_\la$ only on the $C^M$-norm of $a_\lambda$ for some sufficiently large $M\in\N.$ 

Then for $p\ge \frac{2(d_1+2)}{d_1-1}$ and $\alpha > d \, \big(\frac{1}{2}-\frac{1}{p}\big)-\frac{1}{2},$ or $p> \frac{2(d_1+2)}{d_1-1}$ and $\alpha \ge d \, \big(\frac{1}{2}-\frac{1}{p}\big)-\frac{1}{2},$ the following estimate
\begin{equation}\label{kl-estg}
\Big\| \Big( \int_1^2 \big| \sum\limits_{1\le k \le 8\la, \,2^{\ell}k\lesssim \la} T_{\lambda,t}^{k,\ell} f \big|^2 \, dt \Big)^{\frac{1}{2}} \Big\|_p
\le  C_\ve\, C(\ve,a_\la,p)\, \la^\alpha\, \|f\|_p
\end{equation}
holds true, which completes the proof of Theorem~\ref{thm:main-square}.
\end{proposition}

\begin{proof} As a first step, we decompose the time interval $[1,2]$ in \eqref{kl-estg} into about $1/\ve$ intervals of length $\sim \ve.$ Exploiting homogeneity, by means of automorphic scalings, we may then reduce each of the corresponding integrals in $t$ to the interval $[1,\sqrt{1+\ve}],$ which corresponds to the $s$-interval $[1,1+\ve].$

Given $k$ and $\ell,$ we next decompose the $s$-interval $[1,1+\ve]$ into disjoint subintervals
\[
I_j=[s_j,s_{j+1}],\qquad s_1=1,\qquad s_{N+1}=1+\varepsilon,
\]
with $\delta=\delta(k,\ell),$ where $s_{j+1}/s_j=1+\varepsilon\delta$ for $j=1,\dots,N-1$ and $s_{N+1}/s_N\le 1+\varepsilon\delta$. Then $s_j\in[1,1+\varepsilon]$ for all $j\in\{1,\dots,N\}$. Moreover, $s_j^{-1}I_j\subset[1,1+\varepsilon\delta]$ and $|I_j|\lesssim \varepsilon\delta$. Note that $N\sim \delta^{-1}$ with a constant that depends on $\varepsilon$.

Using  Minkowski's inequality, we obtain
\begin{align*}
\Bigl\|\Bigl(\int_{1}^{1+\varepsilon}\bigl|\tilde T^{k,\ell}_{\lambda,s} f\bigr|^2\,ds\Bigr)^{\frac12}\Bigr\|_{p}
&=
\Bigl\|\Bigl(\sum_{j=1}^{N}\int_{I_j}\bigl|\tilde T^{k,\ell}_{\lambda,s} f\bigr|^2\,ds\Bigr)^{\frac12}\Bigr\|_{p}\\
&\le
\biggl(\sum_{j=1}^{N}\Bigl\|\Bigl(\int_{I_j}\bigl|\tilde T^{k,\ell}_{\lambda,s} f\bigr|^2\,ds\Bigr)^{\frac12}\Bigr\|_{p}^{2}\biggr)^{\frac12}.
\end{align*}
For a given $j,$ making the change of variables $s=s_j\tilde s$ and putting $r_j=s_{j+1}/s_j\le 1+\varepsilon\delta$, we have
\[
\Bigl(\int_{I_j}\bigl|\tilde T^{k,\ell}_{\lambda,s} f\bigr|^2\,ds\Bigr)^{\frac12}
=
s_j^{\frac12}\Bigl(\int_{1}^{r_j}\bigl|\tilde T^{k,\ell}_{\lambda,s_j\tilde s}f\bigr|^2\,d\tilde s\Bigr)^{\frac12}.
\]
Since $s_j\in[1,1+\varepsilon]$, the factor $s_j^{\frac{1}{2}}\sim 1$ is harmless. The phase function $\tilde\Phi$ is homogeneous of degree $2$ with respect to the parabolic rescaling $(s,x,u)\mapsto (\alpha^2 s,\alpha x,\alpha^2 u)$. Combining the automorphic rescaling $D_{\sqrt{s_j}}$ with the change of variables $\tilde\sigma=s_j\sigma$, we obtain
\[
K_{\lambda,s_j\tilde s}^{k,\ell}
\bigl(D_{\sqrt{s_j}}(x,u)\bigr)
=
s_j^{-1}K_{\lambda,\tilde s}^{k,\ell}(x,u).
\]
Since $s_j\in[1,1+\varepsilon]$, all resulting scaling factors are
harmless. Extending the $\tilde s$-integral from $[1,r_j]$ to
$[1,1+\varepsilon\delta]$, and using
$N^{\frac12}\sim\delta^{-\frac12}$, we obtain
\[
\|T^{k,\ell}_{\lambda,t}\|_{L^p\to L^p(L^2([1,2]))}
\le C_\ve\,
\delta^{-\frac12}
\|\tilde T^{k,\ell}_{\lambda,s}\|_
{L^p\to L^p(L^2([1,1+\varepsilon\delta]))}.
\]
Hence, by \eqref{kl-est},
\[
\|T^{k,\ell}_{\lambda,t}\|_{L^p\to L^p(L^2([1,2]))}
\ \le C_\ve \,C(\ve,a_\la,p)\, 
\Big[\lambda^{d(\frac 12 - \frac 1p )-\frac 12} \, (2^\ell k)^{-d(\frac 12 - \frac 1p ) + \frac1p } + \Big(\frac{2^{\ell}}{k}\Big)^{\frac{1}{2}}\Big].
\]

Let us denote by $p_1=\frac{2(d_1+1)}{d_1-1}=\frac{2d}{d-2}$ the Stein--Tomas exponent associated with $d_1.$ Then 
$d \, (\frac{1}{2}-\frac{1}{p_{1}})=1,$ and thus if  $p\ge p_1$ and $\alpha > d \, (\frac{1}{2}-\frac{1}{p})-\frac{1}{2},$ then $\alpha>\frac{1}{2}.$
Thus, for $2^\ell k\lesssim \la,$  we may estimate
\[
\Big(\frac{2^{\ell}}{k}\Big)^{\frac{1}{2}} \lesssim \Big(\frac{2^{\ell}}{k}\Big)^{\frac{1}{2}} (2^\ell k)^{-\alpha}\la^{\alpha}=2^{-\ell(\alpha-\frac 12)} k^{-(\alpha+\frac 12)}\la^{\alpha},
\]
so we can bound the sum  of the second terms  $(\frac{2^{\ell}}{k})^{\frac{1}{2}}$ over all $k$ and $\ell$ with $2^\ell k\lesssim \la$  by a constant times $\la^\alpha.$  Similarly, if $p>p_1$ and $\alpha \ge d \, (\frac{1}{2}-\frac{1}{p})-\frac{1}{2},$ then again $\alpha>\frac{1}{2},$ and the summation in $k$ and $\ell$ leads to the same bound of order $\la^\alpha$.  

Under the stronger assumption $p> \frac{2(d_1+2)}{d_1-1}=\frac{2(d+1)}{d-2}$ and $\alpha \ge d \, (\frac{1}{2}-\frac{1}{p})-\frac{1}{2},$ we even have  $-d\,(\frac 12 - \frac 1p ) + \frac1p < -1$, so we are able to sum the first term in $k$ and $\ell$ since
\[
\lambda^{d(\frac 12 - \frac 1p )-\frac 12} \, (2^\ell k)^{-d(\frac 12 - \frac 1p)+\frac 1p}\le \lambda^{\alpha}\, (2^\ell k)^{-d(\frac 12 - \frac 1p)+\frac 1p}.
\]
Similarly, if $p\ge \frac{2(d_1+2)}{d_1-1}$ and $\alpha > d \, (\frac{1}{2}-\frac{1}{p})-\frac{1}{2},$ then 
\begin{align*}
\lambda^{d(\frac 12 - \frac 1p )-\frac 12} \, (2^\ell k)^{-d(\frac 12 - \frac 1p)+\frac 1p}
& \le \lambda^{\alpha}\la^{-\ve_p}\, (2^\ell k)^{-d(\frac 12 - \frac 1p)+\frac 1p} \\
& \le \la^\alpha (2^\ell k)^{-d(\frac 12 - \frac 1p)+\frac 1p-\ve_p}
\end{align*}
for some $\ve_p>0,$ so we can again sum in $k$ and $\ell$.

Thus, \eqref{kl-estg} follows by applying Minkowski's inequality.
\end{proof} 

\section{Reductions to unit scales} \label{main-est-kl}

In this section, we localize the kernels $K_{\lambda,s}^{k,\ell}$ to the regions $D^{k,\ell}$ below and reduce the square function estimate \eqref{kl-est} of Proposition~\ref{prop:conditional} to the rescaled square function estimate \eqref{eq:Tkl-breve} in Proposition~\ref{prop:unit}.

Assume again that $2^\ell k\lesssim \la.$
For  $k,\ell$ fixed, let us denote by $D^{k,\ell}$ the complement of the set 
$E^{k,\ell}$ from Corollary~\ref{decaysmalls}, which is given by
\begin{equation}\label{eq:B-kell}
D^{k,\ell}=\Big\{(x,u):10^{-1}\frac{2^{-\ell}}{k} <|x|< \sqrt{10}\,\frac{2^{-\ell}}{k}\quad \text{and} \quad |u-u_c (\tau_{k,\ell})| <  R_0 \frac{2^{-\ell}}{k^2}\Big\},
\end{equation}
and let $\dot{K}_{\lambda,s}^{k,\ell}=\chr_{D^{k,\ell}} \,K_{\lambda,s}^{k,\ell}$  denote the according  localization of 
$K_{\lambda,s}^{k,\ell},$ with associated convolution operator $ \dot T_{\lambda,s}^{k,\ell} f=f*\dot{K}_{\lambda,s}^{k,\ell}.$
Then decompose 
\begin{equation}\label{TdotTR}
\tilde T_{\lambda,s}^{k,\ell}=\dot {T}_{\lambda,s}^{k,\ell}+R_{\lambda,s}^{k,\ell},
\end{equation}
where 
$R_{\lambda,s}^{k,\ell}f=f*\big(\chr_{E^{k,\ell}} \,K_{\lambda,s}^{k,\ell}\big).$ 
By  Corollary~\ref{decaysmalls}, we have that for $s\in[1,1+\frac{2^{-\ell}}{k}]$ and $2^{\ell}k\lesssim \la$
\[
\|\chr_{E^{k,\ell}} \,K_{\lambda,s}^{k,\ell}\|_{L^1} \,\le C \Big(\frac{2^{\ell}}{k}\Big)^{\frac{1}{2}}.
\]
Then  Minkowski's inequality  shows that the contribution by the operator  $R_{\lambda,s}^{k,\ell}$  to the square function estimate  \eqref{kl-est} in Proposition~\ref{prop:conditional} adds  the second term $\delta(k,\ell)^{\frac 12}(\frac{2^{\ell}}{k})^{\frac{1}{2}}.$

We are thus left with estimating the contributions by the operators $\dot{T}_{\lambda,s}^{k,\ell}.$ 
The goal of this  and the next two sections  is to prove the following

\begin{proposition}\label{prop:inter-Lp}
If $\ve>0$ is sufficiently small and $p\ge \frac{2(d_1+1)}{d_1-1}$, then, for all $k\ge 1$, $\ell \ge 1$,
\begin{multline}
\Big\| \Big( \int_1^{1+\ve\delta(k,\ell)} \big| \dot{T}_{\lambda,s}^{k,\ell} f \big|^2 \, ds \Big)^{\frac{1}{2}} \Big\|_p
\le C(\ve,a_\la,p)\, \delta(k,\ell)^{\frac 12}\\
\times\Big[\lambda^{d(\frac 12 - \frac 1p )-\frac 12} \, (2^\ell k)^{-d(\frac 12 - \frac 1p ) + \frac1p } + \Big(\frac{2^{\ell}}{k}\Big)^{\frac{1}{2}} \Big]
\, \|f\|_p \label{eq:inter-Lp}
\end{multline}
for all $f\in \mathcal S(\R^d),$ where the constant  $C(\ve,a_\lambda, p)$ depends in $a_\la$ only on the $C^M$-norm of $a_\lambda$ for some sufficiently large $M\in\N.$
\end{proposition}

By Proposition~\ref{prop:conditional}, this will then complete the proof of  Theorem~\ref{thm:main-square}.

\subsection{Secondary  localizations for the kernels $K_{\lambda,s}^{k,\ell}$}

We fix again a smooth bump function $\chi$ supported in the Euclidean ball $B_1(0)\subset \R^{d_1}\times \R$ such that $\chi\equiv 1$ on $B_{\frac{1}{2}}(0).$ For $\ve>0,$ we  set
\[
\chi^{k,\ell}_\ve(x,u)=\chi\Big(\frac{k 2^\ell}{\ve}x, \frac{k^2 2^\ell}{\ve}u\Big).
\]

\begin{proposition}\label{prop:localization-ii}
Let $\ve>0$ be sufficiently small and $p\ge \frac{2(d_1+1)}{d_1-1}$. Suppose that, for every $k,\ell\ge 1$, every $(x^0,u^0)\in\bbH_n$ such that
$x^0=(|x^0|,0,\dots,0)$ and
\begin{equation}\label{essentsupp}
100^{-1}\frac{2^{-\ell}}{k} \le |x^0| \le 100\,\frac{2^{-\ell}}{k}
\quad\text{and}\quad
|u^0-u_c(\tau_{k,\ell})| \le 10R_0\frac{2^{-\ell}}{k^2},
\end{equation}
and every $f\in\mathcal S(\bbH_n)$, we have the estimate
\begin{multline}\label{loc-kl}
\Big\|\Big(
\int_1^{1+\ve\delta(k,\ell)}\big|\tilde\chi^{k,\ell}_\ve\,\tilde T_{\lambda,s}^{k,\ell}(\chi^{k,\ell}_\ve f)
\big|^2\,ds\Big)^{\frac12}\Big\|_p
\le C(a_\lambda,\ve,p)\, \delta(k,\ell)^{\frac12} \\
\times \lambda^{d(\frac12-\frac1p)-\frac12}
(2^\ell k)^{-d(\frac12-\frac1p)+\frac1p}
\|f\|_p,
\end{multline}
where $\tilde\chi^{k,\ell}_\ve(x,u):=\chi^{k,\ell}_\ve(x-x^0,u-u^0)$, and where the constant $C(a_\lambda,\ve,p)$ is uniform in
$k,\ell$ and $(x^0,u^0)$ and depends on $a_\lambda$ only through its $C^M$-norm for some sufficiently large $M\in\N$.
Then the square function estimate \eqref{eq:inter-Lp} holds.
\end{proposition}

\begin{proof}
This is again an easy consequence of Proposition~\ref{localize-a}. Assume that \eqref{loc-kl} holds true. As in \eqref{TdotTR}, we decompose again
\[
\tilde T_{\lambda,s}^{k,\ell}=\dot {T}_{\lambda,s}^{k,\ell}+R_{\lambda,s}^{k,\ell}.
\]
Since $\dot {T}_{\lambda,s}^{k,\ell}= \tilde T_{\lambda,s}^{k,\ell}-R_{\lambda,s}^{k,\ell}$,  Corollary~\ref{decaysmalls} and \eqref{loc-kl} imply that 
\begin{multline*}
\Big\| \Big( \int_1^{1+\ve\delta(k,\ell)} \big| \tilde \chi^{k,\ell}_\ve\, \dot{T}_{\lambda,s}^{k,\ell} (\chi^{k,\ell}_\ve \,f) \big|^2 \, ds \Big)^{\frac{1}{2}} \Big\|_p
\le
C(a_\la,\ve,p) \,\delta(k,\ell)^{\frac 12}\\
\times \Big[\lambda^{d(\frac 12 - \frac 1p )-\frac 12} \, (2^\ell k)^{-d(\frac 12 - \frac 1p ) + \frac1p } + \Big(\frac{2^{\ell}}{k}\Big)^{\frac{1}{2}}\Big]\, \|f\|_p.
\end{multline*}

Note first that if we fix a sufficiently large constant  $R\gg 1$,  then for $k=1$ the set $D^{1,\ell}$ is contained in the Euclidean ball $B_{R 2^{-\ell}}(\smash{z_{1,\ell}^0})$ centered at the point $\smash{z_{1,\ell}^0=(0,u_c (\tau_{1,\ell}))}$, which lies in the center of the Heisenberg group $G=\bbH_n$. Since the Euclidean distance and the left-invariant Riemannian distance $d_G$ are equivalent on bounded sets, we may assume that $D^{1,\ell}\subset B(\smash{z_{1,\ell}^0},R2^{-\ell})$ after increasing $R$ if necessary, where $B(z,r)$ denotes the ball of radius $r$ centered at $z$ with respect to the distance $d_G$.

For general $k\ge 1,$  let $\alpha=D_{k^{-1}}$ be the automorphic dilation given by $\alpha(x,u)=(k^{-1}x, k^{-2} u).$ Then we see analogously that $D^{k,\ell}\subset \alpha\big(B(\smash{z_{k,\ell}^0},R2^{-\ell})\big),$ where the point
$\smash{z_{k,\ell}^0}=(0,k^2u_c (\tau_{k,\ell}))$ is again central in $\bbH_n.$ Thus, by Proposition~\ref{localize-a} (ii), in order to prove \eqref{eq:inter-Lp}, it suffices to show this only for $f$ supported in $\alpha(B(0, \ve 2^{-\ell})).$ After adjusting $\ve$ by a fixed constant, it therefore suffices to assume that $f$ is supported 
in 
\[
B^{k,\ell}_\ve=\Big\{(x,u): |x|\le \ve \frac{2^{-\ell}}{k}, |u|\le \ve \frac{2^{-\ell}}{k^2}\Big\}.
\]
But then, if $\ve $ is chosen sufficiently small, it is easy to see that $\dot T_{\lambda,s}^{k,\ell} f=f*\dot{K}_{\lambda,s}^{k,\ell}$ is supported in the group product set $B^{k,\ell}_\ve D^{k,\ell},$ which is contained in
\[
\tilde D^{k,\ell}=\Big\{(x,u):\frac 1{50}\frac{2^{-\ell}}{k} <|x|\le 50\frac{2^{-\ell}}{k}\quad \text{and} \quad |u-u_c (\tau_{k,\ell})| <  5 R_0  \frac{2^{-\ell}}{k^2}\Big\}.
\] 
The latter is clear for $k=1,$ and follows for general $k\ge 1$ again by scaling with the dilation $D_{k^{-1}}$.

Since we can cover the set $\tilde D^{k,\ell}$ by $\mathcal O(\eps^{-d})$ Euclidean translates of the set $B^{k,\ell}_\ve,$ we see that the estimate in \eqref{loc-kl} will imply estimate \eqref{eq:inter-Lp}, where for the given $\ve $ the constant in \eqref{eq:inter-Lp} will be of the form $C(\ve,a_\la,p) =C_\ve\, C(a_\la,\ve,p)$. Arguing as in the proof of Proposition~\ref{prop:localization}, after rotation we may assume that $x_0=(|x_0|,0,\dots,0)$.
\end{proof} 

In analogy with the  operator $T^{0,\ve}_\la,$ we shall here define the operator $T^{k,\ell,\ve}_{\la}$ by
\begin{equation}\label{Tkle}
(T^{k,\ell,\ve}_{\la} f)(x,u,s) : = \tilde \chi^{k,\ell}_\ve \big(\tilde T^{k,\ell}_{\la,s}( \chi^{k,\ell}_\ve f)\big)(x,u),\quad f\in \mathcal S(\R^d),
\end{equation}
where we recall that $\tilde T^{k,\ell}_{\la,s} f = f*K_{\lambda,s}^{k,\ell}$.

\smallskip

Arguing in a similar way as for the case $k=0,$  by means of  further localizations we may then reduce the proof of Proposition~\ref{prop:localization-ii} to showing the following:

\medskip

{\it 
Suppose that   $p\ge \frac{2(d_1+1)}{d_1-1}$.
If the constant $\ve>0$ is  sufficiently small, and if $a_\la$ is supported in an interval of length $\ve$ contained in 
$(\frac 1{16},4),$ then we can bound 
\begin{multline}\label{eq:Tkllambdas}
\Big\| \Big( \int_1^{1+\ve\delta(k,\ell)} \big| T^{k,\ell,\ve}_{\la} f \big|^2 \, ds \Big)^{\frac{1}{2}} \Big\|_p
\le C(a_\lambda,\eps,p)\, \delta(k,\ell)^{\frac 12}\,\lambda^{d(\frac 12 - \frac 1p )-\frac 12} \\
\times (2^\ell k)^{-d(\frac 12 - \frac 1p ) + \frac1p }\,\|f\|_p 
\end{multline}
for all  $f\in \mathcal S(\R^d),$ where the constant  $C(a_\lambda,\eps,p)$ depends in $a_\la$ only on the $C^M$-norm of $a_\lambda$ for some sufficiently large $M\in\N.$}

\smallskip

To prove \eqref{eq:Tkllambdas}, we shall normalize the amplitude functions and pass to  unit scales where all variables are of unit size by means of a couple of coordinate changes. We first concentrate on the convolution kernel $K_{\lambda,s}^{k,\ell}$ before reinstalling the cut-offs in \eqref{Tkle}.
Recall from \eqref{eq:K-kl-alt} that
\[
K_{\lambda,s}^{k,\ell}(x, u) = \lambda^{n + \frac 32} \int_0^\infty \int_0^\infty e^{i\lambda\tilde \Phi(x,u,s,\sigma,\tau)} \, a_\lambda(s\sigma) \, \eta_\ell(\tau -k \pi)\Big(\frac{\tau}{\sin\tau}\Big)^n \, d\sigma \, d\tau,
\]
where
\[
\tilde \Phi(x,u,s,\sigma,\tau) = \sigma(s-|x|^2 g(\tau) + 4u\tau).
\]
Recalling \eqref{eq:eta-ell-def}, we see that if $\eta_\ell(\tau -k \pi)\neq 0,$ then
\begin{equation}\label{supptau}
\tau\in \big[(k-2^{-\ell}\, \tfrac 54)\pi,(k-2^{-\ell}\,\tfrac 38)\pi] \cup [(k+2^{-\ell}\, \tfrac 38)\pi,(k+2^{-\ell}\,\tfrac 54)\pi\big].
\end{equation}
First, changing coordinates via $\tau = k\pi + 2^{1-\ell} \, \tilde \tau$, where $\tilde \tau\in\R$ with $|\tilde \tau|\sim 1$, and observing that
\[
g(k\pi + 2^{1-\ell} \, \tilde \tau) = (k\pi + 2^{1-\ell} \, \tilde \tau)\cot( 2^{1-\ell} \, \tilde \tau),
\]
we see that the above expression equals
\[
\lambda^{n + \frac 32} \, 2^{1-\ell}\int_0^\infty \int_\R \, e^{i \lambda \sigma \tilde \Phi_1} \, a_\lambda(s\sigma) \, \eta_1(\tilde \tau) \, \Big(\frac{k\pi + 2^{1-\ell} \, \tilde \tau}{(-1)^k \sin (2^{1-\ell} \, \tilde \tau)}\Big)^{n} 
 \, d\tilde \tau \, d\sigma,
\]
where the phase function  $\tilde \Phi_1$ is given by
\[
\tilde \Phi_1 = s-\left(k\pi + 2^{1-\ell} \, \tilde \tau\right) \big( |x|^2 \cot( 2^{1-\ell} \, \tilde \tau) - 4 u\big).
\]
Expanding $(k\pi + 2^{1-\ell} \, \tilde \tau)^{n}$, we see that in place of $K_{\lambda,s}^{k,\ell}$, it is sufficient to bound convolution operators with kernels of the form
\begin{equation}\label{eq:kl-tilde}
\tilde K_{\lambda,s}^{k,\ell}(x,u)=\lambda^{n + \frac 32} \, k^{n} \, 2^{-\ell}\int_0^\infty \int_\R \, e^{i \lambda \sigma\tilde \Phi_1} \, a_\lambda(s\sigma) \, \tilde \eta_1(\tilde \tau)   \sin (2^{1-\ell} \, \tilde \tau)^{-n} \, d \tilde \tau \, d\sigma,
\end{equation}
where $\tilde\eta_1$ is an adapted smooth cutoff satisfying the same support conditions as $\eta_1$. In what follows, we drop the tilde and write $\eta_1$ in place of $\tilde\eta_1$ for simplicity.

\subsection{Translation in $u$ and rescaling in space}

Next, we change coordinates  $x$ and $u$ by passing to new coordinates $\breve x$ and  $\breve u$ via
\[
x = \frac {2^{1-\ell}}{k} \breve x,\quad 4 u = 4 u_c(\tau_{k,\ell}) + \frac {2^{1-\ell}}{k^2}   \breve u ,
\]
where $u_c(\tau_{k,\ell})$ is defined as in \eqref{eq:formulas-sing}, and, recalling \eqref{supptau}, we choose
\[
\tau_{k,\ell}: = (k+ 2^{-\ell-1})\pi\in \supp \eta_\ell.
\]
Recall that $u_c(\tau_{k,\ell})\sim 1/k$.  Then, if $(x,u)$ lies in the set $D^{k,\ell}$ defined in \eqref{eq:B-kell}, we have $|\breve x| \sim 1$ and $|\breve u|\lesssim 1$. We define
\begin{equation}\label{eq:kl-double-tilde}
\tilde{\tilde K}_{\lambda, s}^{k,\ell} (\breve x,\breve u) := 2^{(1-\ell) d} k^{-(d+1)} \tilde K_{\lambda,s}^{k,\ell} \Big(\frac {2^{1-\ell}}{k} \breve x, u_c(\tau_{k,\ell}) + \frac {2^{1-\ell}}{4 k^2}\breve u\Big),
\end{equation}
Up to the harmless fixed rescaling of the central variable by a factor of $4$, the above change of coordinates is the  composition of a translation by a central element in $\mathbb H_n,$  a parabolic scaling by $D_{1/k}$ and an isotropic scaling by $r=2^{1-\ell}$.

Now, if $K$  is any  convolution kernel on $\mathbb H_n,$  we first observe that the norms of the operators of right-convolution with $K,$ and with any translate of $K$ by a central element of $\mathbb H_n,$ are the same on $L^p.$ Similarly, since $D_{1/k}$  is an automorphism of  $\mathbb H_n,$ the same holds true for the right convolution with $K,$ and 
the right-convolution with the rescaled kernel $k^{-(d+1)} K\circ D_{1/k}$.

For the isotropic scaling $(x,u)\mapsto (rx,ru)$, the same is not true, however. We here need to pass to a group $\mathbb H_{n,r}$ which is only isomorphic to $\mathbb H_n$: The underlying manifold of $\mathbb H_{n,r}$ is again $\R^{2n}\times \R,$ but the group law in $\mathbb H_{n,r}$ is given by 
\[
(x,u)\cdot(x',u')= (x+x', u+u'+\tfrac 12 (rJx)^\intercal  x'),
\]
i.e., we replace the skew matrix $J$ from \eqref{eq:grouplaw} by $rJ$. The mapping $(x,u)\mapsto (rx,ru)$ defines then an isomorphism from $\mathbb H_{n,r}$  onto $\mathbb H_n.$ Correspondingly, if $f\ast_r g$ denotes the convolution of two functions $f$ and $g$ on $\mathbb H_{n,r}$, one easily computes that
\[
f(r\, \cdot)\ast_r  (r^d K(r\,\cdot))=(f\ast K)(r\,\cdot).
\]
This implies that the operator of right-convolution with $K$ on $\mathbb H_n$  and the operator of right-convolution with 
$r^d K(r\,\cdot)$ on $\mathbb H_{n,r}$ have the same norm on $L^p$.

Altogether, the convolution operators $f\mapsto f * \tilde K_{\lambda,s}^{k,\ell}$ on $\mathbb H_n$  and $f\mapsto f *_r \tilde{\tilde K}_{\lambda,s}^{k,\ell}$ on $\mathbb H_{n,r}$ have the same $L^p$ operator norms, up to the harmless fixed rescaling of the central variable by a factor of $4$.

\subsection{Rescaling in time} Fix $k,\ell\ge 1$ and put
\[
\delta=\delta(k,\ell)=2^{1-\ell}k^{-1}.
\]
Since in \eqref{eq:Tkllambdas} we assume that $s\in [1,1+\ve\delta],$ we finally  rescale the  variable $s$  by setting $s = 1+\delta(\breve s-1)$, so that $\breve s\in[1,1+\varepsilon]$.
For the square function estimate for the operator $T^{k,\ell,\ve}_{\la} $ this scaling leads to a gain by the factor $\delta^{\frac 12},$ i.e., 
\begin{equation}\label{snorm}
\Bigl(\int_{1}^{1+\varepsilon\delta}\bigl|(T^{k,\ell,\ve}_{\la} f)(\cdot,s)\bigr|^2\,ds\Bigr)^{\frac12}
=
\delta^{\frac12}\Bigl(\int_{1}^{1+\varepsilon}\bigl (T^{k,\ell,\ve}_{\la} f)(\cdot,1+\delta(\breve s-1))\big|^2\,d\breve s\Bigr)^{\frac12}.
\end{equation}
The same gain will appear for the kernels $\tilde K_{\lambda,s}^{k,\ell}$ and $\tilde{\tilde K}_{\lambda, s}^{k,\ell}$ when they are used in place of $K_{\lambda, s}^{k,\ell}$ herein.

Gathering all the factors from \eqref{eq:kl-tilde} and \eqref{eq:kl-double-tilde}, we observe that
\begin{equation}\label{eq:factors}
\lambda^{n + \frac 32} \, k^{n} \, 2^{-\ell} \, 2^{(1-\ell)d} \, k^{-(d+1)} = 2^{n-1} \,\delta^{\frac{1}{2}} \, 2^{-n\ell}\,\lambda_*^{n + \frac 32},
\end{equation}
where $\lambda_* := \lambda 2^{1-\ell}k^{-1}=\lambda \delta\gtrsim 1$.

Thus, to summarize, in the new coordinates we are led to consider convolution kernels of the form
\begin{equation}\label{eq:kernel-breve}
\breve K_{\lambda_*,\breve s}^{k,\ell}(\breve x,\breve u) := \lambda_*^{n + \frac 32}\, 2^{-n\ell} \int_0^\infty \int_\R e^{i \lambda_* \sigma \breve \Phi_0} \,\tilde a_{\la,\de}(\sigma,\breve s) \, \eta_1(\tilde \tau)   \sin (2^{1-\ell} \, \tilde \tau)^{-n} \, d \tilde \tau \, d\sigma,
\end{equation}
with phase
\[
\breve \Phi_0= \delta^{-1}-1+\breve s -\left(k\pi + 2^{1-\ell} \, \tilde \tau\right)\Big(\delta |\breve x|^2 \cot( 2^{1-\ell} \, \tilde \tau) -  
\frac {\breve u}{k} -4\delta^{-1} u_c(\tau_{k,\ell})\Big),
\]
where  $\breve s$ is assumed to be in $[1,1+\ve]$, and
\begin{equation}\label{atilde}
\tilde a_{\la,\de}(\sigma,\breve s)=a_{\lambda}\big((1+\delta(\breve s-1))\sigma\big).
\end{equation}
Note also that  by \eqref{eq:factors} the kernel $ \tilde{\tilde K}_{\lambda,1+\delta(\breve s-1)}^{k,\ell} $ corresponds to the kernel $ \delta^{\frac 12} \breve K_{\lambda_*,\breve s}^{k,\ell}$ in the rescaled coordinates.

\smallskip

Now, similar to \eqref{Tkle}, we define the operators $\breve T^{k,\ell,\ve}_{\la_*}$ by
\begin{equation}\label{Tkle-breve}
(\breve T^{k,\ell,\ve}_{\la_*} f)(\breve x,\breve u,\breve s) : = \tilde \chi_\ve \big(( \chi_\ve f) \ast_{2^{1-\ell}} \breve K_{\lambda_*,\breve s}^{k,\ell}\big)(\breve x,\breve u),\quad f\in \mathcal S(\R^d),
\end{equation}
where $\chi_\ve:=\chi^{1,0}_\ve$ is supported in the Euclidean ball of radius $\ve>0$ centered at the origin, and where $\tilde \chi_\ve(\breve x,\breve u):=\chi_\ve(\breve x-x^0,\breve u-u^0)$.  Here,  we may again assume that $x_0=(|x_0|,0,\dots,0)$ with $|x_0|\sim 1$ and $|u_0|\lesssim 1.$
\smallskip

 To prove \eqref{eq:Tkllambdas}, we will establish the following square function   estimates: 

\begin{proposition}\label{prop:unit}
Suppose that $p\ge \frac{2d}{d-2}$.
If the constant $\ve>0$ is  sufficiently small, then, for all $k\ge 1$, $\ell\ge 1$, we can bound
\begin{equation}\label{eq:Tkl-breve}
\Big\| \Big( \int_1^{1+\ve} \big| \breve T^{k,\ell,\ve}_{\la_*} f \big|^2 \, d\breve s \Big)^{\frac{1}{2}} \Big\|_p
\le C(a_\lambda,\eps,p)\, \delta^{-\frac 1p} \, \lambda_*^{d(\frac 12 - \frac 1p )-\frac 12}\,\|f\|_2
\end{equation}
for all  $f\in \mathcal S(\R^d)$, where the constant  $C(a_\lambda,\eps, p)$ depends in $a_\la$ only on the $C^M$-norm of $a_\lambda$ for some sufficiently large $M\in\N$.
\end{proposition}

Note that the estimate \eqref{eq:Tkl-breve} is stronger than \eqref{eq:Tkllambdas}, since the right-hand side involves $\|f\|_2$ rather than $\|f\|_p$. By Hölder’s inequality, exploiting the localization in \eqref{Tkle-breve}, we can immediately pass from $\|f\|_2$ to $\|f\|_p$ in this estimate, after absorbing a constant depending only on $\ve$. Moreover, note that if \eqref{eq:Tkl-breve} holds, then
\begin{align*}
  \Big\|\Big( \int_1^{1+\varepsilon} \big|\tilde \chi_\ve \big(( \chi_\ve f) & \ast_{2^{1-\ell}} \tilde{\tilde K}_{\lambda,1+\delta(\breve s-1)}^{k,\ell}\big)\big|^2 \, d\breve s\Big)^{\frac 12} \Big\|_p \\
 & \lesssim \delta^{\frac 12} \, \|\breve T^{k,\ell,\ve}_{\la_*}\|_{L^p\to L^p(L^2([1,1+\varepsilon]))} \, \|f\|_p \\
& \lesssim \delta^{\frac 12}  \, \delta^{-\frac 1p} \, \lambda_*^{d(\frac 12 - \frac 1p )-\frac 12} \, \|f\|_p \\
 & \lesssim  \lambda^{d(\frac 12 - \frac 1p )-\frac 12} \, (2^\ell k)^{-d(\frac 12 - \frac 1p )+\frac 1p} \, \|f\|_p.
\end{align*}

Our previous discussion, keeping \eqref{snorm} in mind, shows that this estimate, in turn, implies the estimate \eqref{eq:Tkllambdas}. The remaining part  of this and the following two sections will therefore be devoted to the proof of Proposition~\ref{prop:unit}. To that end, with a slight abuse of notation, we shall rename ${\lambda_*}$ back to $\lambda,$ and $(\breve x,\breve u,\breve s)$ back to $(x,u,s)$. From here on, we shall have to distinguish between Case~I where $\ell\ge 2,$ and Case~II where $\ell=1,$ since these cases  exhibit a quite different geometric behavior, as explained in the Introduction. 

\section{The key estimates for the non-horizontal parts given by $\ell\ge 2$}\label{CaseI}

As in the case $k=0$, we apply the Stein-Tomas $TT^*$ argument, so that the square function \eqref{eq:Tkl-breve} of Proposition~\ref{prop:unit} will be a consequence of the estimates in Propositions~\ref{lem:kl-L1} and~\ref{lem:kl-L2}.

Recall from \eqref{eq:kernel-breve} that, in the renamed coordinates,
$$ \breve K_{\lambda,s}^{k,\ell}(x,u) := \lambda^{n + \frac 32}\, 2^{-n\ell} \int_0^\infty \int_\R \, e^{i \lambda \sigma \breve \Phi_0} \,\tilde a_{\la,\de}(\sigma,s) \, \eta_1(\tau)   \sin (2^{1-\ell} \, \tau)^{-n} \, d  \tau \, d\sigma,
$$
with phase (recall that $\de=\de(k,\ell)=2^{1-\ell}k^{-1}$)
\[
\breve \Phi_0 = \delta^{-1}-1+s -\left(k\pi + 2^{1-\ell} \, \tau\right)\Big(\delta |x|^2 \cot( 2^{1-\ell} \, \tau) - 
\frac {u}{k} -4\delta^{-1} u_c(\tau_{k,\ell})\Big).
\]
Here, $s$ is assumed to be in $|1,1+\ve].$ 

Next, note that if $\ell\ge 2$ and $ \tau\in \supp \eta_1,$  then (compare also \eqref{supptau})
$$
2^{1-\ell}\frac 3{16} \pi\le  |2^{1-\ell} \, \tau| \le 2^{1-\ell} \frac 5{8}\pi\le \frac 5{16}\pi <\frac{\pi}{2},
$$
so that we may change coordinates via 
$$
\breve\tau = 2^{\ell-1} \tan(2^{1-\ell} \, \tau), \quad \text{i.e.,} \quad  2^{1-\ell} \tau =  \arctan(2^{1-\ell} {\breve\tau}).
$$
Observe that for $\ell=1$ this change of coordinates is not possible, since then $2^{1-\ell}\tau=\tau$ can become $\frac{\pi}{2}$ and we may run into a singularity of the tangent.

Using ${\sin(\arctan \alpha) = \alpha (1 + \alpha^2)^{-\frac 12}}$, with $\alpha=2^{1-\ell} {\breve\tau},$ we may rewrite 
\[
\breve K_{\lambda,s}^{k,\ell}(x,u) = \lambda^{n + \frac 32} \int_0^\infty \int_\R \, e^{i \lambda \sigma \breve \Phi_1} \,\tilde a_{\la,\de}(\sigma,s) \, \breve\eta_{1,\ell} (\breve \tau)   \, d  \breve\tau \, d\sigma,
\]
with phase 
\[
\breve \Phi_1 =\delta^{-1}-1+s -\Big(\pi + \frac{\arctan(2^{1-\ell} \breve \tau)}{k} \Big) \Big(\frac{| x|^2}{{\breve\tau}} - u  -2^{\ell+1} k^2 u_c(\tau_{k,\ell})\Big),
\]
and 
\[
\breve\eta_{1,\ell} (\breve\tau):=2^{-n}\, \eta_{1,\ell}\big(2^{\ell-1} \arctan(2^{1-\ell} \breve \tau)\big) \frac {(1 + (2^{1-\ell} {\breve\tau})^2)^{\frac n 2 -1}}{\breve \tau^n}.
\]
Note that the function $\breve\eta_{1,\ell}$ has a similar support as $\eta_1$, so that $|\breve\tau|\sim 1$ on its support, and has uniform bounds on its derivatives, independently of $\ell\ge 2$.

Finally, changing coordinates via ${\breve\tau}=1/\tau$ and setting
\[
 \varphi_\ell(\tau) := 2^{\ell-1} \arctan\Big(\frac{2^{1-\ell}}{ \tau}\Big),
\] 
we may write 
\begin{equation}\label{eq:kernel-bar-5}
\breve K_{\lambda,s}^{k,\ell}(x,u) = \lambda^{n + \frac 32} \int_0^\infty \int_\R \, e^{i \lambda \sigma \breve \Phi} \,\tilde a_{\la,\de}(\sigma,s) \, \breve\eta_{1,\ell} (\tau)   \, d  \tau \, d\sigma,
\end{equation}
with a slightly modified function $\breve\eta_{1,\ell} $ and phase 
\begin{equation}\label{eq:Phi-breve}
\breve \Phi := \delta^{-1}-1 + s - (\pi + \delta  \varphi_\ell(\tau)) \big( {\tau}|x|^2  - u-2^{\ell+1} k^2 u_c(\tau_{k,\ell})\big).
\end{equation}
Note also that every derivative of $\varphi_\ell$ is uniformly bounded on the support of $\breve\eta_{1,\ell},$ independently of $\ell\ge 2$.

\subsection{The $TT^*$ argument for the operators $\breve T^{k,\ell,\varepsilon}_\lambda$}\label{TT*}

Analogous to the case $k=0$, we will prove $L_{(x,u)}^{p'}(L_s^2) \to L_{(x,u)}^{p}(L_s^2)$ estimates for the operator $\breve T^{k,\ell,\varepsilon}_\lambda (\breve T^{k,\ell,\varepsilon}_\lambda)^*$.

\begin{lemma}\label{lem:kl-TTstar}
The integral kernel $\breve Q^{k,\ell,\varepsilon}_{\lambda}$ of $\breve T^{k,\ell,\varepsilon}_\lambda(\breve T^{k,\ell,\varepsilon}_\lambda)^*$ is given by
\begin{multline}\label{eq:Qkl}
\breve Q^{k,\ell,\varepsilon}_{\lambda}(x,u,s,x',u',s')
 = \lambda^{d+1}\,\chi_\varepsilon(x-x^0,u-u^0)\,\chi_\varepsilon(x'-x^0,u'-u^0) \\ \times \int_{\eta\sim 1} \int_{|\tau'|\sim 1} \int_{|\tau|\sim 1} \int_{\R^{d_1}} e^{i\lambda \eta \Psi} a(\eta,\tau,\tau',y,s,x',u',s')\, dy \, d\tau \, d\tau' d\eta, 
\end{multline}
with phase function
\begin{multline}
\Psi = \frac{\delta^{-1}-1 + s}{\pi + \delta \varphi_\ell(\tau)}
 - \frac{\delta^{-1}-1 + s'}{\pi + \delta \varphi_\ell(\tau')} 
 - \tau| x-y|^2 + \tau'| x'-y|^2 \\
 + \big( u - u' + \tfrac 1 2 \langle 2^{1-\ell} J(x-x'),y\rangle\big) \label{eq:Psi-kl}
\end{multline}
and full amplitude
\begin{multline}\label{eq:kl-TTstar-amplitude}
a(\eta,\tau,\tau',y,s,x',u',s')
= a_{1,\la,\de}(\eta,\tau,s)\,\chi_\varepsilon(y)\,\breve\eta_{1,\ell}(\tau) \, \breve\eta_{1,\ell}(\tau')\\
\times \frac{h_\lambda(\eta,\tau,\tau',s',x',u',y)}{(\pi+\delta\varphi_\ell(\tau))(\pi+\delta\varphi_\ell(\tau'))},
\end{multline}
where
\begin{equation}\label{eq:kl-TTstar-a}
a_{1,\la,\de}(\eta,\tau,s) 
= a_{\la}\Big(\big(1+\delta(s-1)\big)\frac \eta{\pi + \delta \varphi_\ell(\tau)}\Big)
\end{equation}
and
\begin{equation}
\label{hlambda}
h_\lambda(\eta,\tau,\tau',s',x',u',y) 
= \int_\R \, e^{i\lambda\xi\Psi_2} 
\lambda\, \hat{\chi_\varepsilon}(\lambda\xi)
\, \overline{a_{1,\la,\de}(\eta-\xi,\tau',s')}
\,d\xi,
\end{equation}
with
\begin{equation}\label{Psi2}
\Psi_2
= \frac{\delta^{-1}-1 + s'}{\pi + \delta \varphi_\ell(\tau')} - \tau'\left|x'-y\right|^2 +  2^{\ell+1} k^2 u_c(\tau_{k,\ell} ) + u' + \tfrac 1 2 \langle 2^{1-\ell} Jx',y\rangle.
\end{equation}
The full amplitude $a$ and all its partial derivatives are uniformly bounded on the support of the integral, with bounds controlled by finitely many derivatives of $a_\la$.
\end{lemma}

\begin{proof}
Let us again set $r:=2^{1-\ell}$. Arguing in an analogous way as in the proof of Lemma~\ref{lem:0-TTstar}, the integral kernel 
$\breve Q_\lambda^{k,\ell,\varepsilon}(x,u,s,x',u',s')$ is given by
\begin{multline*}
\tilde \chi_\ve(x,u)\, \tilde \chi_\ve(x',u')\int_{\R^{d_1}} \int_{\R} \breve K_{\lambda,s}^{k,\ell}\big(x-y,u-v + \tfrac 1 2 \langle rJx,y\rangle\big) \\
\times \overline{ \breve K_{\lambda,s'}^{k,\ell}\big(x'-y,u'-v + \tfrac 1 2 \langle rJx',y\rangle\big)} \, \chi_\varepsilon(y) \, \chi_\varepsilon(v) \, dv \, dy,
\end{multline*}
with slightly modified functions $\chi_\ve.$ Assuming in the sequel the localizations given by $\tilde \chi_\ve(x,u)\, \tilde \chi_\ve(x',u'),$ we shall suppress these localization factors from here on. The above then equals
\begin{multline*}
\lambda^{d+2} \int_{\sigma'\sim1} \int_{\sigma\sim1}  \int_{|\tau'|\sim1} \int_{|\tau|\sim1} \, \int_{\R^{d_1}} \int_{\R} e^{i\lambda \breve \Psi} \, \tilde a_{\la,\de}(\sigma,s)\, \overline{\tilde a_{\la,\de}(\sigma',s')} \, \breve\eta_{1,\ell}(\tau) \\
\times \breve\eta_{1,\ell}(\tau')\, \chi_\varepsilon(y) \, \chi_\varepsilon(v) \, dv \, dy \, d\tau \, d\tau' \, d\sigma \, d\sigma',
\end{multline*}
where  the phase function $\breve\Psi$ is given by
\begin{align}
\breve\Psi = \ & \sigma \,\breve\Phi(\tau,x-y,u-v + \tfrac 1 2 \langle rJx,y\rangle,s) \label{eq:Phi-Phi'}\\
& - \sigma' \breve\Phi(\tau',x'-y, u'-v + \tfrac 1 2 \langle rJx',y\rangle,s').  \notag
\end{align}
Using \eqref{eq:Phi-breve}, and changing coordinates from $\sigma$ and $\sigma'$ to
\[
\eta = \sigma (\pi + \delta \varphi_\ell(\tau) ) \quad\text{and}\quad \eta' = \sigma' (\pi + \delta\varphi_\ell(\tau') ),
\]
we see that $\eta\sim 1$ and $\eta'\sim 1,$ and \eqref{eq:Phi-Phi'} becomes
\begin{multline}
 \eta \Big( \frac{\delta^{-1}-1 + s}{\pi + \delta \varphi_\ell(\tau)} - \tau| x-y|^2  +  \big( 2^{\ell+1} k^2 u_c(\tau_{k,\ell}) + u-v + \tfrac 1 2 \langle rJx,y\rangle\big)\Big) \label{eq:Psi-i}\\
 - \eta' \Big( \frac{\delta^{-1}-1 + s'}{\pi + \delta \varphi_\ell(\tau')} - \tau'\left|x'-y\right|^2  + \big(2^{\ell+1} k^2 u_c(\tau_{k,\ell}) + u'-v + \tfrac 1 2 \langle rJx',y\rangle\big)\Big).
\end{multline}
Note that the terms containing $v$ sum up to $(\eta'-\eta)v$. We may therefore first carry out the integration in $v$, obtaining
\[
\int_{\R} e^{i\lambda (\eta' -\eta) v} \chi_\varepsilon(v) \, dv = \, \hat{ \chi_\ve} \left( \la(\eta-\eta') \right).
\]
This yields a strong localization to $|\eta-\eta'|\lesssim\lambda^{-1}$. We next change coordinates from $\eta'$ to $\xi$ such that $\eta -\eta'= \xi$. Then the phase function in \eqref{eq:Psi-i} becomes $\breve\Psi=\eta \Psi+ \xi \Psi_2$, with $\Psi$ and $\Psi_2$ given by \eqref{eq:Psi-kl} and \eqref{Psi2}, respectively.

Hence, after carrying out the integration in $v$ and changing coordinates from $\eta'$ to $\xi$, we obtain the oscillatory integral \eqref{eq:Qkl}, with phase function given by \eqref{eq:Psi-kl} and full amplitude given by \eqref{eq:kl-TTstar-amplitude}--\eqref{hlambda}.

Finally, for the uniform bounds for the full amplitude and all its partial derivatives, note that $\pi+\delta\varphi_\ell(\tau')\sim 1$ and that all derivatives of $\delta^{-1}/(\pi+\delta\varphi_\ell(\tau'))$ with respect to $\tau'$ are $\mathcal O(1)$, uniformly in $\delta$. Moreover, the factor $\lambda\hat\chi_\varepsilon(\lambda\xi)$ localizes effectively to $|\xi|\lesssim\lambda^{-1}$.
\end{proof}

Given $x,x'\in \R^{d_1}$, we split coordinates as $x = (x_1,\underline x)$ and $x' = (x_1',\underline x')$. For fixed $x_1\in \R$, let $\breve T_{\lambda,x_1}^{k,\ell,\varepsilon}$ be the operator given by
\[
(\breve T_{\lambda,x_1}^{k,\ell,\varepsilon} \varphi)(\underline x,u,s) = (\breve T^{k,\ell,\varepsilon}_\lambda \varphi)(x,u,s).
\]
If we write $f(x',u',s') = f_{x_1'}(\underline x',u',s')$, then
\[
\big(\breve T_\lambda^{k,\ell,\varepsilon} (\breve T_\lambda^{k,\ell,\varepsilon})^* f\big)(x, u,s) = \int_\R \, \big( \breve T_{\lambda,x_1}^{k,\ell,\varepsilon} (\breve T_{\lambda,x_1'}^{k,\ell,\varepsilon} )^* \, f_{x_1'}\big) (\underline x,u,s) \, dx_1'.
\]
Note that the integral kernel $\breve Q_{\lambda,x_1,x_1'}^{k,\ell,\varepsilon}$ of $\breve T_{\lambda,x_1}^{k,\ell,\varepsilon} (\breve T_{\lambda,x_1'}^{k,\ell,\varepsilon} )^*$ is given by $\breve Q_\lambda^{k,\ell,\varepsilon}(x_1,\cdot,x_1',\cdot)$.
\smallskip

In analogy with Propositions \ref{lem:0-L1} and \ref{lem:0-L2} we shall here prove the following propositions:
\begin{proposition}\label{lem:kl-L1}
We have
\[
\|\breve T_{\lambda,x_1}^{k,\ell,\varepsilon} (\breve T_{\lambda,x_1'}^{k,\ell,\varepsilon} )^* \|_{L_{(\underline x,u)}^{1}(L^2_s)\to L_{(\underline x,u)}^\infty(L^2_s)}
\lesssim \lambda^{d-1} \, \langle\lambda\left|x_1'-x_1\right|\rangle^{-\frac{d_1-1}{2}} .
\]
\end{proposition}

\begin{proposition}\label{lem:kl-L2}
We have
\[
\|\breve T_{\lambda,x_1}^{k,\ell,\varepsilon} (\breve T_{\lambda,x_1'}^{k,\ell,\varepsilon} )^*\|_{L^2(\R^{d})\to L^2(\R^{d})} \lesssim \delta^{-1}.
\]
\end{proposition}

The square function estimate in Proposition~\ref{prop:unit}  follows from these two propositions by the same arguments as in the proof of Proposition~\ref{L:T0lambda}. The proofs of  Propositions \ref{lem:kl-L1} and \ref{lem:kl-L2} will be given in Sections \ref{sec:pointwise-inter} and \ref{eq:inter-L2}.

\section{Proof of the key estimates for $\ell\ge 2$} 

In this section, we prove Propositions \ref{lem:kl-L1} and \ref{lem:kl-L2}.

\subsection{Pointwise estimates for $\ell\ge 2$}\label{sec:pointwise-inter}

We follow essentially the scheme of the proof of Proposition~\ref{lem:0-L1}. By Schur's test, Proposition~\ref{lem:kl-L1} will again be an immediate consequence  of  the next analogue of Proposition~\ref{lem:pointwiseQ0}.

\begin{proposition}\label{lem:pointwiseQkl} 
If the constant $\ve>0$ is sufficiently small, then for all $s\in[1,1+\ve],$
\begin{equation}\label{sprimeintestkl}
\int_1^{1+\ve}| \breve Q^{k,\ell,\varepsilon}_\la(x,u,s,x',u',s')|\, ds'  
\le C(\ve,a_\la)\, \lambda^{d-1}
\, \langle\lambda\left|\Delta x\right|\rangle^{-\frac{d_1-1}2} ,
\end{equation}
and for all $s'\in[1,1+\ve],$
\begin{equation}\label{sintestkl}
\int_1^{1+\ve}| \breve Q^{k,\ell,\varepsilon}_{\lambda}(x,u,s,x',u',s')|\, ds 
\le C(\ve,a_\la)\, \lambda^{d-1}
\, \langle\lambda\left|\Delta x\right|\rangle^{-\frac{d_1-1}2} ,
\end{equation}
uniformly in $(x,u)$ and $(x',u')$, $k\ge 1$, and $\ell\ge 2$, where the constant $C(\ve,a_\la)$ can be chosen to be independent of $(x^0,u^0)$ and to depend only on the $C^M$-norm of $a_\la$ for some sufficiently large $M$. 
\end{proposition}

We will prove Proposition~\ref{lem:pointwiseQkl} by showing the following.

\begin{lemma}\label{lem:pointwiseQkl-prec}
We can decompose $\breve Q^{k,\ell,\varepsilon}_\la=\breve Q^{k,\ell,\varepsilon}_{1,\lambda}+E^{k,\ell,\varepsilon}_\la,$ where the main term 
$\breve Q^{k,\ell,\varepsilon}_{1,\la}$ can be estimated pointwise by 
\begin{multline}\label{eq:pointwiseQkl}
|\breve Q^{k,\ell,\varepsilon}_{1,\la}(x,u,s,x',u',s')| \le C_N(\ve,a_\la)\, \lambda^d
\, \langle\lambda\left|\Delta x\right|\rangle^{-\frac{d_1-1}2} \\
\times \int_\R \big\langle \lambda\, (\Delta s - q(s',\Delta x,\Delta w_\ell,\tau) ) \big\rangle^{-N} \chi_1(\tau)\, d\tau
\end{multline}
for any $N\in\N,$ where $\chi_1$ localizes to $|\tau|\sim 1,$ and where, with $r=2^{1-\ell}$,
\[
\Delta x = x-x',\, \Delta w_\ell = u-u' +\tfrac 12\langle rJ\Delta x,x\rangle \text{ and }\Delta s=s-s'.
\] 
Here $q$ is a smooth function satisfying estimates of the form $\partial^j_s q=\mathcal O(\ve)$ for every $j\in\N,$
and the  constant $C_N(\ve,a_\la)$ can be chosen to be independent of $\tau$ and to depend, for a given $N\in\N$, on $a_\la$ only through the $C^M$-norm of $a_\la$ for some sufficiently large $M$ (which may depend on $N$). 

For every $N\in \N$, for all $s\in[1,1+\ve],$
\[
\int_1^{1+\ve}| E^{k,\ell,\varepsilon}_\la(x,u,s,x',u',s')|\, ds'  
\le C_N(\ve,a_\la)\, \lambda^d
\, \langle\lambda\left|\Delta x\right|\rangle^{-N} ,
\]
and for all $s'\in[1,1+\ve],$
\[
\int_1^{1+\ve}| E^{k,\ell,\varepsilon}_\la(x,u,s,x',u',s')|\, ds 
\le C_N(\ve,a_\la)\, \lambda^d
\, \langle\lambda\left|\Delta x\right|\rangle^{-N}  ,
\]
uniformly in $(x,u)$ and $(x',u')$, $k\ge 1$, and $\ell\ge 2$.
\end{lemma}

The rest of this section will be  devoted to the proof of Lemma~\ref{lem:pointwiseQkl-prec}, which will be kept a bit more sketchy compared to the proof of Proposition~\ref{lem:pointwiseQ0} and will concentrate mostly on differences compared to the arguments used for Proposition~\ref{lem:pointwiseQ0}.

\subsubsection{Freezing $\tau$}
By Lemma~\ref{lem:kl-TTstar}, we have
\[
\breve Q^{k,\ell,\varepsilon}_{\lambda}
 =  \int_{|\tau|\sim1}  \breve Q^{k,\ell,\varepsilon}_{\lambda,\tau} \, d\tau ,
\]
where (suppressing again the localizing factors $\chi_\varepsilon(x-x^0,u-u^0)\,\chi_\varepsilon(x'-x^0,u'-u^0)$) 
\begin{equation}
\breve Q^{k,\ell,\varepsilon}_{\lambda,\tau} =\lambda^{d+1} \int_{\eta\sim 1} \int_{|\tau'|\sim 1} \int_{\R^{d_1}} e^{i\lambda \eta \Psi} a(\eta,\tau,\tau',y,s,x',u',s') \, dy \, d\tau' d\eta, \label{eq:q-kl}
\end{equation}
with phase $\Psi$ given by \eqref{eq:Psi-kl}. In the sequel, we assume that  $\tau$ with $|\tau|\sim 1$ is frozen.

\subsubsection{Changes of  variables and  stationary phase for a spherical integration}
Let us again use the abbreviation $r=2^{1-\ell}$. In the oscillatory integral \eqref{eq:q-kl}, we first change coordinates from $y$ to $y'=y-x$.  Since $|x|\sim 1$, we then have $|y'|\sim 1$. We use polar coordinates $y' = \rho\,\omega$, where $\rho\in(0,\infty)$ satisfies $\rho \sim 1$, and where  $\omega\in S^{d_1-1}$. Then  the phase function  $\Psi$ in \eqref{eq:Psi-kl} changes to
\[
\Psi_1
 = \frac{\delta^{-1}-1 + s}{\pi + \delta \varphi_\ell(\tau)}
 - \frac{\delta^{-1}-1 + s'}{\pi + \delta \varphi_\ell(\tau')}
+(\tau'-\tau)\rho^2 + \tau'|\Delta x|^2 + \rho \, \langle\theta ,\omega\rangle 
 + \Delta w_\ell,
\]
where
\[
\theta = 2\tau' \Delta x+\tfrac r2 J\Delta x.
\]
Since $J^2=-I_{d_1}$ and $J$ is skew-symmetric, we have
\[
| \theta|^2
= \Big(4 \, (\tau')^2+\frac {r^2}4 \Big) |\Delta x|^2 
\sim |\Delta x|^2.
\]
Note also that the factor $\chi_{\varepsilon}(\rho\omega + x)$ localizes the integration in $\omega$ to an $\mathcal{O} (\ve)$ neighborhood of the point $-x^0/|x^0|$ in $S^{d_1-1}$. Absorbing the polar-coordinate Jacobian and the cutoffs into a new full amplitude $a_1$, we obtain an oscillatory integral of the form
\[
\lambda^{d+1} \int_{\eta\sim 1} \int_{|\tau'|\sim 1} \int_{\rho\sim1}\int_{S^{d_1-1}}
 e^{i\lambda\eta\Psi_1}  a_1\,d\omega\,d\rho\,d\tau' d\eta.
\]
Next, we change coordinates from $\tau'$ to $\beta$ via
\[
\frac{\tau'}{\tau} = 1 + |\Delta x| \beta ,
\]
so that $\tau'-\tau = \beta\tau |\Delta x|$. This turns \eqref{eq:q-kl} into an oscillatory integral of the form
\begin{multline}\label{eq:change}
\lambda^{d+1}  |\Delta x| \int_{\eta\sim 1} \int_{S^{d_1-1}} \int_{\rho\sim 1}  \int_\R e^{i\lambda\eta \Psi_2} a_2(\eta,\tau,\beta,s,s',x,x',u',\rho,\omega)\\ \times \chi_\varepsilon(|\Delta x|\beta) \,
 d\beta\,  d\rho\, d\omega\,d\eta,
\end{multline}
where the Jacobian factor has been absorbed into $a_2$, and
\begin{multline*}
\Psi_2 = 
 \frac{\delta^{-1}-1 + s}{\pi + \delta \varphi_\ell(\tau)}
- \frac{ \delta^{-1}-1 + s'}{\pi + \delta \varphi_\ell\left(\tau + \beta\tau |\Delta x|\right)}
+\beta\tau \left|\Delta x\right|\rho^2\\
 + \rho \, \langle\tilde\theta ,\omega\rangle +\tau\left(1+ \beta|\Delta x|\right)|\Delta x|^2 + \Delta w_\ell,
\end{multline*}
with
\[
\tilde\theta = 2\tau \left(1+ \beta|\Delta x|\right) \Delta x+\tfrac r 2 J\Delta x.
\]
We write
\[
\Psi_2 = \frac{s}{\pi + \delta \varphi_\ell(\tau)}-\frac{s'}{\pi + \delta \varphi_\ell(\tau + \beta\tau |\Delta x|)}
+q_\ell(\tau, \beta|\Delta x|,\rho,\omega),
\]
where
\begin{multline*}
q_\ell(\tau, \beta|\Delta x|,\rho,\omega)
=(1-\delta)\, \frac{\varphi_\ell(\tau + \beta\tau |\Delta x|)-\varphi_\ell(\tau)}{(\pi + \delta \varphi_\ell(\tau + \beta\tau |\Delta x|))(\pi+\delta\varphi_\ell(\tau))}
+\beta\tau \left|\Delta x\right|\rho^2  \\
+ \rho \, \langle\tilde \theta ,\omega\rangle 
+\tau\left(1+ \beta|\Delta x|\right)|\Delta x|^2 + \Delta w_\ell.
\end{multline*}
The amplitude $a_2$ and all its partial derivatives are uniformly bounded on the support of the integral, with the same type of uniformity as the amplitude $a$ in Lemma~\ref{lem:kl-TTstar}.

Let us denote by $E^{\varepsilon,2}_{\lambda,\tau}$ the contribution of the region where $|\beta|\gg 1$ to \eqref{eq:change}. For $\beta$ in this region, we can integrate by parts in $\rho$ in order to gain negative powers $\jp{\lambda|\Delta x||\beta|}^{-N}.$ Applying subsequently integrations by parts in $\eta,$ we may write  
\begin{multline*}
E^{\varepsilon,2}_{\lambda,\tau}= \la^d \,\int_{\eta\sim 1} \int_{S^{d_1-1}} \int_{\rho\sim 1}  \int_{|\beta|\gg 1} e^{i\lambda\eta \Psi_2}  
\big\langle \lambda\, \Psi_2 \big\rangle^{-N} \,  (\la |\Delta x|)\jp{\lambda|\Delta x||\beta|}^{-N}\\
\times a_{2,N} \,  \chi_\ve(\left|\Delta x\right|\beta)\, d\beta\,  d\rho\, d\omega\,d\eta,
\end{multline*}
where $a_{2,N}$ is another amplitude that satisfies the corresponding uniform symbol bounds. 
From here on, we can argue as for the error term $\smash{E^{\varepsilon,2}_{\lambda,\tau'}}$ in the proof of Proposition~\ref{lem:pointwiseQ0} to see that $E^{\varepsilon,2}_{\lambda,\tau}$ is an error term which gives a contribution to 
$E^{k,\ell,\varepsilon}_\la.$ Note here that, in contrast to the situation in the latter proposition, only the term 
\[
\frac{s}{\pi + \delta \varphi_\ell(\tau)}-\frac{s'}{\pi + \delta \varphi_\ell(\tau + \beta\tau |\Delta x|)}
\]
depends on $s$ and $s'$, so after freezing the frequency variables $\tau,\beta, \omega, \rho,$ the estimates \eqref{sprimeintestkl} and \eqref{sintestkl} become clear.

We are thus reduced to considering the contribution by the  region where $|\beta|\lesssim 1$, that is, an oscillatory integral of the form
\[
Q^{\varepsilon,2}_{\lambda,\tau}=
\lambda^{d} (\la|\Delta x|)\,   \int_{\eta\sim 1} \int_{S^{d_1-1}} \int_{\rho\sim 1}  \int_{|\beta| \lesssim 1} e^{i\lambda\eta \Psi_2} \, \tilde a_2 \, \chi_\ve(\left|\Delta x\right|\beta)\, d\beta\,  d\rho\, d\omega\,d\eta,
\]
where, as often in what follows, we have suppressed the dependence of $Q^{\varepsilon,2}_{\lambda,\tau}$ on  $k,\ell$.

Next, performing stationary phase in $\omega$ when $|\underline{\Delta x}|\lesssim \ve |\Delta x_1|,$   and $N$ integrations by parts in $\omega$ along $S^{d_1-1}$ when $|\underline{\Delta x}|\gg  \ve |\Delta x_1|,$ we find that 
$$
Q^{\varepsilon,2}_{\lambda,\tau}=Q^{\varepsilon,3}_{\lambda,\tau}+E^{\varepsilon,3}_{\lambda,\tau}.
$$
Here, the first term is given by 
\begin{multline}
Q^{\varepsilon,3}_{\lambda,\tau}=
\lambda^{d} (\la|\Delta x|)\,   \int_{\eta\sim 1}  \int_{\rho\sim 1}  \int_{|\beta| \lesssim 1} 
\langle \lambda\eta \rho |\Delta x|\rangle^{-\frac{d_1-1}{2}} e^{i\lambda\eta \Psi^\pm_3}  a_3 \\
\times  \chi_\ve(\left|\Delta x\right|\beta)\, d\beta\,  d\rho\,d\eta, \label{eq:stat-rho-i}
\end{multline}
where $a_3$ is an amplitude satisfying the same type of uniform symbol bounds and in addition localizes to the region where 
$|\underline{\Delta x}|\lesssim \ve |\Delta x_1|,$ and the  phase is given by
\begin{multline*}
\Psi^\pm_3 = 
 \frac{\delta^{-1}-1 + s}{\pi + \delta \varphi_\ell(\tau)}
- \frac{ \delta^{-1}-1 + s'}{\pi + \delta \varphi_\ell\left(\tau + \beta\tau |\Delta x|\right)}
+\beta\tau \left|\Delta x\right|\rho^2 \\
\pm |\tilde \theta| \rho +\tau\left(1+ \beta|\Delta x|\right)|\Delta x|^2 + \Delta w_\ell.
\end{multline*}

The second term $E^{\varepsilon,3}_{\lambda,\tau}$  corresponding to the region where $|\underline{\Delta x}|\gg \ve |\Delta x_1|$ is again an error term, which can be handled in  a similar way as the one in \eqref{eq:Q0-stat} by means of integrations by parts in $\omega$ and $\eta$. It thus again just contributes to $E^{k,\ell,\varepsilon}_\la.$

\subsubsection{Stationary phase  for the integration in $\rho$}
Note that 
\[
\psi^\pm=\beta\tau \left|\Delta x\right|\rho^2\pm |\tilde \theta| \rho =\rho^2\psi_2+\rho \psi_1
\]
collects all terms of the phase $\Psi^\pm_3$ which depend on $\rho.$ Here, $|\psi_2|\sim \beta |\Delta x|,$ and, as we have seen,
$|\psi_1|\sim |\tilde \theta|\sim |\Delta x|.$

Thus, if $|\beta|\ll 1,$ then $|\partial_\rho\psi^{\pm}| \sim |\Delta x|,$  so we can again integrate by parts in $\rho$ in order  to gain  any factors $ \jp{\lambda\Delta x}^{-N}.$ This shows that  also the contributions from the region where  $|\beta|\ll 1$ can be handled as an error term in a similar way as before.

We are thus left with an oscillatory integral of the form
\begin{equation}
Q^{\varepsilon,4}_{\lambda,\tau}=
\lambda^{d} (\la|\Delta x|)\,   \int_{\eta\sim 1}  \int_{\rho\sim 1}  \int_{|\beta| \sim 1} 
\langle \lambda\eta \rho |\Delta x|\rangle^{-\frac{d_1-1}{2}} e^{i\lambda\eta \Psi^\pm_3}   \tilde a_3\, d\beta\,  d\rho\,d\eta.   \label{eq:stat-rho-ii}
\end{equation}
Note that
\[
\partial_\rho \psi^\pm
=2\beta \tau \left|\Delta x\right| \rho \pm |\tilde\theta|.
\]
Thus, if $|\beta|\sim 1$, then $\psi^\pm$ has a unique critical point
$\rho_c\sim 1$ in $\rho$ if and only if $\mp\beta\sim 1$. For the following discussion, we shall consider the phase $\Psi_3^-$; the case of the phase $\Psi_3^+$ can be treated analogously.

If  $-\beta\sim 1,$ we can then again integrate by parts in $\rho$ and see that the corresponding contribution can be treated as an error term as before. Thus, we shall assume in the sequel that $\beta\sim 1,$ and that $\psi^-$ does have a critical point for any such $\beta,$ which is then clearly given by
\[
\rho_c = \frac{|\tilde\theta|}{2\beta \tau|\Delta x|}=\Big(4\tau^2(1+\beta|\Delta x|)^2+\tfrac {r^2}4\Big)^{\frac{1}{2}}/(2\beta\tau).
\]
Note also that $|\partial_\rho^2 \psi^-|\sim |\Delta x|.$ Thus, applying stationary phase in $\rho$, we are reduced to the oscillatory integral 
\begin{equation}
Q^{\varepsilon,4}_{\lambda,\tau}=
\lambda^{d} (\la|\Delta x|)\,   \int_{\eta\sim 1}  \int_{\beta \sim 1} \,
\langle \lambda|\Delta x|\rangle^{-\frac{d_1}{2}} e^{i\lambda\eta \Psi_4} a_4 \,   d\beta\,d\eta,   \label{eq:stat-rhoc}
\end{equation}
where
\begin{multline*}
\Psi_4 = 
 \frac{\delta^{-1}-1 + s}{\pi + \delta \varphi_\ell(\tau)}
- \frac{ \delta^{-1}-1 + s'}{\pi + \delta \varphi_\ell\left(\tau + \beta\tau |\Delta x|\right)}
- \frac{4\tau^2(1+\beta|\Delta x|)^2+\tfrac {r^2}4}{4\beta\tau} \,|\Delta x|\\
+\tau\left(1+ \beta|\Delta x|\right)|\Delta x|^2 + \Delta w_\ell.
\end{multline*}
As in Section \ref{pointwise0}, the factor $\langle \lambda|\Delta x|\rangle^{-\frac{d_1}{2}}$ in the integral is here to be interpreted as a shorthand for a symbol of order $-\frac{d_1}{2}$ in $\la|\Delta x|$.

\subsubsection{Stationary phase for the integration in $\beta$}

Note first that 
\[
\partial_\beta \Psi_4=H(\beta, s',\tau, |\Delta x|)\,|\Delta x|,
\] 
where
\begin{multline*}
H = \frac{ 1+\delta (s'-1)}{\big(\pi + \delta \varphi_\ell(\tau + \beta\tau |\Delta x|)\big)^2}\, \tau
\varphi'_\ell\left(\tau + \beta\tau |\Delta x|\right) \\
+  \frac{4\tau^2(1+\beta |\Delta x|)^2+\tfrac {r^2}4}{4\beta^2\tau}  
+ \scriptO(|\Delta x|).
 \end{multline*}
 For fixed $\tau_0$ and $s'=1,|\Delta x|=0,$ we obtain
\[
H(\beta,1,\tau_0,0)
=\frac{\tau_0\varphi'_\ell(\tau_0)}
{(\pi+\delta\varphi_\ell(\tau_0))^2}
+\frac{4\tau_0^2+\tfrac{r^2}{4}}{4\beta^2\tau_0}.
\]
 This function has a unique zero $\beta_c(1,\tau_0,0)\sim 1.$ Moreover,
\[
\partial_\beta H(\beta,1,\tau_0,0)
=-2\frac{4\tau_0^2+\tfrac{r^2}{4}}
{4\beta^3\tau_0},
\qquad
\left|\partial_\beta H(\beta,1,\tau_0,0)\right|\sim 1.
\]
By the implicit function theorem, we thus see that for $|s'-1|$,
$|\Delta x|$, and $|\tau-\tau_0|$ sufficiently small, there is a
unique zero
\[
\beta_c(s',\tau,|\Delta x|)\sim 1
\]
of $H(\beta,s',\tau,|\Delta x|)$ in $\beta$, which is then also the
unique critical point of $\Psi_4$ in $\beta$.
 
Note that by means of a decomposition into intervals of length $\ve$,
we may indeed reduce to assuming that $\tau$ lies in an interval
$[\tau_0,\tau_0+\ve]$. Thus, if we assume that $\ve$ is sufficiently
small, the existence of the unique critical point
$\beta_c(s',\tau,|\Delta x|)\sim 1$ of
$H(\beta,s',\tau,|\Delta x|)$ is guaranteed. Moreover,
\[
\left|\partial_\beta^2\Psi_4\right|\sim|\Delta x|
\]
at this critical point.
Applying the method of stationary phase in $\beta$, we then obtain an oscillatory integral of the form
 \begin{equation}
Q^{\varepsilon,4}_{\lambda,\tau}=
\lambda^{d} (\la|\Delta x|)\,   \int_{\eta\sim 1}\,
\langle \lambda|\Delta x|\rangle^{-\frac{d_1+1}{2}} e^{i\lambda\eta \Psi_5}  a_5 \, d\eta,   
\end{equation}
with phase $\Psi_5=\Psi_4|_{\beta=\beta_c}$. Since 
\begin{align*}
& \frac{\delta^{-1}-1 + s}{\pi + \delta \varphi_\ell(\tau)}
- \frac{ \delta^{-1}-1 + s'}{\pi + \delta \varphi_\ell(\tau + \beta_c \tau |\Delta x|)}\\
& = \frac{1}{\pi + \delta \varphi_\ell(\tau)} \bigg[ s-s'+ \left(1-\delta+\delta s'\right)
 \frac{\varphi_\ell(\tau + \beta_c \tau |\Delta x|)-\varphi_\ell(\tau) }{\pi+\delta\varphi_\ell(\tau + \beta_c\tau |\Delta x|)}\bigg]\\
& =  \frac{1}{\pi + \delta \varphi_\ell(\tau)} \Big(\Delta s+\scriptO(|\Delta x|)\Big),
\end{align*}
we see that 
\[
\Psi_5=\frac{1}{\pi + \delta \varphi_\ell(\tau)} \left(\Delta s - q(s',\Delta x,\Delta w_\ell,\tau)\right),
\]
where
\[
q=\scriptO\big(|\Delta x|+|\Delta w_\ell|\big)=\scriptO(\ve),
\qquad
\partial_{s'}^j q=\scriptO(|\Delta x|),\quad j\geq1.
\]
By means of integrations by parts in $\eta,$ it now follows easily that the main term 
\[
\breve Q^{k,\ell,\varepsilon}_{1,\lambda}=
  \int_{|\tau|\sim1}  \breve Q^{\ve,4}_{\lambda,\tau} \, d\tau 
\]
satisfies estimates of the form \eqref{eq:pointwiseQkl}.
\qed

\subsection{Refined $L^2$ estimates for  $\ell\ge 2$}\label{eq:inter-L2}

In this section we shall prove Proposition~\ref{lem:kl-L2}. Recall from Lemma~\ref{lem:kl-TTstar} that the integral kernel $\breve Q^{k,\ell,\varepsilon}_{\lambda}$ of $\breve T^{k,\ell,\varepsilon}_\lambda(\breve T^{k,\ell,\varepsilon}_\lambda)^*$ can be written as
\begin{multline*}
\breve Q^{k,\ell,\varepsilon}_{\lambda}(x,u,s,x',u',s')
 = \lambda^{d+1}\,\chi_\varepsilon(x-x^0,u-u^0)\,\chi_\varepsilon(x'-x^0,u'-u^0) \\
\times \int_{\eta\sim 1} \int_{|\tau'|\sim 1} \int_{|\tau|\sim 1} \int_{\R^{d_1}}
e^{i\lambda \eta \Psi}
a(\eta,\tau,\tau',y,s,x',u',s')\,dy\,d\tau\,d\tau'\,d\eta,
\end{multline*}
with phase function (recall that $r=2^{1-\ell}$)
\begin{align*}
\Psi
= \frac{\delta^{-1}-1+s}{\pi+\delta\varphi_\ell(\tau)}
-\frac{\delta^{-1}-1+s'}{\pi+\delta\varphi_\ell(\tau')}
&-\tau|x-y|^2+\tau'|x'-y|^2 \\
&+\big(u-u'+\tfrac12\langle rJ(x-x'),y\rangle\big)
\end{align*}
and amplitude $a$ given by \eqref{eq:kl-TTstar-amplitude}, which comes with the factors
\[
a_{1,\la,\de}(\eta,\tau,s)
=
a_\la\Big(\big(1+\delta(s-1)\big)
\frac{\eta}{\pi+\delta\varphi_\ell(\tau)}\Big)
\]
and
\[
h_\lambda(\eta,\tau,\tau',s',x',u',y)
=
\int_\R e^{i\lambda\xi\Psi_2}
\lambda\,\hat\chi_\varepsilon(\lambda\xi)\,
\overline{a_{1,\la,\de}(\eta-\xi,\tau',s')}\,d\xi.
\]
Here $\Psi_2$ is given by \eqref{Psi2}. The function $h_\lambda$ and all its partial derivatives are uniformly bounded on the support of the integral, with bounds controlled by finitely many derivatives of $a_\la$.

Regretfully, compared to the proof of Proposition~\ref{lem:0-L2}, it turns out that here the Monge--Ampère determinant of $\Psi$ will only be of order $\mathcal{O}(\delta^2)$. To address this problem, we shall need to apply a more sophisticated argument, which exploits the special structure of the phase $\Psi$ and which is based on  a rescaling argument in the variable $s,$ in which $s$ will be replaced by $\delta^{-1} s$. 

For later use, let us write the full amplitude in the form
\[
a(\eta,\tau,\tau',y,s,x',u',s')
=
b(\eta,\tau,\tau',y,s)\,
h_\lambda(\eta,\tau,\tau',s',x',u',y),
\]
where
\begin{equation}\label{amplia}
b(\eta,\tau,\tau',y,s)
:=
a_{1,\la,\de}(\eta,\tau,s)\,\chi_\varepsilon(y)\,
\frac{\breve\eta_{1,\ell}(\tau)\,\breve\eta_{1,\ell}(\tau')}
{(\pi+\delta\varphi_\ell(\tau))(\pi+\delta\varphi_\ell(\tau'))}.
\end{equation}
Note that the scaling in $s$ turns $a_{1,\la,\de}(\eta,\tau,s)$ into
\begin{equation}\label{atilde3}
\breve a_{1,\la,\de}(\eta,\tau,s)=  a_{\la}\Big((1+s-\delta)\frac \eta{\pi + \delta \varphi_\ell(\tau)}\Big).
\end{equation}
Thus, all derivatives of the rescaled amplitude factor
\[
\breve b(\eta,\tau,\tau',y,s)
:=
b(\eta,\tau,\tau',y,\delta^{-1}s)
\]
are still uniformly bounded. Regretfully, the same is not immediately true for the rescaled function $h_\lambda$.

\subsubsection{Freezing $\xi$ in the defining integral for $h_\la$} 
Write 
\[
\Psi_2=\psi(s')+u'+\dot\Psi_2,
\qquad
\psi(s')=\frac{\delta^{-1}-1+s'}{\pi},
\]
where
\[
\dot \Psi_2
= -\frac{1+\delta(s'-1)}{\pi + \delta \varphi_\ell(\tau')}\frac {\varphi_\ell(\tau')}{\pi} - \tau'\left|x'-y\right|^2  +
2^{\ell+1} k^2 u_c(\tau_{k,\ell} ) + \tfrac 1 2 \langle rJx',y\rangle.
\]
We may then write
\begin{equation}\label{hladecomp}
h_\lambda(\eta,\tau,\tau',s',x',u',y)
=
\int_\R
e^{i\la\xi(\psi(s')+u')}
H_{\la,\xi}(\eta,\tau,\tau',s',x',y)\,d\xi,
\end{equation}
where
\[
H_{\la,\xi}(\eta,\tau,\tau',s',x',y)
:=
e^{i\la\xi\dot\Psi_2}
\lambda\,\hat\chi_\varepsilon(\lambda\xi)\,
\overline{a_{1,\la,\de}(\eta-\xi,\tau',s')}.
\]
Let us also put
\[
\breve H_{\la,\xi}(\eta,\tau,\tau',s',x',y)
:=
H_{\la,\xi}(\eta,\tau,\tau',\delta^{-1}s',x',y).
\]
Then, for every $M\in\N$, we have
\begin{equation}\label{tildeH}
\int_\R \|\breve H_{\la,\xi}\|_{C^M}\,d\xi
\lesssim_M 1,
\end{equation}
uniformly in $\la$, where the $C^M$ norm is taken in $(\eta,\tau,\tau',s',x',y)$, with $s'\sim \delta$.

For frozen $\xi$ (and $\ve$), let us then define
\begin{multline}\label{eq:Qklxi}
\breve Q^{k,\ell}_{\lambda,\xi}(x,u,s,x',u',s')
 := \lambda^{d+1}\,\chi_\varepsilon(x-x^0,u-u^0)\,\chi_\varepsilon(x'-x^0,u'-u^0)\\   \times \int_{\eta\sim 1} \int_{|\tau'|\sim 1} \int_{|\tau|\sim 1} \int_{\R^{d_1}} e^{i\lambda \eta \Psi} \, b(\eta,\tau,\tau',y,s)\\
\times H_{\la,\xi}(\eta,\tau,\tau',s',x',y)\,
 dy \, d\tau \, d\tau' \, d\eta, 
\end{multline}
so that
\[
\breve Q^{k,\ell,\varepsilon}_{\lambda}(x,u,s,x',u',s')=\int_\R \breve Q^{k,\ell}_{\lambda,\xi}(x,u,s,x',u',s') \, e^{i\la\xi(\psi(s')+u')} \, d\xi.
\]
For fixed $\xi$ and $x_1,x'_1$ frozen, let $S_{\lambda,\xi,x_1,x_1'}^{k,\ell}$ be the operator whose kernel is obtained from \eqref{eq:Qklxi} by replacing the two outer localization factors by $\chi_\varepsilon(x-x^0)\,\chi_\varepsilon(x'-x^0)$. The operators $S_{\lambda,\xi,x_1,x_1'}^{k,\ell}$ and $e^{i\la  \xi(\psi(s')+u')}\,S_{\lambda,\xi,x_1,x_1'}^{k,\ell}$ have the same operator norm on $L^2$. Thus, if
\begin{equation}\label{estxix1}
\|S_{\lambda,\xi,x_1,x_1'}^{k,\ell}\|_{L^2\to L^2}\lesssim \delta^{-1} \|\breve H_{\la,\xi}\|_{C^M},
\end{equation}
uniformly in $\xi,\la,$ etc., for some $M\in\N,$ then we can bound 
\[
\|\breve T_{\lambda,x_1}^{k,\ell,\varepsilon}
(\breve T_{\lambda,x_1'}^{k,\ell,\varepsilon})^*\|_{L^2\to L^2}\lesssim \delta^{-1}\int_\R \|\breve H_{\la,\xi}\|_{C^M}\, d\xi
\lesssim \delta^{-1},
\]
which will prove Proposition~\ref{lem:kl-L2}.

\subsubsection{Exploiting the partial Fourier transform in $u$}

Note the summand $u-u'$ in the phase function $\Psi,$ and let us change coordinates from $\eta$  to $\mu = \lambda\eta,$ so that $|\mu|\sim\la$. Then 
\begin{align*}
\int_{\R}f(\underline x',u',s') \, e^{i \lambda \eta(u-u')} \, d u' 
 = \scriptF_u f(\underline x', \lambda \eta,s') \, e^{i \lambda \eta u} 
 = f^{\mu}(\underline x',s') \, e^{i \mu u}
\end{align*}
Here,  $\scriptF_u $ denotes the partial Fourier transform in $u,$ in which we now have suppressed the factor $2\pi,$ as also for $f^\mu$.

For given $\xi$ and $\mu,$ let $S_{\lambda,\xi,\mu,x_1,x'_1}^{k,\ell}$ be the integral operator given by
\[
S_{\lambda,\xi,\mu,x_1,x'_1}^{k,\ell} \, g(\underline x,s) \\
 = \int_{\R^{d_1-1}} \int_0^\infty q^{k,\ell}_{\lambda,\xi,\mu,x_1,x'_1}(\underline x,s,\underline x',s') \, g(\underline x',s') \, d\underline x' \, ds',
\]
with integral kernel
\begin{multline*}
q^{k,\ell}_{\lambda,\xi,\mu,x_1,x'_1}
(\underline x,s,\underline x',s')
=
\chi_\varepsilon(x-x^0)\,
\chi_\varepsilon(x'-x^0)\\
\times \int_\R\,\int_\R\,\int_{\R^{d_1}}
e^{i\mu\Theta}\,
b(\tfrac\mu\la,\tau,\tau',y,s)\,
H_{\la,\xi}
(\tfrac\mu\la,\tau,\tau',s',x',y)\,
dy\,d\tau\,d\tau'.
\end{multline*}
 The phase $\Theta$ is given by
\[
\Theta
 = \frac{\delta^{-1}-1 + s}{\pi + \delta \varphi_\ell(\tau)} - \frac{\delta^{-1}-1 + s'}{\pi + \delta \varphi_\ell(\tau')} - \tau| x-y|^2 + \tau'| x'-y|^2 + \tfrac12 \langle rJ\Delta x,y\rangle.
\]
Then, by the change of coordinates from $\eta$ to $\mu,$ we see that  
\[
\big(S_{\lambda,\xi,x_1,x_1'}^{k,\ell}f\big)(\underline x,u,s)=\lambda^{d} \int_{\mu\sim \la} \big(S_{\lambda,\xi,\mu,x_1,x'_1}^{k,\ell} f^{\mu}\big)(\underline x,s) \, e^{i \mu u} \, d\mu. 
\]
Thus, by Plancherel in $u$, 
\[
\big\| S_{\lambda,\xi,x_1,x_1'}^{k,\ell} f \big\|_2 
\sim \lambda^d \Big( \int_{\mu\sim \la}  \, \big\| S_{\lambda,\xi,\mu,x_1,x'_1}^{k,\ell} f^{\mu}\big\|_2^2 \, d\mu\Big)^{\frac 1 2}.
\]
Suppose we can show that
\begin{equation}\label{Smuest}
\|S_{\lambda,\xi,\mu,x_1,x'_1}^{k,\ell} g \|_2 \lesssim \delta^{-1} \lambda^{-d} \|\breve H_{\la,\xi}\|_{C^M}\,\|g\|_2,
\end{equation}
uniformly for $\mu\sim \lambda$.
Then, again by Plancherel, 
\begin{align*}
\big\| S_{\lambda,\xi,x_1,x_1'}^{k,\ell} f \big\|_2 
& \lesssim \lambda^d \Big( \int_{\mu\sim \la}  \int_\R \, \int_{\R^{d_1-1}} \, \delta^{-2} \, \lambda^{-2d} \, \|\breve H_{\la,\xi}\|_{C^M}^2\, |f^{\mu}(\underline x,s)|^2 \, d \underline x \, d s \, d\mu \Big)^{\frac 1 2}\\
& \sim \delta^{-1} \, \|\breve H_{\la,\xi}\|_{C^M}\|f\|_2,
\end{align*} which would prove  \eqref{estxix1}, hence also Proposition~\ref{lem:kl-L2}.

\subsubsection{Scaling in $s$}\label{scaletrick}

We are thus left with showing estimate \eqref{Smuest}.
To this end, let $\breve S_{\lambda,\xi,\mu,x_1,x'_1}^{k,\ell}$ denote the scaled operator
\[
\breve S_{\lambda,\xi,\mu,x_1,x'_1}^{k,\ell} \, g(\underline x,s) \\
 = \int_{\R^{d_1-1}} \int_0^\infty \breve q^{k,\ell}_{\lambda,\xi,\mu,x_1,x'_1}(\underline x,s,\underline x',s') \, g(\underline x',s') \, d\underline x' \, ds',
\]
with integral kernel
\begin{multline*}
\breve q^{k,\ell}_{\lambda,\xi,\mu,x_1,x'_1}(\underline x,s,\underline x',s')
:=q^{k,\ell}_{\lambda,\xi,\mu,x_1,x'_1}(\underline x,\delta^{-1}s,\underline x',\delta^{-1}s')\\
= \chi_\varepsilon(x-x^0)\,
\chi_\varepsilon(x'-x^0)  \int_\R\, \int_\R\, \int_{\R^{d_1}}
e^{i\mu\breve\Theta}\,
\breve b(\tfrac\mu\la,\tau,\tau',y,s)\\
\times \breve H_{\la,\xi}
(\tfrac\mu\la,\tau,\tau',s',x',y)\,
\, dy \, d\tau \, d\tau'.
\end{multline*}
This kernel is obtained from $q^{k,\ell}_{\lambda,\xi,\mu,x_1,x'_1}$ by scaling $s$ and $s'$ by the factor $\delta^{-1}.$

The scaled phase function $\breve \Theta$ is given by 
\[
\breve\Theta
 = \frac {\delta^{-1}(s+1)-1}{\pi + \delta \varphi_\ell(\tau)} - \frac {\delta^{-1}(s'+1)-1}{\pi + \delta \varphi_\ell(\tau')} - \tau| x-y|^2 + \tau'| x'-y|^2 + \tfrac 12 \left\langle rJ\Delta x,y\right\rangle.
\]
Since
\[
\frac {\delta^{-1}(s+1)-1}{\pi + \delta \varphi_\ell(\tau)}-\frac {\delta^{-1}(s+1)-1}{\pi}=
-\frac {(1-\delta+s)\varphi_\ell(\tau)}{\pi(\pi + \delta \varphi_\ell(\tau))},
\]
we may write
\begin{align*}
\breve\Theta
& = \frac{s-s'}{\delta \pi}+\frac {(1-\delta+s')\varphi_\ell(\tau')}{\pi(\pi + \delta \varphi_\ell(\tau'))}-\frac {(1-\delta+s)\varphi_\ell(\tau)}{\pi(\pi + \delta \varphi_\ell(\tau))}\\
 &\hskip4cm - \tau| x-y|^2 + \tau'| x'-y|^2 + \tfrac 12 \left\langle rJ\Delta x,y\right\rangle.
\end{align*}
In order to remove the first summand in this phase, which has a possibly large coefficient $\delta^{-1},$ we pass to the modified operator $\tilde S_{\lambda,\xi,\mu,x_1,x'_1}^{k,\ell}$ defined by 
\begin{equation}\label{Smutildeb}
\tilde S_{\lambda,\xi,\mu,x_1,x'_1}^{k,\ell} \, g(\underline x,s) = e^{-i\frac{\mu s}{\delta \pi}} \,\breve S_{\lambda,\xi,\mu,x_1,x'_1}^{k,\ell}\big(e^{i\frac{\mu s'}{\delta \pi}} g\big)(\underline x,s).
\end{equation}
The operator $\tilde S_{\lambda,\xi,\mu,x_1,x'_1}^{k,\ell}$ is defined  in the same way as $\breve S_{\lambda,\xi,\mu,x_1,x'_1}^{k,\ell}$, but with the phase $\breve \Theta$ replaced by  the phase $\tilde \Theta$ given by
\begin{align*}
\tilde \Theta  =
\frac {(1-\delta+s')\varphi_\ell(\tau')}{\pi(\pi + \delta \varphi_\ell(\tau'))}-\frac {(1-\delta+s)\varphi_\ell(\tau)}{\pi(\pi + \delta \varphi_\ell(\tau))} 
 - \tau| x-y|^2 + \tau'| x'-y|^2 + \tfrac 12 \left\langle rJ\Delta x,y\right\rangle.
\end{align*}
Thus, 
\[
\tilde S_{\lambda,\xi,\mu,x_1,x'_1}^{k,\ell} g(\underline x,s) = \int_{\R^{d_1-1}} \int_0^\infty \tilde q^{k,\ell}_{\lambda,\xi,\mu,x_1,x'_1}
(\underline x,s,\underline x',s') \, g(\underline x',s') \,ds'\,d\underline x'.
\]
with
\begin{multline*}
\tilde q^{k,\ell}_{\lambda,\xi,\mu,x_1,x'_1}
(\underline x,s,\underline x',s')
=
\chi_\varepsilon(x-x^0)\,
\chi_\varepsilon(x'-x^0)\\
\times
\int_\R\,\int_\R\,\int_{\R^{d_1}}
e^{i\mu\tilde\Theta}\,
\breve b(\tfrac\mu\la,\tau,\tau',y,s)\,
\breve H_{\la,\xi}
(\tfrac\mu\la,\tau,\tau',s',x',y)\,
dy\,d\tau\,d\tau'.
\end{multline*}
Clearly, $\breve S_{\lambda,\xi,\mu,x_1,x'_1}^{k,\ell}$ and $\tilde  S_{\lambda,\xi,\mu,x_1,x'_1}^{k,\ell}$ have the same operator norms on $L^2,$ i.e.,
\[
\|\breve S_{\lambda,\xi,\mu,x_1,x'_1}^{k,\ell} \|_{2\to 2}=\|\tilde S_{\lambda,\xi,\mu,x_1,x'_1}^{k,\ell} \|_{2\to 2}.
\]
By means of the next lemma, we   shall conclude  that
\begin{equation}\label{redtobreve}
\| S_{\lambda,\xi,\mu,x_1,x'_1}^{k,\ell} \|_{2\to 2} \lesssim \delta^{-1} \|\breve  S_{\lambda,\xi,\mu,x_1,x'_1}^{k,\ell} \|_{2\to 2}.
\end{equation}
This reduces the estimate for $\|S_{\lambda,\xi,\mu,x_1,x'_1}^{k,\ell}\|_{2\to2}$ to that for $\|\tilde S_{\lambda,\xi,\mu,x_1,x'_1}^{k,\ell}\|_{2\to2}$.

Estimate \eqref{redtobreve} will follow easily from the  following general scaling result.
\begin{lemma}\label{scaleins}
Suppose $K=K(s,s')$ is a square integrable  integral kernel taking values in the space of bounded linear  operators $\scriptB(\scriptH)$ on a Hilbert space $\scriptH$, i.e., $K\in L^2(\R\times\R, \scriptB(\scriptH))$. Let $A$ be any measurable subset of $\R,$ and let $0<\delta<1$.   
Define the integral operators $T$ and $\breve T$ from $L^2(\R,\scriptH)$ to $L^2(\R,\scriptH)$ by 
$$
T f(s)=\bbone_{A}(s) \int_{\R} K(s,s')\, (\bbone_{A}\, f)(s')\, ds'
$$
and 
$$
\breve T f(s)= \int_{\R} K\Big(\frac s \delta,\frac {s'}\delta\Big) \, f(s')\, ds',
$$
for $f\in L^2(\R,\scriptH).$ Then
\[
\|T\|_{L^2(\R,\scriptH)\to L^2(\R,\scriptH)} \le \delta^{-1} \, \|\breve T\|_{L^2(\R,\scriptH)\to L^2(\R,\scriptH)}.
\]
\end{lemma}

\begin{proof} For $f\in L^2(\R,\scriptH)$, put $f_\delta(s)=f(s/\delta).$ Then 
\begin{align*}
(Tf)_\delta (s) &=\bbone_A\Big(\frac s \delta\Big) \int_\R K\Big(\frac s \delta,s'\Big) (\bbone_A\, f)(s')\, ds'\\
&=\delta^{-1} \, \bbone_{A}\Big(\frac s \delta\Big) \int_\R K\Big(\frac s \delta,\frac{s'}\delta\Big) (\bbone_{A}\, f)\Big(\frac{s'}\delta\Big)\, ds'\\
&= \delta^{-1} \, \bbone_{A}\Big(\frac s \delta\Big) \, \breve T \big((\bbone_{A})_\delta \, f_\delta\big)(s).
\end{align*}
Thus,
\[
\|(Tf)_\delta\|_{L^2(\R,\scriptH)}
\le
\delta^{-1} \, \|\breve T\|_{L^2(\R,\scriptH)\to L^2(\R,\scriptH)} \,
\|f_\delta\|_{L^2(\R,\scriptH)},
\]
from which our asserted estimate is immediate.
\end{proof}

To prove \eqref{redtobreve}, let $\scriptH=L^2(\R^{d_1-1})$, and let $K(s,s')$ be the operator on $\scriptH$ with integral kernel
\[
\big(K(s,s')\big)(\underline x,\underline x')
:= q^{k,\ell}_{\lambda,\xi,\mu,x_1,x'_1}(\underline x,s,\underline x',s').
\]
Then the operator $S_{\lambda,\xi,\mu,x_1,x'_1}^{k,\ell}$ corresponds to the operator-valued integral kernel
\[
\bbone_{[1,1+\ve]}(s) \, K(s,s') \, \bbone_{[1,1+\ve]}(s').
\]
Here, recall that we had implicitly  always assumed that $s,s'\in [1,1+\ve]$, but suppressed the localizing factors $\bbone_{[1,1+\ve]}(s)$ and $\bbone_{[1,1+\ve]}(s')$ for a while. Therefore, choosing  $A=[1,1+\ve]$ in Lemma~\ref{scaleins},  we see that estimate \eqref{redtobreve} is indeed an  immediate consequence of Lemma~\ref{scaleins}.  

\medskip

Finally, to estimate the operator $\tilde S_{\lambda,\xi,\mu,x_1,x'_1}^{k,\ell}$, recall that its integral kernel is given by the oscillatory integral
\begin{multline*}
\tilde q^{k,\ell}_{\lambda,\xi,\mu,x_1,x'_1}
(\underline x,s,\underline x',s')
=
\chi_\varepsilon(x-x^0)\,
\chi_\varepsilon(x'-x^0)\\
\times
\int_\R\,\int_\R\,\int_{\R^{d_1}}
e^{i\mu\tilde\Theta}\,
\breve b(\tfrac\mu\la,\tau,\tau',y,s)\,
\breve H_{\la,\xi}
(\tfrac\mu\la,\tau,\tau',s',x',y)\,
dy\,d\tau\,d\tau'.
\end{multline*}
The scaled amplitude
\[
\breve b(\tfrac\mu\la,\tau,\tau',y,s)\,
\breve H_{\la,\xi}
(\tfrac\mu\la,\tau,\tau',s',x',y)
\]
is now smooth, with uniform bounds on the derivatives of the first factor and with $\breve H_{\la,\xi}$ satisfying \eqref{tildeH}. The corresponding phase $\tilde\Theta$ is smooth and of order $\mathcal{O}(1)$, and the same is true of all its derivatives.

\begin{lemma}\label{eq:MA-kl}
If $\ve>0$ is chosen sufficiently small, then
\[
\left|
\,\det
\begin{pmatrix}
\tilde\Theta_{(\underline x,s),(\underline x',s')} & \tilde\Theta_{(\underline x,s),(\tau,\tau',y)}\\
\tilde\Theta_{(\tau,\tau',y),(\underline x',s')} & \tilde\Theta_{(\tau,\tau',y),(\tau,\tau',y)}\\
\end{pmatrix}
\right|
\sim 1,
\]
uniformly in $k\ge 1$ and $\ell\ge 2$.
\end{lemma}

\begin{proof}
We compute that 
\[
\begin{pmatrix}
\tilde\Theta_{(\underline x,s),(\underline x',s')} & \tilde\Theta_{(\underline x,s),(\tau,\tau',y)}\\
\tilde\Theta_{(\tau,\tau',y),(\underline x',s')} & \tilde\Theta_{(\tau,\tau',y),(\tau,\tau',y)}\\
\end{pmatrix}
 = 
\begin{pmatrix}
0 & 0 & \tilde\Theta_{\underline x\tau} & 0 & \tilde\Theta_{\underline xy} \\
0 & 0 & \tilde\Theta_{s\tau} & 0 & 0 \\
0 & 0 & \tilde\Theta_{\tau\tau} & 0 & \tilde\Theta_{\tau y} \\
\tilde\Theta_{\tau'\underline x'} & \tilde\Theta_{\tau's'} & 0 & \tilde\Theta_{\tau'\tau'} & \tilde\Theta_{\tau' y} \\
\tilde\Theta_{y\underline x'} & 0 & \tilde\Theta_{y\tau} & \tilde\Theta_{y\tau'} & \tilde\Theta_{yy} 
\end{pmatrix}.
\]
Recall that $r=2^{1-\ell}$ and $\varphi_\ell'(\tau)=-1/(\tau^2+r^2)$.
Moreover,
\[
\tilde\Theta_{s\tau}
=-\frac{\varphi_\ell'(\tau)}
{\big(\pi+\delta\varphi_\ell(\tau)\big)^2},
\]
so that 
$|\tilde\Theta_{s\tau}|\sim 1,$ and similarly $|\tilde\Theta_{\tau's'}|\sim 1.$
Using the $s'$-column and then the $s$-row for cofactor expansions, we thus see that the modulus of the determinant of the matrix above equals $|\tilde\Theta_{s\tau}| \, |\tilde\Theta_{\tau's'}|\sim 1$ times the modulus of
\[
\begin{vmatrix}
0 & 0 & \tilde\Theta_{\underline xy} \\
0 & 0 & \tilde\Theta_{\tau y} \\
\tilde\Theta_{y\underline x'}& \tilde\Theta_{y\tau'} & \tilde\Theta_{yy} 
\end{vmatrix}.
\]
The modulus of the above determinant equals $|\det( \tilde\Theta_{y,(\underline x',\tau')} )| \, |\det( \tilde\Theta_{(\underline x,\tau),y} )|$. Due to symmetry, it suffices to show that the latter determinant is comparable to 1. 
\smallskip

We write $\underline x = (x_2,\underline{\underline x})$ and $y = (y_1,y_2,\underline{\underline y}),$
and present  $J$  as a block matrix with respect to the coordinate blocks $(y_1,y_2,\underline{\underline y})$ in the form 
\[
J=\begin{pmatrix}
0 &  -1& 0 \\
1 & 0 &0\\
0 & 0& \underline{\underline J} \\
\end{pmatrix},
\]
where $\underline{\underline J}=J_{n-1}.$ 
It will be convenient to rather consider the matrix $E:=\tilde\Theta_{(\tau,\underline x),y}$ in place of $\tilde\Theta_{(\underline x,\tau),y}$, which, up to sign, has the same determinant. 
Then
 \[
E = 
\begin{pmatrix}
2(x_1-y_1) & 2(x_2-y_2) & 2(\underline{\underline x}-\underline{\underline y})\\ 
-\tfrac r2 & 2\tau &0\\
0  & 0& 2\tau I_{d_1-2} - \tfrac r2\underline{\underline J} \\
\end{pmatrix}.
\]
Since $|x_1-y_1|\sim 1$ and $|\underline x-\underline y|\lesssim \ve$, assuming that $\ve$ is sufficiently small we see that
\[
|\det E|\sim \begin{vmatrix}
 2\tau &0\\
0& 2\tau I_{d_1-2} - \tfrac r2\underline{\underline J} \\
\end{vmatrix}.
\]
Here $|\tau|\sim 1,$ and since $\underline{\underline J}$ is skew-symmetric, we see that
\[
|\det E|
\sim |2\tau|\, |\det (2\tau I_{d_1-2} - \tfrac r2\underline{\underline J})|\sim 1\cdot 1.
\]
Hence $|\det \tilde\Theta_{(\underline x,\tau),y} |\sim 1$.
Consequently, the Monge--Ampère matrix of $\tilde\Theta$ also has determinant comparable to $1$.
\end{proof}

Since the oscillatory integral operator $\tilde S_{\lambda,\xi,\mu,x_1,x'_1}^{k,\ell}$ acts on functions of $d_1$ space coordinates and is defined by means of $d_1+2$ frequency variables, Lemma~\ref{eq:MA-kl} implies that for a sufficiently large $N,$ 
\[
 \|\tilde S_{\lambda,\xi,\mu,x_1,x'_1}^{k,\ell}\|_{2\to 2}\lesssim \la^{-\frac 12\big(d_1+(d_1+2)\big)}\, \|\breve H_{\la,\xi}\|_{C^N}=\la^{-d}\, \\|\breve H_{\la,\xi}\|_{C^N},
\] 
uniformly for $k\ge 1$, $\ell\ge 2$, and $|\mu|\sim\la$ (see \cite{Hoermander-Acta1971} and \cite{GreenleafSeeger-Escorial}). 

In combination with \eqref{redtobreve}, this implies estimate \eqref{Smuest}. The proof of Proposition~\ref{lem:kl-L2} is thus complete.

\section{Estimates  for the horizontal parts given by  $\ell=1$}\label{CaseII}

In this case, we have $\de=\de(k,1)=k^{-1}$. Ignoring a factor $2^{-n},$   \eqref{eq:kernel-breve} shows that  in the renamed coordinates we have 
\[
\breve K_{\lambda,s}^{k,1}(x,u) 
= \lambda^{n + \frac 32}\,  \int_0^\infty \int_\R \, e^{i \lambda \sigma \breve \Phi_0} \,\tilde a_{\la,\de}(\sigma,s) \, \eta_1(\tau)   \sin (\tau)^{-n} \, d  \tau \, d\sigma,
\]
with phase 
\[
\breve \Phi_0 = \delta^{-1}-1+s -\left(k\pi + \tau\right)\Big(\delta |x|^2 \cot(\tau) - 
\frac {u}{k} -4\delta^{-1} u_c(\tau_{k,1})\Big).
\]
Here, $s$ is assumed to be in $|1,1+\ve].$ Recall that
\[
\supp\eta_1
\subset
\big[{-\tfrac58\pi},{-\tfrac3{16}\pi}\big]
\cup
\big[\tfrac3{16}\pi,\tfrac58\pi\big].
\]
Hence, the factor $(\sin \tau)^{-n}$ is harmless and can be hidden in the amplitude function $\eta_1$ by slightly modifying this function. Note that $\pm \tfrac \pi 2\in \pm [\frac 3{16}\pi,\frac 58\pi\big]$, and recall that the parameters $\tau=\pm \frac \pi 2$ (i.e., $k\pi\pm \frac12\pi$ in the original expression for $\smash{K_{\lambda,s}^{k,1}}$) correspond to horizontal points. Thus, in the sense of Subsection \ref{nh+h}, the contribution of each of the subintervals $\pm [\frac 3{16}\pi,\frac 58\pi\big]$ to $\smash{\breve K_{\lambda,s}^{k,1}}$ defines a ``horizontal part''.

\begin{remark}\label{hpoints}
Even though this seems to suggest that for any $k$ there are two different horizontal points associated with the $k$th ``zigzag'' of $\Sigma_s,$ there is in fact only one such horizontal point. Recall to this end  that  the cut-offs $\eta_0(\cdot -k)$ and $\eta_0(\cdot -(k+1))$ overlap, and note that $k\pi+\frac \pi 2= (k+1)\pi -\frac \pi 2.$
\end{remark}

We next decompose $\eta_1=\eta_{+}+\eta_{-},$ where the smooth bump functions $\eta_{\pm}$  are such that $\supp \eta_{\pm} \subset \pm [\frac 3{16}\pi, \frac 58\pi]$ and decompose accordingly
\[
\breve K_{\lambda,s}^{k,1}=\breve K_{\lambda,s}^{k,+}+\breve K_{\lambda,s}^{k,-},
\]
where 
\[
\breve K_{\lambda,s}^{k,\pm}(x,u)
 = \lambda^{n + \frac 32}\,  \int_0^\infty \int_\R \, e^{i \lambda \sigma \breve \Phi_0} \,\tilde a_{\la,\de}(\sigma,s) \, \eta_{\pm}(\tau) \, d  \tau \, d\sigma.
\]
The corresponding contributions  to $\breve T^{k,1,\varepsilon}_\lambda$ are denoted  by $\breve T^{k,\pm,\varepsilon}_\lambda$.

Changing  coordinates via $\tau=\pm \tfrac\pi2+\tilde\tau$, so that $\cot(\tau)=-\tan(\tilde \tau)$, and setting $\tilde \eta_{\pm}(\tilde \tau):=\eta_{\pm}(\pm\tfrac\pi 2+\tilde\tau)$, we may rewrite
\begin{equation}\label{eq:kernel-kpm}
\breve K_{\lambda,s}^{k,\pm}(x,u) = \lambda^{n + \frac 32}\,  \int_0^\infty \int_\R \, e^{i \lambda \sigma \tilde \Phi^\pm} \,\tilde a_{\la,\de}(\sigma,s) \, \tilde\eta_{\pm}(\tilde \tau)\, d  \tilde\tau \, d\sigma,
\end{equation}
with phase 
\begin{equation}\label{eq:Phi-pm}
\tilde \Phi^\pm = \delta^{-1}-1+s +\left(k\pi \pm \tfrac12\pi+\tilde\tau\right) \Big(\delta |x|^2 \tan(\tilde\tau)+\frac {u}{k} +4\delta^{-1} u_c(\tau_{k,1})\Big),
\end{equation}
where 
$$
\supp \tilde\eta_\pm\subset\pm \big[-\tfrac{5}{16}\pi,\tfrac{1}{8}\pi\big]\subset (-\tfrac{3}{8}\pi,\tfrac{3}{8}\pi)
$$
and $|\pm \tfrac12\pi+\tilde\tau|\le \tfrac 58 \pi$ on  the support.

\smallskip
Finally, changing coordinates from $\tilde \tau $ to  $\tau =\tan (\tilde\tau),$ and putting here 
$$\varphi_1(\tau):=\arctan{\tau},$$
we may write $\breve K_{\lambda,s}^{k,\pm}$ in a very similar way as in \eqref{eq:kernel-bar-5} as
\begin{equation}\label{eq:kernel-lis1}
\breve K_{\lambda,s}^{k,\pm}(x,u) = \lambda^{n + \frac 32}\,  \int_0^\infty \int_\R \, e^{i \lambda \sigma \breve \Phi^\pm} \,\tilde a_{\la,\de}(\sigma,s) \, \breve\eta_{\pm}(\tau)\, d  \tau \, d\sigma,
\end{equation}
with phase 
\begin{equation}\label{eq:Phi-breve1}
\breve \Phi^\pm := \delta^{-1}-1 + s + (\pi \pm \delta\tfrac{\pi} 2+\delta   \varphi_1(\tau)) \big( {\tau}|x|^2 +u
+ 4 k^2 u_c(\tau_{k,1})\big).
\end{equation}
Note that here 
\begin{equation}\label{etapm}
\supp \breve\eta_\pm\subset \tan \big([-\tfrac{3}{8}\pi,\tfrac{3}{8}\pi]\big),
\end{equation}
so that $|\varphi_1(\tau)|\le \tfrac 38 \pi <\frac\pi 2\le \pi \pm \delta\tfrac{\pi} 2$ on $ \supp \breve\eta_\pm$, and
\[
\varphi'_1(\tau)\sim 1 \quad\text{and}\quad
|\varphi^{(j)}_1(\tau)|\lesssim1 \text{ for } j\ge 2.
\]
To simplify the notation, let us put $\pi_\de:=\pi \pm \delta\tfrac{\pi} 2,$ so that 
\[
\breve K_{\lambda,s}^{k,\pm}(x,u)=\breve K_{\lambda,s}^k (x,u)
:= \lambda^{n + \frac 32}\,  \int_0^\infty \int_\R \, e^{i \lambda \sigma \breve \Phi} \,\tilde a_{\la,\de}(\sigma,s) \, \breve\eta_0(\tau)\, d  \tau \, d\sigma,
\]
with phase  
\[
\breve \Phi := \delta^{-1}-1 + s + (\pi_\delta+\delta  \varphi_1(\tau)) \big( {\tau}|x|^2 +u
+ 4 k^2 u_c(\tau_{k,1})\big)
\]
and $\breve\eta_0:=\breve\eta_\pm$. The  form of the  oscillatory integral  $\breve K_{\lambda,s}^k $ and its phase are  thus very similar to \eqref{eq:kernel-bar-5} and \eqref{eq:Phi-breve}, except that here the amplitude is no longer  supported where $|\tau| \sim 1$, but  satisfies \eqref{etapm} and thus includes $\tau=0$. However, the Jacobian factor arising from the change of variables $\tau=\tan(\tilde\tau)$ is comparable to 1 on the support of $\breve\eta_{0}$.

Even though $\breve K_{\lambda,s}^{k,1}$ is in fact the sum of the two terms corresponding to the signs $\pm,$ with  a slight abuse of notation we shall from here on simply assume that $\breve K_{\lambda,s}^{k,1}=\breve K_{\lambda,s}^k$, with the understanding that the following arguments are applied separately to each fixed sign. Note also that for $\ell=1$ we  are still working on the Heisenberg group $\mathbb H_n,$ since  here $2^{1-\ell}=1,$ so that we did not  perform an isotropic scaling but only a parabolic, hence automorphic, scaling by $1/k$.

\medskip

We next look at the analogue of the $TT^*$ argument from Subsection \ref{TT*}:

\begin{lemma}\label{lem:k1-TTstar} 
Let $\delta:=k^{-1}.$ The integral kernel $\breve Q^{k,1,\varepsilon}_{\lambda}$ of $\breve T^{k,1,\varepsilon}_\lambda(\breve T^{k,1,\varepsilon}_\lambda)^*$ is given by
\begin{multline}\label{eq:Qk1}
\breve Q^{k,1,\varepsilon}_{\lambda}(x,u,s,x',u',s')
 = \lambda^{d+1}\,\chi_\varepsilon(x-x^0,u-u^0)\,\chi_\varepsilon(x'-x^0,u'-u^0) \\ \times \int_{\eta\sim 1} \int_{|\tau'|\lesssim 1}\int_{|\tau|\lesssim 1} \int_{\R^{d_1}} e^{i\lambda \eta \Psi} a(\eta,\tau,\tau',y,s,x',u',s')\,dy \,d\tau\, d\tau'\, d\eta, 
\end{multline}
with phase function
\begin{multline}
\Psi = \frac{\delta^{-1}-1 + s}{\pi_\delta + \delta \varphi_1(\tau)}
 - \frac{\delta^{-1}-1 + s'}{\pi_\delta+ \delta \varphi_1(\tau')} 
 +\tau | x-y|^2 - \tau'| x'-y|^2  \\
 +( u - u' + \tfrac 1 2 \langle J(x-x'),y\rangle) \label{eq:Psi-k1}
\end{multline}
and amplitude
\begin{multline}\label{eq:k1-TTstar-amplitude}
a(\eta,\tau,\tau',y,s,x',u',s')
=a_{1,\la,\de}(\eta,\tau,s)\,\chi_\varepsilon(y)\,\breve\eta_0(\tau)\,\breve\eta_0(\tau') \\
\times \frac{h_\lambda(\eta,\tau,\tau',s',x',u',y)}{\big(\pi_\delta+\delta\varphi_1(\tau)\big) \big(\pi_\delta+\delta\varphi_1(\tau')\big)},
\end{multline}
where
\begin{equation}\label{eq:k1-TTstar-a}
a_{1,\la,\de}(\eta,\tau,s)
= a_\la\Big(\big(1+\delta(s-1)\big)
\frac{\eta}{\pi_\delta+\delta\varphi_1(\tau)}\Big)
\end{equation}
and
\begin{equation}\label{eq:k1-hlambda}
h_\lambda(\eta,\tau,\tau',s',x',u',y)
=\int_{\R}
e^{i\lambda\xi\Psi_2}
\lambda\,\widehat{\chi_\varepsilon}(\lambda\xi)
\,\overline{a_{1,\la,\de}(\eta-\xi,\tau',s')}
\,d\xi,
\end{equation}
with
\begin{equation}\label{eq:k1-Psi2}
\Psi_2
= \frac{\delta^{-1}-1+s'}{\pi_\delta+\delta\varphi_1(\tau')}
+\tau'|x'-y|^2+4k^2u_c(\tau_{k,1})+u'
+\tfrac12\langle Jx',y\rangle.
\end{equation}
The amplitudes $a$ and $h_\lambda$, together with all their partial derivatives, are uniformly bounded on the support of the integral, with bounds depending only on finitely many derivatives of $a_\lambda$.
\end{lemma}

We shall skip the proof, since it is obtained by following almost verbatim the proof of Lemma~\ref{lem:kl-TTstar}, up to the indicated minor modifications.

\medskip

Next, in contrast to the case $\ell\ge2$, we shall here freeze the variable $x_2$. Thus, given $x,x'\in\R^{d_1}$, we set
\[
\underline x:=(x_1,x_3,\dots,x_{d_1})
\]
and define, for fixed $x_2\in\R$, the operator
$\breve T_{\lambda,x_2}^{k,1,\varepsilon}$ by
\[
(\breve T_{\lambda,x_2}^{k,1,\varepsilon}\varphi)(\underline x,u,s) = (\breve T^{k,1,\varepsilon}_\lambda \varphi)(x,u,s).
\]
If we write $f(x',u',s' ) = f_{x_2'}(\underline x',u',s')$, then
\[
\big(\breve T_\lambda^{k,1,\varepsilon} (\breve T_\lambda^{k,1,\varepsilon})^* f\big)(x, u,s)
= \int_\R \, \big( \breve T_{\lambda,x_2}^{k,1,\varepsilon} (\breve T_{\lambda,x_2'}^{k,1,\varepsilon} )^* \, f_{x_2'}\big) (\underline x,u,s) \, dx_2'.
\]
Note that the integral kernel $\breve Q_{\lambda,x_2,x_2'}^{k,1,\ve}$ of $\breve T_{\lambda,x_2}^{k,1,\varepsilon} (\breve T_{\lambda,x_2'}^{k,1,\varepsilon} )^*$ is given by $\breve Q_\lambda^{k,1,\ve}(x_2,\cdot,x_2',\cdot)$.

\smallskip

In analogy with Propositions \ref{lem:kl-L1} and \ref{lem:kl-L2}, we shall here prove the following propositions:
\begin{proposition}\label{lem:k1-L1}
We have
\[
\|\breve T_{\lambda,x_2}^{k,1,\varepsilon} (\breve T_{\lambda,x_2'}^{k,1,\varepsilon} )^* \|_{L_{(\underline x,u)}^{1}(L^2_s)\to L_{(\underline x,u)}^\infty(L^2_s)}
\lesssim \lambda^{d-1} \, \langle\lambda\left|x_2'-x_2\right|\rangle^{-\frac{d_1-1}{2}} .
\]
\end{proposition}

\begin{proposition}\label{lem:k1-L2}
We have
\[
\|\breve T_{\lambda,x_2}^{k,1,\varepsilon} (\breve T_{\lambda,x_2'}^{k,1,\varepsilon} )^*\|_{L^2(\R^{d})\to L^2(\R^{d})} \lesssim \delta^{-1}.
\]
\end{proposition}

The square function estimate in Proposition~\ref{prop:unit}  follows from these two propositions by the same arguments as in the proof of Proposition~\ref{L:T0lambda}. The proofs of  Propositions \ref{lem:k1-L1} and \ref{lem:k1-L2} will be given in Sections \ref{sec:pointwise-inter1} and \ref{eq:inter-L21}.

\subsection{Pointwise estimates for $k\ge 1$ and $\ell=1$}\label{sec:pointwise-inter1}

We follow again the strategy of the proofs of Propositions \ref{lem:0-L1} and \ref{lem:kl-L1} by establishing analogous pointwise estimates as in Proposition~\ref{lem:pointwiseQkl}. Since the following arguments are very similar to those of Sections \ref{pointwise0} and \ref{sec:pointwise-inter}, we keep it again a bit more sketchy and mainly focus on the main contributions of the oscillatory integral.

\subsubsection{Freezing $\tau'$} 

By Lemma~\ref{lem:k1-TTstar}, we have
\[
\breve Q^{k,1,\varepsilon}_{\lambda}
 =  \int_{|\tau'|\lesssim 1} \breve Q^{k,1,\varepsilon}_{\lambda,\tau'} \, d\tau' ,
\]
where, suppressing again the localizing factors $\chi_\varepsilon(x-x^0,u-u^0)\,\chi_\varepsilon(x'-x^0,u'-u^0)$,
\[
\breve Q^{k,1,\varepsilon}_{\lambda,\tau'}
 = \lambda^{d+1} \int_{\eta\sim 1} \int_{|\tau|\lesssim 1} \int_{\R^{d_1}} e^{i\lambda \eta \Psi} a(\eta,\tau,\tau',y,s,x',u',s')\,dy \, d\tau\, d\eta, 
\]
with phase function given by \eqref{eq:Psi-k1}, that is,
\begin{align}\label{eq:ell1-Psi}
\Psi = \frac{\delta^{-1}-1 + s}{\pi_\delta + \delta \varphi_1(\tau)}
 - \frac{\delta^{-1}-1 + s'}{\pi_\delta + \delta \varphi_1(\tau')} 
 & + \tau\,| x-y|^2 - \tau' | x'-y|^2  \\
 & + ( u - u' + \tfrac 1 2 \langle J(x-x'),y\rangle) , \notag
\end{align}
where $\delta=k^{-1}$ and $\varphi_1(\tau)=\arctan \tau$.
In the following, we assume that $\tau'$ is frozen.

\subsubsection{Changing variables and stationary phase for a spherical integration}

In the oscillatory integral above, we change variables $y= x +y'$ and write $y'=\rho\omega$, where $\rho\sim 1$ and $\omega\in S^{d_1-1}$. Then, using $| x -y|^2=|y'|^2=\rho^2$ and
\[
| x'-y|^2
= |\Delta x + y'|^2
= |\Delta x|^2 + \rho^2 + 2\rho\langle \Delta x,\omega\rangle,
\qquad \Delta x = x-x',
\]
the phase \eqref{eq:ell1-Psi} becomes
\begin{align}\label{eq:Psi1-k1-ell1}
\Psi_1
=
\frac{\delta^{-1}-1+s}{\pi_\delta+\delta\varphi_1(\tau)}
-\frac{\delta^{-1}-1+s'}{\pi_\delta+\delta\varphi_1(\tau')}
+(\tau-\tau')\rho^2
& -\tau' |\Delta x |^2 \\
& +\rho\langle\theta,\omega\rangle
+\Delta w, \notag
\end{align}
where $\theta = \tfrac12J\Delta x -2\tau' \Delta x$ and $\Delta w = u-u' + \tfrac 12\langle J\Delta x,x\rangle$.
Note that
\begin{equation}\label{eq:theta-size}
|\theta|^2=\left(\tfrac14+4\,(\tau')^2\right)|\Delta x |^2\sim |\Delta x |^2.
\end{equation}
If we absorb the Jacobian factor $\rho^{d_1-1}$ and all cutoffs into a new amplitude function $a_1$, we are reduced to an oscillatory integral of the form
\begin{equation}\label{eq:Qtauprime-prebeta}
\lambda^{d+1}\int_{\eta\sim1}\int_{|\tau|\lesssim 1}\int_{\rho\sim1}\int_{S^{d_1-1}}
e^{i\lambda\eta \Psi_1}\, a_1 \,d\omega\,d\rho\,d\tau \,d\eta.
\end{equation}
We now change coordinates from $\tau$ to $\beta$ by
\[
\tau = \tau' +\beta|\Delta x|.
\]
Note that $|\beta \Delta x|\lesssim 1$ on the support of the oscillatory integral above. Hence, we are reduced to an oscillatory integral of the form
\begin{equation}\label{eq:Qbeta-ell1}
\lambda^{d+1}\,|\Delta x|
\int_{\eta\sim1}\int_{\rho\sim1}\int_{S^{d_1-1}}\int_{\R}
e^{i\lambda\eta \Psi_2 }\,
  a_2\, \chi_0(\beta|\Delta x|)\,
d\beta\,d\omega\,d\rho\,d\eta,
\end{equation}
where $\chi_0$ is a smooth cutoff supported where $|\beta\Delta x|\lesssim 1$, and $\Psi_2 $ is given by
\begin{align}\label{eq:Psi2-k1-ell1}
\Psi_2 
 = \frac{\delta^{-1}-1+s}{\pi_\delta+\delta\varphi_1(\tau'+\beta|\Delta x|)}
 -\frac{\delta^{-1}-1+s'}{\pi_\delta+\delta\varphi_1(\tau')} 
 +\beta\,|\Delta x|\,\rho^2
 &- \tau' |\Delta x|^2 \\
 &+ \rho \langle  \theta , \omega\rangle
 + \Delta w. \notag
\end{align}
Note that the vector $\theta$ is independent of $\beta$. For the region where $|\beta|\gg 1$, we can integrate by parts in $\rho$, so this part of the integral is an error term. The main contribution is given by the region where $|\beta|\lesssim 1$. Thus, since also $|\Delta x|\lesssim \varepsilon$, the integral \eqref{eq:Qbeta-ell1} reduces to 
\begin{equation}\label{eq:Qbeta-ell1-main}
\lambda^{d+1}\,|\Delta x|
\int_{\eta\sim1}\int_{\rho\sim1}\int_{S^{d_1-1}}\int_{|\beta|\lesssim 1}
e^{i\lambda\eta \Psi_2 }\,
  \tilde a_2\,\chi_\varepsilon(\beta|\Delta x|)\,
d\beta\,d\omega\,d\rho\,d\eta.
\end{equation}
Next, we focus on the integration in $\omega$. Recall that $y=x+y'$ and $y\in \supp \chi_\varepsilon$ in our original coordinates. Thus, we have $|\rho\omega+x|\lesssim \varepsilon$, which localizes $\omega$ to an $O(\varepsilon)$ neighborhood of $-x^0/|x^0|$, whose first component dominates all others. Writing $\theta=(\theta_1,\underline \theta)$, for the region where $|\underline{\theta}|\gg \varepsilon\,|\theta_1(\tau')|$, we can argue via integration by parts, which only contributes to the error term. For $|\underline{\theta}|\lesssim \varepsilon\,|\theta_1(\tau')|$, we perform stationary phase in $\omega$ which reduces \eqref{eq:Psi2-k1-ell1} to an oscillatory integral of the form
\begin{equation}\label{eq:Qbeta-ell1-stat}
\lambda^{d+1}\,|\Delta x|
\int_{\eta\sim1}\int_{\rho\sim1}\int_{|\beta|\lesssim 1} \langle \lambda\eta\rho|\theta| \rangle^{-\frac{d_1-1}{2}}
e^{i\lambda\eta \Psi_3^{\pm}}\,
   a_3\, \chi_\varepsilon(\beta|\Delta x|)\,
d\beta\,d\rho\,d\eta,
\end{equation}
where the factor $\langle \lambda\eta\rho|\theta| \rangle^{-\frac{d_1-1}{2}}$ actually denotes a symbol of order $-\frac{d_1-1}{2}$, $a_3$ is a smooth amplitude satisfying the same symbol bounds as $\tilde a_2$, and
\begin{align}\label{eq:Psi3-k1-ell1}
\Psi_3^{\pm}
 = \frac{\delta^{-1}-1+s}{\pi_\delta+\delta\varphi_1(\tau'+\beta|\Delta x|)}
 -\frac{\delta^{-1}-1+s'}{\pi_\delta+\delta\varphi_1(\tau')} 
 +\beta\,|\Delta x|\,\rho^2
& - \tau' |\Delta x|^2 \\
& \pm \rho|\theta|
 + \Delta w, \notag
\end{align}
where the sign $\pm$ depends on the stationary point $\pm\theta/|\theta|$ which is contained in the $\mathcal O(\varepsilon)$ cap.

\subsubsection{Stationary phase for the integration in $\rho$}

Next, note that
\[
\psi^{\pm}=\beta\,|\Delta x|\,\rho^2 \pm |\theta|\,\rho
= \rho^2\,\psi_2 + \rho\,\psi_1^{\pm}
\]
collects all the terms in \eqref{eq:Psi3-k1-ell1} that depend on $\rho$. Here $|\psi_2|\sim |\beta\Delta x|$, and, by \eqref{eq:theta-size}, we have $|\psi_1^{\pm}|\sim |\theta|\sim |\Delta x|$. Thus, if $|\beta|\ll 1$, then, for $\rho\sim 1$, we have $|\partial_\rho\psi^{\pm}|\gtrsim |\Delta x|$, and we can integrate by parts in $\rho$, whence this region again only contributes to the error term. Hence \eqref{eq:Qbeta-ell1-stat} reduces to an integral of the form
\begin{equation}\label{eq:ell1-rhoSP-pre}
\lambda^{d}(\lambda|\Delta x|)\int_{\eta\sim 1}\int_{\rho\sim 1}\int_{|\beta|\sim 1}
\langle \lambda |\Delta x|\rangle^{-\frac{d_1-1}{2}}
e^{i\lambda\eta \Psi_3^{\pm}}\,
\tilde a_3\, d\beta\, d\rho\, d\eta,
\end{equation}
where $\tilde a_3$ satisfies the same symbol bounds as $a_3$. Next, consider again
\[
\partial_\rho\psi^{\pm} = 2\beta|\Delta x|\rho \pm |\theta|.
\]
Since $|\beta|\sim 1$, $\psi^{\pm}$ has a unique critical point $\rho_c\sim 1$ in $\rho$ if and only if $\mp\beta\sim 1$. In the following, we consider only the phase $\Psi_3^{-}$ (the case of $\Psi_3^{+}$ can be treated analogously). If $\beta\sim -1$, then $\psi^{-}$ has no critical point on $\rho\sim 1$ and we can again integrate by parts in $\rho$, showing that the corresponding contribution is an error term. For the main contribution where $\beta\sim 1$, the function $\psi^{-}$ does have a critical point for any such $\beta$, which is given by
\[
\rho_c=\frac{|\theta|}{2\beta|\Delta x|}\sim 1.
\]
Moreover, note that $|\partial_\rho^2\psi^{-}|=2\beta|\Delta x|\sim |\Delta x|$ for $\beta \sim 1$. Thus, applying stationary phase in $\rho$ reduces us to the oscillatory integral
\begin{equation}\label{eq:ell1-rhoSP}
\lambda^{d}(\lambda|\Delta x|)\int_{\eta\sim 1}\int_{\beta\sim 1}
\langle \lambda|\Delta x|\rangle^{-\frac{d_1}{2}}
e^{i\lambda\eta \Psi_4 }\,
a_4 \, d\beta\, d\eta,
\end{equation}
where $a_4$ is a smooth amplitude satisfying the same symbol bounds as $\tilde a_3$, and
\begin{multline}\label{eq:ell1-Psi4}
\Psi_4 
=
\frac{\delta^{-1}-1+s}{\pi_\delta+\delta\varphi_1(\tau'+\beta|\Delta x|)}
-\frac{\delta^{-1}-1+s'}{\pi_\delta+\delta\varphi_1(\tau')}
-\tau' |\Delta x|^2
-\frac{|\theta|^2}{4\beta|\Delta x|}
+\Delta w.
\end{multline}

\subsubsection{Stationary phase for the integration in $\beta$}
Note that
\[
\partial_\beta \Psi_4 = |\Delta x|\, H(\beta,s,\tau', |\Delta x|),
\]
where, recalling that $|\theta|^2=(\frac14+4\,(\tau')^2)|\Delta x |^2$ by \eqref{eq:theta-size},
\[
H = -(\delta^{-1}-1+s)\,
\frac{\delta \varphi_1'(\tau'+\beta|\Delta x|)}
{(\pi_\delta+\delta\varphi_1(\tau'+\beta|\Delta x|))^2}
+\frac{\frac14+4\,(\tau')^2}{4\beta^2}.
\]
Thus, for fixed $\tau_0'$ and $s=1$, $|\Delta x|=0$, we obtain
\[
H(\beta,1,\tau_0',0)
=
-\frac{\frac1{1+(\tau_0)^2}}{(\pi_\delta+\delta \arctan \tau_0')^2}
+\frac{\frac14+4\,(\tau_0')^2}{4\beta^2}.
\]
Note that this function has a unique zero $\beta_c(1,\tau_0',0)\sim 1$. On the other hand,
\[
\partial_\beta H  = -\frac{\frac14+4\,(\tau')^2}{2\beta^3}+\mathcal O(|\Delta x|).
\]
Thus, $|\partial_\beta H| \sim 1$ for $|\beta|\sim 1$, and the implicit function theorem implies that for $|s-1|$, $|\Delta x|$, and $|\tau'-\tau_0'|$ sufficiently small, there is a unique zero $\beta_c=\beta_c(s,\tau',|\Delta x|)$ for $H$. Note that we can assume that $\tau'$ lies in an interval $[\tau_0',\tau_0'+\varepsilon]$ by decomposing the $\tau'$ support into $\mathcal O(\varepsilon^{-1})$ many subintervals.

Since $\partial_\beta^2\Psi_4 (\beta)\sim |\Delta x|$ for $|\beta|\sim 1$, applying stationary phase in $\beta$, we see that \eqref{eq:ell1-rhoSP} reduces to an oscillatory integral of the form
\begin{equation}\label{eq:ell1-betaSP-post}
\lambda^{d}(\lambda|\Delta x|)\int_{\eta\sim 1}
\langle \lambda|\Delta x|\rangle^{-\frac{d_1+1}{2}}
e^{i\lambda\eta \Psi_5}\,
a_5\, d\eta,
\end{equation}
where $a_5$ is a smooth amplitude satisfying the same symbol bounds as $a_4$, and the new phase is given by $\Psi_5=\Psi_4 |_{\beta=\beta_c}$.

Finally, note that
\[
\frac{\delta^{-1}-1+s}{\pi_\delta+\delta\varphi_1(\tau'+\beta_c|\Delta x|)}
-\frac{\delta^{-1}-1+s'}{\pi_\delta+\delta\varphi_1(\tau')}
 = \frac{\Delta s}{\pi_\delta+\delta\varphi_1(\tau')} + \mathcal O(|\Delta x|).
\]
Thus, integration by parts in $\eta$ shows that we arrive at an analogous estimate to that in Proposition~\ref{lem:pointwiseQkl} for the main contribution of the oscillatory integral.

\subsection{Refined $L^2$ estimates for $k\ge 1$ and $\ell=1$}\label{eq:inter-L21}

In this section we  prove Proposition~\ref{lem:k1-L2}. We shall closely follow the arguments  from Section \ref{eq:inter-L2}.

\subsubsection{ Freezing $\xi$ in the defining integral for $h_\la$} 
Write again
\[
\Psi_2=\psi(s')+u'+\dot\Psi_2,
\qquad
\psi(s')=\frac{\delta^{-1}-1+s'}{\pi_\delta},
\]
where now
\[
\dot\Psi_2
= -\frac{1+\delta(s'-1)}{\pi_\delta+\delta\varphi_1(\tau')}
\frac{\varphi_1(\tau')}{\pi_\delta}
+\tau'|x'-y|^2+4k^2u_c(\tau_{k,1})
+\tfrac12\langle Jx',y\rangle.
\]
Writing the amplitude in \eqref{eq:k1-TTstar-amplitude} as
\[
a(\eta,\tau,\tau',y,s,x',u',s')=b(\eta,\tau,\tau',y,s)\,
 h_\lambda(\eta,\tau,\tau',s',x',u',y),
\]
where
\[
b(\eta,\tau,\tau',y,s):=a_{1,\la,\de}(\eta,\tau,s)\,\chi_\varepsilon(y)\,
\frac{\breve\eta_0(\tau)\,\breve\eta_0(\tau')}
{\big(\pi_\delta+\delta\varphi_1(\tau)\big)
 \big(\pi_\delta+\delta\varphi_1(\tau')\big)},
\]
we decompose
\[
h_\lambda(\eta,\tau,\tau',s',x',u',y) =\int_\R e^{i\la\xi(\psi(s')+u')}
H_{\la,\xi}(\eta,\tau,\tau',s',x',y)\,d\xi,
\]
where
\[
H_{\la,\xi}(\eta,\tau,\tau',s',x',y)
:=e^{i\la\xi\dot\Psi_2}\,
\lambda\,\widehat\chi_\varepsilon(\lambda\xi)\,
\overline{a_{1,\la,\de}(\eta-\xi,\tau',s')}.
\]
For frozen $\xi$ (and $\ve$), we define
\begin{multline}
\breve Q^{k,1}_{\lambda,\xi}(x,u,s,x',u',s')
:=\lambda^{d+1}\,\chi_\varepsilon(x-x^0,u-u^0)\,
\chi_\varepsilon(x'-x^0,u'-u^0)\\
\times\int_{\eta\sim1}\int_{|\tau'|\lesssim1}\int_{|\tau|\lesssim1}
\int_{\R^{d_1}}e^{i\lambda\eta\Psi}\,
 b(\eta,\tau,\tau',y,s)\\
\times H_{\la,\xi}(\eta,\tau,\tau',s',x',y)\,
 dy\,d\tau\,d\tau'\,d\eta. \label{eq:Qk1xi}
\end{multline}
Arguing as in the case $\ell\ge 2,$ in order to prove Proposition~\ref{lem:k1-L2} it will then suffice to bound the operator $S_{\lambda,\xi,x_2,x_2'}^{k,1}$ with integral kernel $\breve Q_{\lambda,\xi,x_2,x_2'}^{k,1}$ by
\begin{equation}\label{estxix11}
\|S_{\lambda,\xi,x_2,x_2'}^{k,1}\|_{L^2\to L^2}\lesssim \delta^{-1} \|\breve H_{\la,\xi}\|_{C^M}
\end{equation}
for some $M\in \N,$ uniformly in $\xi,\la,$ etc., where again
\[
\breve H_{\la,\xi}(\eta,\tau,\tau',s',x',y)
:=H_{\la,\xi}(\eta,\tau,\tau',\delta^{-1}s',x',y).
\]

\subsubsection{ Exploiting the partial Fourier transform in $u$}
Following the arguments in Section \ref{eq:inter-L2}, for fixed $\xi$ and $\mu$ of size $\mu\sim \la,$ we define $S_{\lambda,\xi,\mu,x_2,x'_2}^{k,1}$   to be the integral operator  given by
\[
S_{\lambda,\xi,\mu,x_2,x'_2}^{k,1} \, g(\underline x,s) \\
 = \int_{\R^{d_1-1}} \int_0^\infty q^{k,1}_{\lambda,\xi,\mu,x_2,x'_2}(\underline x,s,\underline x',s') \, g(\underline x',s') \, d\underline x' \, ds',
\]
with integral kernel
\begin{multline*}
q^{k,1}_{\lambda,\xi,\mu,x_2,x_2'}
(\underline x,s,\underline x',s')
=\chi_\varepsilon(x-x^0)\,\chi_\varepsilon(x'-x^0)\\
\times\int_{|\tau'|\lesssim 1}\int_{|\tau|\lesssim 1}\int_{\R^{d_1}}
 e^{i\mu\Theta}\,
 b(\tfrac\mu\la,\tau,\tau',y,s)\\
\times H_{\la,\xi}(\tfrac\mu\la,\tau,\tau',s',x',y)\,
 dy\,d\tau\,d\tau',
\end{multline*}
where the phase $\Theta$ is given by
\[
\Theta
 = \frac{\delta^{-1}-1 + s}{\pi_\delta + \delta \varphi_1(\tau)}
 - \frac{\delta^{-1}-1 + s'}{\pi_\delta + \delta \varphi_1(\tau')} +\tau| x-y|^2 -  \tau'| x'-y|^2 \notag \\
 + \tfrac 1 2 \langle J \Delta x,y\rangle. 
\]

\subsubsection{Scaling in $s$}

As before, let $\breve S_{\lambda,\xi,\mu,x_2,x_2'}^{k,1}$ denote the integral operator whose integral kernel is obtained from that of
$S_{\lambda,\xi,\mu,x_2,x_2'}^{k,1}$ by replacing $s$ and $s'$ with
$\delta^{-1}s$ and $\delta^{-1}s'$, respectively. Its phase can be written as
\begin{align*}
\breve\Theta
=\frac{s-s'}{\delta\pi_\delta}
+\frac{(1-\delta+s')\varphi_1(\tau')}
{\pi_\delta(\pi_\delta+\delta\varphi_1(\tau'))}
-\frac{(1-\delta+s)\varphi_1(\tau)}
{\pi_\delta(\pi_\delta+\delta\varphi_1(\tau))}
& +\tau|x-y|^2\\
& -\tau'|x'-y|^2 
+\tfrac12\langle J\Delta x,y\rangle.
\end{align*}
By replacing  this phase by the one in which the first term is removed, i.e., by
\begin{align*}
\tilde\Theta
=\frac{(1-\delta+s')\varphi_1(\tau')}
{\pi_\delta(\pi_\delta+\delta\varphi_1(\tau'))}
-\frac{(1-\delta+s)\varphi_1(\tau)}
{\pi_\delta(\pi_\delta+\delta\varphi_1(\tau))}
& + \tau|x-y|^2 \\ 
& - \tau'|x'-y|^2
+\tfrac12\langle J\Delta x,y\rangle,
\end{align*}
we obtain the related operator $\tilde S_{\lambda,\xi,\mu,x_2,x'_2}^{k,1}.$ 

By means of a similar conjugation trick  as in \eqref{Smutildeb} and the   rescaling trick based on Lemma~\ref{scaleins} in Subsection \ref{scaletrick}, we can finally reduce the proof of estimate \eqref{estxix11} to showing
\begin{equation}\label{Smuest1}
\|\tilde S_{\lambda,\xi,\mu,x_2,x'_2}^{k,1} \|_{2\to 2}  \lesssim \lambda^{-d}  \|\breve H_{\la,\xi}\|_{C^M},
\end{equation}
uniformly in $\xi$ and $\mu\sim \lambda$.

The scaled amplitude of both operators $\breve S_{\lambda,\xi,\mu,x_2,x_2'}^{k,1}$ and
$\tilde S_{\lambda,\xi,\mu,x_2,x_2'}^{k,1}$ is given by
\[
\breve b(\tfrac\mu\la,\tau,\tau',y,s)\,
\breve H_{\la,\xi}(\tfrac\mu\la,\tau,\tau',s',x',y),
\]
where $\breve b(\eta,\tau,\tau',y,s):=b(\eta,\tau,\tau',y,\delta^{-1}s)$. The amplitude is smooth, with uniform bounds on the derivatives of
$\breve b$, while $\breve H_{\la,\xi}$ satisfies the estimates in \eqref{tildeH}. The phase $\tilde\Theta$ is smooth and of order $\mathcal O(1)$, and the same is true of all its derivatives.

Thus, the proof of estimate \eqref{Smuest1} will be an immediate consequence of the following analogue of Lemma~\ref{eq:MA-kl}.

\begin{lemma}\label{eq:MA-k1}
If $\ve>0$ is chosen sufficiently small, then
\[
\left|
\,\det
\begin{pmatrix}
\tilde\Theta_{(\underline x,s),(\underline x',s')} & \tilde\Theta_{(\underline x,s),(\tau,\tau',y)}\\
\tilde\Theta_{(\tau,\tau',y),(\underline x',s')} & \tilde\Theta_{(\tau,\tau',y),(\tau,\tau',y)}\\
\end{pmatrix}
\right|
\sim 1,
\]
uniformly in $k\ge1$.
\end{lemma}

\begin{proof}
The matrix above equals
\[
\begin{pmatrix}
0 & 0 &\tilde\Theta_{\underline x\tau}&0&
   \tilde\Theta_{\underline x y}\\
0 &  0&\tilde\Theta_{s\tau}&0&0\\
0 & 0 & \tilde\Theta_{\tau\tau}&0&
   \tilde\Theta_{\tau y}\\
\tilde\Theta_{\tau'\underline x'}&
   \tilde\Theta_{\tau's'}& 0&
   \tilde\Theta_{\tau'\tau'}&
   \tilde\Theta_{\tau' y}\\
\tilde\Theta_{y\underline x'} & 0&
   \tilde\Theta_{y\tau}&
   \tilde\Theta_{y\tau'}&
   \tilde\Theta_{yy}
\end{pmatrix}.
\]
Note that $\varphi_1'(\tau)=1/(1+\tau^2)\sim1$ on the support of $\breve\eta_0$. We have
\[
\tilde\Theta_{s\tau}
=-\frac{\varphi_1'(\tau)}{(\pi_\delta+\delta\varphi_1(\tau))^2},
\qquad
\tilde\Theta_{\tau's'}
=\frac{\varphi_1'(\tau')}{(\pi_\delta+\delta\varphi_1(\tau'))^2},
\]
and therefore $|\tilde\Theta_{s\tau}|\sim1$ and
$|\tilde\Theta_{\tau's'}|\sim1$.
By cofactor expansion, as in the first part of the proof of
Lemma~\ref{eq:MA-kl}, we are reduced to considering the matrix 
\[
\begin{pmatrix} 0&0&\tilde\Theta_{\underline x y}\\
0&0&\tilde\Theta_{\tau y}\\
\tilde\Theta_{y\underline x'}& \tilde\Theta_{y\tau'}& \tilde\Theta_{yy}
\end{pmatrix}.
\]
By symmetry, we are thus reduced to showing that
\[
|\det(\tilde\Theta_{(\underline x,\tau),y})|\sim1.
\]
To this end, it will be convenient to flip  the first two coordinates of $x$ and write $x=(x_2,x_1,\underline{\underline x})$ and $\underline x = (x_1,\underline{\underline x}),$ and accordingly $y = (y_2,y_1,\underline{\underline y})$, and to write  $J$  as a block matrix with respect to the coordinate blocks $(y_2,y_1,\underline{\underline y})$ in the form 
\[
J=\begin{pmatrix}
0 &  1& 0 \\
-1 & 0 &0\\
0 & 0& \underline{\underline J} \\
\end{pmatrix},
\]
where $\underline{\underline J}=J_{n-1}$. It will be convenient to consider the matrix $E:=\tilde\Theta_{(\tau,\underline x),y}$ in place of $\tilde\Theta_{(\underline x,\tau),y}$, which, up to sign, has the same determinant. Then
\[
E = 
\begin{pmatrix}
2(y_2-x_2) & 2(y_1-x_1)& 2(\underline{\underline y}-\underline{\underline x})\\ 
 \frac{1}{2} & -2\tau &0\\
0  & 0&  -2\tau I_{d_1-2}  - \tfrac 12\underline{\underline J} 
\end{pmatrix}.
\]
Here, $|\tau|\lesssim 1$, and moreover $|x_1-y_1|\sim 1, |x_2-y_2|\lesssim \ve $ and $|\underline{\underline x}- \underline {\underline y}|\lesssim \ve$. Hence, for $\ve>0$ sufficiently small,
\[
|{\det E}|\sim 1\cdot  |\det (-2\tau I_{d_1-2}  - \tfrac 12\underline{\underline J})|\sim 1,
\]
since $\underline{\underline J}$ is skew-symmetric and  $\underline{\underline J}^2=-I_{d_1-2}.$ Hence $|\det \tilde\Theta_{(\underline x,\tau),y} |\sim 1.$
\end{proof}

\appendix

\section{Proof of the decay estimates in Proposition~\ref{prop:decay-away}}\label{decayproof}

In this section, we shall give the proof of Proposition~\ref{prop:decay-away}. To this end, recall from Section \ref{decay} that the critical points of the phase
\[
\tilde \Phi(x,u,s,\sigma,\tau) = \sigma(s-|x|^2 g(\tau) + 4u\tau)
\]
in \eqref{eq:phase-kl-alt} are given by 
\[
|x|_c(s,\tau): = \sqrt{s} \, \Big|\frac{\sin\tau}{\tau}\Big| \quad\text{and}\quad 
u_c(s,\tau): = - s \, \frac{\tau-\sin\tau\cos\tau}{ 4 \tau^2},
\]
where the factor $\eta_\ell(\tau-k \pi)$ localizes to the union  $J_{k,\ell}$ of two intervals on which $|\tau-k \pi| \sim 2^{-\ell}.$

We assume that $s \in[1,4],$ and set
\[
\Delta x^2 (s,\tau): = |x|^2-|x|_{c}^2 (s,\tau) \quad\text{and}\quad \Delta u(s,\tau): = u-u_c(s,\tau).
\]
Recall finally that $\tau_{k,\ell} \in J_{k,\ell}$ is fixed. 

\begin{lemma}\label{lem:decay-away-1}
Suppose that $k,\ell\ge 1$ and $s \in[1,4],$ Then the following statements hold: 
\begin{enumerate}
 \item If $|x|^2 \leq \frac{1}{100}(\tfrac{2^{-\ell}}{k})^2$ or $|x|^2 \ge 10 \, (\tfrac{2^{-\ell}}{k})^2 $, then
 \[
 |\Delta x^2 (s,\tau)| \gtrsim(\tfrac{2^{-\ell}}{k})^2 \quad\text{for all } \tau \in J_{k,\ell}.
 \]
\item There is some constant $R_0 > 0$ such that, if $\left|\Delta u(s,\tau_{k,\ell})\right| \ge R_0  \tfrac{2^{-\ell}}{k^2}$, then
\[
\left|\Delta u(s,\tau)\right| \gtrsim \tfrac{2^{-\ell}}{k^2} \quad \text{for all } \tau \in J_{k,\ell}.
\]
\end{enumerate}
\end{lemma}

\begin{proof}
Recall that $\left|u_c(s,\tau)-u_c (s,\tau_{k,\ell}) \right| \lesssim \tfrac{2^{-\ell}}{k^2}$ for $\tau\in J_{k,\ell}$.
Thus, the second part of the lemma follows immediately since
\[
|\Delta u(s,\tau)| \ge \left|\Delta u(s,\tau_{k,\ell})\right| -\left|u_c(s,\tau)-u_c(s,\tau_{k,\ell}) \right| \gtrsim \tfrac{2^{-\ell}}{k^2}
\]
if $\left|\Delta u(s,\tau_{k,\ell})\right| \ge R_0  \tfrac{2^{-\ell}}{k^2}$ for sufficiently large $R_0 > 0$.
To show the first part of the lemma, recall that 
\[
\Delta x^2 (s,\tau): = |x|^2-|x|_{c}^2 (s,\tau) = |x|^2-s \, \frac{\sin^2\tau}{\tau^2}.
\]
Thus, we need to bound $(\sin ^2 \tau)/\tau^2 $ from above and below.

Recall from the discussion preceding \eqref{eq:Tlakl-def} that $\eta_\ell(\tau) = \eta_{0}(2^{\ell-1} \tau)-\eta_{0}(2^{\ell} \tau)$ for $\ell \ge 1$, where $\eta_0$ is supported in ${(-\frac{5 \pi}{8}, \frac{5 \pi}{8})}$ with $\eta_0(\tau) = 1$ for ${\tau \in(-\frac{3 \pi}{8}, \frac{3 \pi}{8})}$. Hence, if $\tau\in J_{k,l}$, we can write $\tau = (k \pm 2^{-\ell} \alpha) \pi$ for some $\frac{3}{8} \le \alpha<\frac{5}{4}$.

We distinguish the cases $\ell = 1$ and $\ell \ge 2$. For $\ell =1$ obverse that for $\tau\in J_{k,l}$
\[
\frac{\sin ^2 \tau}{\tau^2 }
= \frac{\sin ^2 (\frac12 \alpha \pi)}{((k \pm \frac12 \alpha) \pi)^2 } 
\le \frac{1}{k^2} \frac{1}{((1-\frac\alpha 2)\pi)^2}
\le \frac{64}{9\pi^2} \frac{1}{k^2}
< \frac{1}{k^2}.
\]
On the other hand,
\[
\frac{\sin ^2 \tau}{\tau^2 }
\ge \frac{1}{k^2} \frac{\sin ^2 (\frac12 \alpha \pi)}{((1+\frac58)\pi)^2}
\le \left.\frac{1}{k^2} \frac{\sin ^2 (\frac12 \alpha \pi)}{((1+\frac58)\pi)^2}\right|_{\alpha=\frac38}
> 0.011\, \frac{1}{k^2}
\]
Now suppose that $k\ge 1$ and $\ell\ge 2$. Using the Taylor series expansion $\sin ^2 \tau = (1-\cos (2 \tau))/2 = \tau^2 -\tau^{ 4 }/{3} \pm \cdots$, we see that
\[
\frac{\sin ^2 \tau}{\tau^2 } = \frac{\sin ^2 (2^{-\ell} \alpha \pi)}{((k \pm 2^{-\ell} \alpha) \pi)^2 }
\sim\Big(\frac{2^{-\ell}}{k}\Big)^2 \frac{\alpha^2 }{(1 \pm \frac{2^{-\ell}}{k} \alpha)^2 }.
\]
The error of this approximation is sufficiently small: Note that
\[
\Big(\frac{d}{d \tau}\Big)^{ 4 } \sin ^2 \tau = \Big(\frac{d}{d \tau}\Big)^{ 4 } \frac{1-\cos (2 \tau)}{2} = 8 \cos (2 \tau).
\]
Hence $|\sin ^2 \tau-\tau^2 | \leq 8 \, | \tau|^{ 4 }/ 4 ! = | \tau|^{ 4 }/3$ for all $\tau\in\R$. Thus
\begin{align*}
\Big|\frac{\sin ^2 (2^{-\ell} \alpha \pi)}{((k \pm 2^{-\ell} \alpha) \pi)^2 }-\Big(\frac{2^{-\ell}}{k}\Big)^2 \frac{\alpha^2 }{(1 \pm \frac{2^{-\ell}}{k} \alpha)^2 }\Big|
& \le \frac{\frac{1}{3}\left(2^{-\ell} \alpha \pi\right)^{ 4 }}{((k \pm 2^{-\ell} \alpha) \pi)^2 } \\
& = \Big(\frac{2^{-\ell}}{k}\Big)^2 \frac{\frac{1}{3}(2^{-\ell} \pi)^2 \alpha^{ 4 }}{(1 \pm \frac{2^{-\ell}}{k} \alpha)^2 }.
\end{align*}
Since $k\ge 1$ and $\ell\ge 2$, we obtain
\[
\begin{aligned}
& \frac{\sin ^2 \tau}{\tau^2 }
\le\Big(\frac{2^{-\ell}}{k}\Big)^2 \frac{\alpha^2 + \frac{1}{3}(2^{-\ell} \pi)^2 \alpha^{ 4 }}{(1 \pm \frac{2^{-\ell}}{k} \alpha)^2 }
\le\Big(\frac{2^{-\ell}}{k}\Big)^2 \frac{\alpha^2 + \frac{1}{3}\left(\frac \pi 4\right)^2 \alpha^{ 4 }}{(1 - \frac{1}{4} \alpha)^2 } \\
& =\Big(\frac{2^{-\ell}}{k}\Big)^2 \frac{\alpha^2 + \frac{1}{48} \pi^2 \alpha^{ 4 }}{(1-\frac{1}{4} \alpha)^2 } \le\left.\Big(\frac{2^{-\ell}}{k}\Big)^2 \frac{\alpha^2 + \frac{1}{48} \pi^2 \alpha^{ 4 }}{(1-\frac{1}{4} \alpha)^2 }\right|_{\alpha = \frac{5}{4}}
< 1.092 \, \Big(\frac{2^{-\ell}}{k}\Big)^2 ,
\end{aligned}
\]
since $\alpha\mapsto 1/(1-\frac{1}{4} \alpha)^2 $ and $\alpha\mapsto \alpha^2 + \frac{1}{48} \pi^2 \alpha^{ 4 } $ are increasing functions.

On the other hand,
\[
\frac{\sin ^2 \tau}{\tau^2 } \geq\Big(\frac{2^{-\ell}}{k}\Big)^2 \frac{\alpha^2 -\frac{1}{3}(2^{-\ell} \pi)^2 \alpha^{ 4 }}{(1 \pm \frac{2^{-\ell}}{k} \alpha)^2 }.
\]
Note that the enumerator on the right hand side above is positive. Hence
\[
\frac{\alpha^2 -\frac{1}{3}(2^{-\ell} \pi)^2 \alpha^{ 4 }}{(1 \pm \frac{2^{-\ell}}{k} \alpha)^2 }
\geq \frac{\alpha^2 -\frac{1}{48} \pi^2 \alpha^{ 4 }}{(1 + \frac{1}{4} \alpha)^2 }
= \frac{\alpha}{(1 + \frac{1}{4} \alpha)^2 }\left(\alpha-\frac{1}{48} \pi^2 \alpha^{3}\right) = :R(\alpha).
\]
Recall that $\frac{3}{8}\le \alpha < \frac 5 4$. Note that
\[
\begin{aligned}
& \frac{d}{d\alpha} \Big(\frac{\alpha}{ (1 + \frac{1}{4} \alpha)^2 } \Big)
 = \frac{(1 + \frac{1}{4} \alpha)^2 -\alpha(1 + \frac{1}{4} \alpha)}{(1 + \frac{1}{4} \alpha)^{ 4 }} = \frac{1-\frac{1}{4} \alpha}{(1 + \frac{1}{4} \alpha)^{3}} > 0, \\
& \frac{d}{d\alpha} \Big(\alpha-\frac{1}{48} \pi^2 \alpha^{3}\Big)
 = 1-\frac{1}{ 16 } \pi^2 \alpha^2 \geq 1-\frac{1}{ 16 } \pi^2 \Big(\frac{5}{4}\Big)^2 = 0.03 \ldots > 0.
\end{aligned}
\]
Hence $R(\alpha) \geq R\left(\frac{3}{16}\right) > 0.028$. Altogether, we have
\[
0.028 \, \Big(\frac{2^{-\ell}}{k}\Big)^2 < \frac{\sin ^2 \tau}{\tau^2 } < 1.092 \, \Big(\frac{2^{-\ell}}{k}\Big)^2 
\quad\text{for all } \tau \in J_{k,\ell}.
\]
This finishes the proof since $\Delta x^2 (s,\tau) = |x|^2-|x|_{c}^2 (s,\tau) = |x|^2-s \, (\sin^2\tau)/\tau^2$.
\end{proof}

Recall next from \eqref{nablaPhi} that
\[
 \nabla_{\sigma,\tau} ^{\mathsf T}\tilde \Phi (x,u,s,\sigma,\tau)
 = A_{\sigma, \tau}\binom{\Delta x^2 (s,\tau)}{\Delta u(s,\tau)}.
 \]
We want to integrate by parts to obtain rapid decay in terms of $\la\Delta x^2 (s,\tau)$ and $\la\Delta u(s,\tau)$. 
\smallskip

To that end, observe that if we choose $a=(a_1,a_2)$ so that $aA_{\sigma, \tau}=(1,0),$ then $(a \nabla_{\sigma,\tau}^{\mathsf T})\tilde \Phi (x,u,s,\sigma,\tau)=\Delta x^2 (s,\tau),$ hence $(a \nabla_{\sigma,\tau}^{\mathsf T})\, e^{i \lambda \tilde \Phi}=i\la \Delta x^2 (s,\tau) \, e^{i \lambda \tilde \Phi}.$ From \eqref{Ainvers} we obtain 
\[
a=\Big(-\frac {\sin^2\tau}{\tau^2}, \frac {\sin^2\tau}{\sigma \tau}\Big)
\]
Thus, if we define the vector field 
\[
V: = -\frac {\sin^2\tau}{\tau^2} \partial_{\sigma}+\frac {\sin^2\tau}{\sigma\tau}\partial_{\tau},
\]
then 
\begin{equation}\label{VPhi}
V( e^{i \lambda \tilde \Phi}) =i\la \Delta x^2 (s,\tau)e^{i \lambda \tilde \Phi}.
\end{equation}
Similarly, if we choose $b=(b_1,b_2)$ so that $bA_{\sigma, \tau}=(0,1),$ then $$(b \nabla_{\sigma,\tau}^{\mathsf T})\tilde \Phi (x,u,s,\sigma,\tau)=\Delta u (s,\tau),$$ hence $(b \nabla_{\sigma,\tau}^{\mathsf T})\, e^{i \lambda \tilde \Phi}=i\la \Delta u(s,\tau) \, e^{i \lambda \tilde \Phi}.$ From \eqref{Ainvers} we obtain 
$$
b=\frac 1 {\sigma\tau^2}\Big(\frac{1}{4}\sigma(\tau-\sin\tau\cos\tau), \frac{1}{4}\tau\sin\tau\cos\tau\Big).
$$
Thus, if we define the vector field 
\[
W: = \varphi(\tau) \partial_{\sigma} + \psi(\sigma, \tau) \partial_{\tau},
\]
where
\[
\varphi(\tau): = \frac{\tau-\sin \tau \cos \tau }{\tau^2 } \quad\text{and}\quad \psi(\sigma, \tau): = \frac{\sin \tau \cos \tau }{\sigma \tau},
\]
then 
\begin{equation}\label{WPhi}
W (e^{i \lambda  \tilde\Phi} )=4 i\la \Delta u(s,\tau)e^{i \lambda \tilde \Phi}.
\end{equation}

Recall also from \eqref{eq:K-kl-alt} that 
$$
K_{\lambda,s}^{k,\ell}(x, u) = \lambda^{n + \frac 32} \int_0^\infty \int_0^\infty e^{i\lambda\tilde \Phi(x,u,s,\sigma,\tau)} \, a_\lambda(s\sigma) \, \eta_\ell(\tau -k \pi)\Big(\frac{\tau}{\sin\tau}\Big)^n \, d\sigma \, d\tau.
$$

\begin{lemma}\label{lem:decay-away-2}
For every $N,M\in\N$, the kernel $K_{\lambda,s}^{k,\ell}$ can be written as a finite series
\[
K_{\lambda,s}^{k,\ell} = \sum_{\alpha\in \N^2 \times \Z\times \N^6} K_{\lambda,s,N,M,\alpha}^{k,\ell}
\]
of summands of the form
\begin{multline*}
K_{\lambda,s,N,M,\alpha}^{k,\ell}(x, u)
 = 
\lambda^{n + \frac{3}{2}-N-M} \int_0^\infty \int_0^\infty  e^{i \lambda \tilde \Phi} \,\partial_\sigma^{\alpha_1} (a_{\lambda}(s \sigma))\, \partial_\tau^{\alpha_2}(\eta_\ell(\tau-k \pi))\\
\times \left(\frac{\tau}{\sin \tau}\right)^{n} \frac{\sigma^{\alpha_3} \left(\sin \tau\right)^{\alpha_5} P_{\alpha_6}(\cos\tau,\sin\tau)}{\tau^{\alpha_ 4 } \left( \Delta x^2 (s,\tau) \right)^{\alpha_7}\left( \Delta u(s,\tau) \right)^{\alpha_8} } \, s^{\alpha_9} \, d\sigma \, d \tau,
\end{multline*}
where $P_{\alpha_6}\in \Z[X,Y]$ is a polynomial with integer coefficients and $\alpha = (\alpha_1,\dots,\alpha_9)\in \N^2 \times \Z\times \N^6$ satisfies the relations
\begin{itemize}
 \item $\alpha_ 4 -2\alpha_7-2\alpha_8\ge -N-M$,
 \item $\alpha_5-\alpha_2-2\alpha_7-\alpha_8\ge -N-M$,
 \item $\alpha_7\ge N$ and $\alpha_8\ge M$.
\end{itemize}
Moreover, if $N = 0$, then $\alpha_7 = 0$, and if $M = 0$, then $\alpha_8 = 0$.
\end{lemma}

\begin{proof}
We do induction over $|N| + |M|$. For $|N| + |M| = 0$ we have $K_{\lambda,s}^{k,\ell} = K_{\lambda,\alpha}^{k,\ell}$ if $P_{\alpha_6} = 1$ and $\alpha_j = 0$ for all $j$.

 Now suppose that the statement holds for the tuple $(N,M)$. We will show that the statement then also holds for $(N + 1,M)$ and $(N,M + 1)$. We use the notation
\begin{multline*}
K_{\lambda,s,N,M,\alpha}^{k,\ell}[E](x, u)
 = \lambda^{n + \frac{3}{2}-N-M} \int_0^\infty \int_0^\infty E(x,u,s,\tau,\sigma) \, \partial_\sigma^{\alpha_1} (a_{\lambda}(s \sigma)) \\
 \times \partial_\tau^{\alpha_2}(\eta_\ell(\tau-k \pi))
 \left(\frac{\tau}{\sin \tau}\right)^{n} \frac{\sigma^{\alpha_3} \left(\sin \tau\right)^{\alpha_5} P_{\alpha_6}(\cos\tau,\sin\tau)}{ \tau^{\alpha_ 4 } \left( \Delta x^2 (s,\tau) \right)^{\alpha_7}\left( \Delta u(s,\tau) \right)^{\alpha_8} } \, s^{\alpha_9} \, d\sigma \, d \tau.
\end{multline*}
With this notation, $K_{\lambda,s,N,M,\alpha} = K_{\lambda,s,N,M,\alpha}^{k,\ell}[e^{i\lambda \tilde \Phi}]$.
 By \eqref{VPhi}, this yields
\begin{align*}
& K_{\lambda,s,N,M,\alpha}^{k,\ell}[i\,e^{i\lambda \tilde \Phi}]
 = K_{\lambda,s,N + 1,M,\alpha + e_7}^{k,\ell}[V(e^{i\lambda \tilde \Phi})] \\ 
& = -K_{\lambda,s,N + 1,M,\alpha + 2e_ 4 + 2e_5 + e_7}^{k,\ell}[\partial_\sigma(e^{i\lambda \tilde \Phi})] + K_{\lambda,s,N + 1,M,\alpha-e_3 + e_ 4 + 2e_5 + e_7}^{k,\ell}[\partial_\tau(e^{i\lambda \tilde \Phi})].
\end{align*}
When integrating by parts  with respect to $\sigma$ in the first integral and employing the Leibniz  rule, we see that $\partial_\sigma$ either hits $\partial_\sigma^{\alpha_1} (a_{\lambda}(s \sigma))$, which increases $\alpha_1$, or $\sigma^{\alpha_3}$, which decreases $\alpha_3$. By induction hypothesis,
\begin{align*}
(\alpha_ 4 + 2)-2(\alpha_7 + 1)-2\alpha_8 & > -(N + 1)-M, \\
(\alpha_5 + 2)-\alpha_2-2(\alpha_7 + 1)-\alpha_8 & > -(N + 1)-M.
\end{align*}
For the second integral, when  integrating  by parts is $\tau$ and  employing the Leibniz   rule again, we make the following observations:
\begin{itemize}
\item[(i)] When $\partial_\tau$ hits $\partial_\tau^{\alpha_2}(\eta_\ell(\tau-k \pi))$, then $\alpha_2$ is additionally increased by 1. However, the induction hypothesis still implies
\[
(\alpha_5 + 2)-(\alpha_2 + 1)-2(\alpha_7 + 1)-2\alpha_8 \ge -(N + 1)-M.
\]
\item[(ii)] When $\partial_\tau$ hits $\tau^{n}$ or $1/\tau^{\alpha_ 4 }$, then $\alpha_ 4 $ is additionally increased by 1. This additional gain is not relevant for our purposes.

\item[(iii)] When $\partial_\tau$ hits $1/(\sin\tau)^{n}$ or $(\sin \tau)^{\alpha_5}$, then $\alpha_5$ changes additionally by  $-1$, but $\alpha_2$ is unchanged, and 
\[
(\alpha_5 + 1)-\alpha_2-2(\alpha_7 + 1)-\alpha_8  \ge -(N + 1)-M.
\]
\item[(iv)] When $\partial_\tau$ hits $P_{\alpha_6}(\cos\tau,\sin\tau)$, we just pass to another polynomial.

\item[(v)] When $\partial_\tau$ hits $(\Delta x^2 (s,\tau))^{-(\alpha_7+1)}$, we get
\[
\partial_\tau (\Delta x^2 (s,\tau))^{-(\alpha_7+1)}
 =( \alpha_7+1) (\Delta x^2 (s,\tau))^{-(\alpha_7 + 1)} \partial_\tau\Big( s \, \frac{\sin^2\tau}{\tau^2} \Big).
\]
Thus $\alpha_7$ is increased in total by 2. If $\partial_\tau$ hits the enumerator  of the factor $\sin^2\tau/\tau^2$, $\alpha_ 4 $ and $\alpha_5$ are both increased in total by 3, and if $\partial_\tau$ hits the denominator   of the factor $\sin^2\tau/\tau^2$, then both $\alpha_ 4 $ and $\alpha_5$ are increased in total by 4.   Note that
\begin{align*}
(\alpha_ 4 + 3 )-2(\alpha_7 + 2)-2\alpha_8 & \ge -(N + 1)-M, \\
(\alpha_5 + 3)-\alpha_2-2(\alpha_7 + 2)-\alpha_8 & \ge -(N + 1)-M,
\end{align*}
and if we increase both $\alpha_4$ and $\alpha_5$ by 4, we even get strict inequalities.

\item[(vi)] When $\partial_\tau$ hits $(\Delta u(s,\tau))^{-\alpha_8}$, we get
\[
\partial_\tau \left(\Delta u(s,\tau)\right)^{-\alpha_8}
 = \alpha_8 \left(\Delta u(s,\tau)\right)^{-(\alpha_8 + 1)} \partial_\tau\Big( s \, \frac{\tau-\sin\tau\cos\tau}{4\tau^2} \Big).
\]
Similar as in (v), $\alpha_8$ is increased by 1, while the factor $\partial_\tau( (\tau-\sin\tau\cos\tau)/\tau^2 )$ leads to two summands where $\alpha_ 4 $ or $\alpha_5$ is additionally increased.  If $\partial_\tau$ hits $1/\tau,$ or the enumerator of $-\sin\tau\cos\tau/\tau^2,$ then $\alpha_4$ is increased in total by $3$. Note here that
\begin{align*}
(\alpha_ 4 + 3 )-2(\alpha_7 + 1)-2(\alpha_8 + 1) &\ge -(N + 1)-M,\\
(\alpha_5 + 2)-\alpha_2-2(\alpha_7 + 1)-(\alpha_8 + 1) &\ge -(N + 1)-M.
\end{align*}
If $\partial_\tau$ hits  the denominator of $-\sin\tau\cos\tau/\tau^2,$ then $\alpha_4$ is increased in total by $4,$ and we even get strict inequalities.
\end{itemize}

This verifies the statement for $(N + 1,M)$. We can argue similarly for $(N,M + 1)$ by using \eqref{WPhi}. 
This yields
\begin{multline*}
 K_{\lambda,s,N,M,\alpha}^{k,\ell}[4i\,e^{i\lambda \tilde \Phi}]
 = K_{\lambda,s,N,M + 1,\alpha + e_8}[W(e^{i\lambda \tilde \Phi})] \\
 = K_{\lambda,s,N,M + 1,\alpha + e_ 4 + e_8}[\partial_\sigma(e^{i\lambda \tilde \Phi})]
-K_{\lambda,s,N,M + 1,\beta + 2e_ 4 + e_5 + e_8}[\partial_\sigma(e^{i\lambda \tilde \Phi})] \\
+ K_{\lambda,s,N,M + 1,\gamma - e_3 + e_ 4 + e_5 + e_8}[\partial_\tau(e^{i\lambda \tilde \Phi})]
\end{multline*}
with some new polynomials $P_{\beta_6},P_{\gamma_6}\in \Z[X,Y]$ and $\alpha_j = \beta_j = \gamma_j$ for all $j\neq 6$. For the first two summands, observe that
\begin{align*}
(\alpha_ 4 + 1)-2\alpha_7-2(\alpha_8 + 1) &\ge -N-(M + 1),\\
(\alpha_5 + 1)-\alpha_2-2\alpha_7-(\alpha_8 + 1) & > -N-(M + 1).
\end{align*}
For the third summand, we can argue similarly to (i)--(iv) above. For case (v), observe that
\begin{align*}
(\alpha_ 4 + 3)-2(\alpha_7 + 1)-2(\alpha_8 + 1) &\ge -N-(M + 1), \\
(\alpha_5 + 2)-\alpha_2-2(\alpha_7 + 1)-(\alpha_8 + 1) &\ge -N-(M + 1).
\end{align*}
For case (vi), observe that
\begin{align*}
(\alpha_ 4 + 3)-2\alpha_7-2(\alpha_8 + 2) &\ge -N-(M + 1), \\
(\alpha_5 + 1)-\alpha_2-2\alpha_7-(\alpha_8 + 2) &\ge -N-(M + 1).
\end{align*}
Altogether, we have verified the statement also for $(N,M+1),$ which  finishes the induction.
\end{proof}

\begin{proof}[Proof of Proposition~\ref{prop:decay-away}]
We use polar coordinates with $\rho = |x|$. Note that $K_{\lambda,s}^{k,l}$ and $K_{\lambda,s,N,M,\alpha}^{k,\ell}$ are radial symmetric in $x$. In the following, we write $K_{\lambda,s}^{k,l}(x,u) = \mathcal K_{\lambda,s}^{k,l} (\rho,u) $ and $K_{\lambda,s,N,M,\alpha}^{k,\ell}(x,u) = \mathcal K_{\lambda,s,N,M,\alpha}^{k,\ell}(\rho,u)$.
Let
\begin{align*}
I_{\lambda,1}^{k,\ell} &: = \{ \rho\in (0,\infty):
\tfrac{1}{100}(\tfrac{2^{-\ell}}{k})^2 < \rho^2 < 10 \, (\tfrac{2^{-\ell}}{k})^2 
\} \\
I_{\lambda,2}^{k,\ell}(s,\tau_{k,\ell}) &: = 
\{ u \in \R:
\left|\Delta u (s,\tau_{k,\ell})\right| < R_0  \tfrac{2^{-\ell}}{k^2}\}.
\end{align*}
Moreover, let 
\begin{align*}
E_{\lambda,1}^{k,\ell}(s,\tau_{k,\ell})
&: = \{(\rho,u)\in(0,\infty) \times \R :\rho\notin I_{\lambda,1}^{k,\ell} \text{ and } u \in I_{\lambda,2}^{k,\ell}(s,\tau_{k,\ell}) \},\\
E_{\lambda,2}^{k,\ell}(s,\tau_{k,\ell})
&: = \{(\rho,u)\in(0,\infty) \times \R :\rho\notin I_{\lambda,1}^{k,\ell} \text{ and } u \notin I_{\lambda,2}^{k,\ell}(s,\tau_{k,\ell}) \},\\
E_{\lambda,3}^{k,\ell}(s,\tau_{k,\ell})
&: = \{(\rho,u)\in(0,\infty) \times \R : \rho\in I_{\lambda,1}^{k,\ell} \text{ and } u \notin I_{\lambda,2}^{k,\ell}(s,\tau_{k,\ell}) \}.
\end{align*}
It suffices to prove the bound 
\[
\iint_{E_{\lambda,j}^{k,\ell}(s,\tau_{k,\ell})} |\mathcal K_{\lambda,s}^{k,\ell}(\rho,u)| \, \rho^{d_1-1} d\rho \, du\le C_N
\Big(\frac{\lambda}{2^\ell k}\Big)^{-N}\Big(\frac{2^{\ell}}{k}\Big)^{\frac{1}{2}}\quad \text{for all } j\in\{1,2,3\}.
\]
Suppose that $j\in\{1,2,3\}$ is fixed. Let $N,M\in\N$. We will choose $N$ and $M$ later depending on $j$. By Lemma~\ref{lem:decay-away-2}, given $N,M\in\N$, we may write the kernel $\mathcal K_{\lambda,s}^{k,\ell}$ as a finite series
\[
\mathcal K_{\lambda,s}^{k,\ell} = \sum_{\alpha\in \N^2 \times \Z\times \N^6} \mathcal K_{\lambda,s,N,M,\alpha}^{k,\ell},
\]
where $P_{\alpha_6}\in \Z[X,Y]$ and $\alpha = (\alpha_1,\dots,\alpha_7)\in \N^2 \times \Z\times \N^6$ satisfies
\begin{itemize}
 \item $\alpha_ 4 -2\alpha_7-2\alpha_8\ge -N-M$,
 \item $\alpha_5-\alpha_2-2\alpha_7-\alpha_8\ge -N-M$,
 \item $\alpha_7\ge N$ and $\alpha_8\ge M$.
\end{itemize}
Clearly, it suffices to bound each of the summands $\mathcal K_{\lambda,s,N,M,\alpha}^{k,\ell}$ separately. We have
\begin{multline}
 \iint_{E_{\lambda,j}^{k,\ell}(s,\tau_{k,\ell})} | \mathcal K_{\lambda,s,N,M,\alpha}^{k,\ell}(\rho,u)| \, \rho^{d_1-1} d\rho \, du\\
 = \lambda^{n + \frac{3}{2}-N-M} \iint_{E_{\lambda,j}^{k,\ell}(s,\tau_{k,\ell})}\bigg| \int_0^\infty \int_0^\infty  e^{i \lambda \tilde \Phi} \partial_\sigma^{\alpha_1} (a_{\lambda}(s \sigma)) \partial_\tau^{\alpha_2}(\eta_\ell(\tau-k \pi))\left(\frac{\tau}{\sin \tau}\right)^{n}  \\
  \hskip3cm\times \frac{\sigma^{\alpha_3} \left(\sin \tau\right)^{\alpha_5} P_{\alpha_6}(\cos\tau,\sin\tau)}{\tau^{\alpha_ 4 } \left( \Delta \rho^2(s,\tau) \right)^{\alpha_7}\left( \Delta u (s,\tau) \right)^{\alpha_8} } \, s^{\alpha_9} \, d\sigma \, d \tau \bigg| \, \rho^{d_1-1} d\rho \, du.\label{KNM}
\end{multline}
Let
\begin{align*}
\tilde I_{\lambda,1}^{k,\ell}(s,\tau) &: = \{ \rho\in (0,\infty):
\left|\Delta \rho^2(s,\tau)\right| \lesssim (\tfrac{2^{-\ell}}{k})^2 
\} , \\
\tilde I_{\lambda,2}^{k,\ell}(s,\tau) &: = 
\{ u \in \R:
\left|\Delta u (s,\tau)\right| \lesssim \tfrac{2^{-\ell}}{k^2}\}.
\end{align*}
By Lemma~\ref{lem:decay-away-1}, we have $(I_{\lambda,1}^{k,\ell})^c\subseteq (\tilde I_{\lambda,1}^{k,\ell}(s,\tau))^c$ and $(I_{\lambda,2}^{k,\ell}(s,\tau_{k,\ell}))^c\subseteq (\tilde I_{\lambda,2}^{k,\ell}(s,\tau))^c$.   For $\tau\in J_{k,\ell}$ let 
\begin{align*}
\tilde E_{\lambda,1}^{k,\ell}(s,\tau)
&: = \{(\rho,u)\in(0,\infty) \times \R :\rho\notin \tilde I_{\lambda,1}^{k,\ell}(s,\tau) \text{ and } u \in \tilde I_{\lambda,2}^{k,\ell}(s,\tau) \},\\
\tilde E_{\lambda,2}^{k,\ell}(s,\tau)
&: = \{(\rho,u)\in(0,\infty) \times \R :\rho\notin \tilde I_{\lambda,1}^{k,\ell}(s,\tau) \text{ and } u \notin \tilde I_{\lambda,2}^{k,\ell}(s,\tau) \},\\
\tilde E_{\lambda,3}^{k,\ell}(s,\tau)
&: = \{(\rho,u)\in (0,\infty) \times \R : \rho^2\sim (\tfrac{2^{-\ell}}{k})^2 \text{ and } u \notin \tilde I_{\lambda,2}^{k,\ell}(s,\tau)\}.
\end{align*}
Note that $\rho\in I_{\lambda,1}^{k,\ell}$ implies $\rho^2\sim (\tfrac{2^{-\ell}}{k})^2 $. On the other hand, if $ u \in I_{\lambda,2}^{k,\ell}(s,\tau_{k,\ell})$, then
\[
|\Delta u (s,\tau)| \le |\Delta u (s,\tau_{k,\ell} )| + | u _c(s,\tau)- u _c(s,\tau_{k,\ell} )| \lesssim \tfrac{2^{-\ell}}{k^2}.
\]
Thus $I_{\lambda,2}^{k,\ell}(s,\tau_{k,\ell})\subseteq \tilde I_{\lambda,2}^{k,\ell}(s,\tau)$, hence  
$E_{\lambda,1}^{k,\ell}(s,\tau_{k,\ell}) \subseteq \tilde E_{\lambda,1}^{k,\ell}(s,\tau).$
 
 Similarly, if $ u \notin I_{\lambda,2}^{k,\ell}(s,\tau_{k,\ell})$, then by choosing $R_0$ sufficiently large, we  have 
 \[
|\Delta u (s,\tau)| \ge |\Delta u (s,\tau_{k,\ell} )| - | u _c(s,\tau)- u _c(s,\tau_{k,\ell} )| \gtrsim \tfrac{2^{-\ell}}{k^2}, 
\]
hence 
$(I_{\lambda,2}^{k,\ell}(s,\tau_{k,\ell}))^c\subseteq (\tilde I_{\lambda,2}^{k,\ell}(s,\tau))^c$, hence  
$E_{\lambda,2}^{k,\ell}(s,\tau_{k,\ell}) \subseteq \tilde E_{\lambda,2}^{k,\ell}(s,\tau).$

In a similar way, we see that $E_{\lambda,3}^{k,\ell}(s,\tau_{k,\ell}) \subseteq \tilde E_{\lambda,3}^{k,\ell}(s,\tau)$.

Hence, for every $j\in\{1,2,3\}$,
\[
E_{\lambda,j}^{k,\ell}(s,\tau_{k,\ell}) \subseteq \tilde E_{\lambda,j}^{k,\ell}(s,\tau).
\]
Using Fubini's theorem, we may bound the term in \eqref{KNM}  by
\begin{multline*}
\lambda^{n + \frac{3}{2}-N-M} \int_{|\tau-k\pi|\sim 2^{\ell}} \int_{\sigma\sim 1} \iint_{\tilde E_{\lambda,j}^{k,\ell}(s,\tau)} \frac{(2^{\ell})^{\alpha_2} (2^\ell k)^{n} (2^{-\ell})^{\alpha_5} }{k^{\alpha_ 4 } \left| \Delta \rho^2(s,\tau) \right|^{\alpha_7}\left| \Delta u (s,\tau) \right|^{\alpha_8} } \\ \times \rho^{d_1-1} d\rho \, d u \, d\sigma \, d \tau.
\end{multline*}
Recall that $\Delta \rho^2(s,\tau) = \rho^2 - \rho^2_c(s,\tau)$ with $\rho^2_c(s,\tau)\sim (\tfrac{2^{-\ell}}{k})^2$, and $\Delta u (s,\tau) = u - u _c(s,\tau)$ with $u_c(s,\tau)\sim k^{-1}.$

Thus, if we choose $N,M$ large enough, since $\alpha_7\ge N$ and $\alpha_8\ge M$, we obtain via homogeneity
\[
\int_{\left|\Delta \rho^2(s,\tau)\right| \gtrsim (\frac{2^{-\ell}}{k})^2 } | \Delta \rho^2(s,\tau) |^{-\alpha_7} \rho^{d_1-1} d\rho \sim \Big(\frac{2^{-\ell}}{k}\Big)^{-2\alpha_7 + d_1},
\]
and 
\[
\int_{ |\Delta u (\tau)| \gtrsim \frac{2^{-\ell}}{k^2} } | \Delta u (s,\tau) |^{-\alpha_8} d u \sim \Big(\frac{2^{-\ell}}{k^2}\Big)^{-\alpha_8 + 1}.
\]
Suppose $j = 1$. In that case, we choose $M = 0$ and $N$ sufficiently large. By Lemma~\ref{lem:decay-away-2}, we then have $\alpha_8 = 0$, $\alpha_ 4 -2\alpha_7\ge -N$ and $\alpha_5-\alpha_2-2\alpha_7\ge -N$. Thus, the above term is bounded by
\begin{align*}
&\lambda^{n + \frac{3}{2}-N}
\frac{(2^{\ell})^{\alpha_2} (2^\ell k)^{n} (2^{-\ell})^{\alpha_5} }{k^{\alpha_ 4 } }
\Big(\frac{2^{-\ell}}{k}\Big)^{-2\alpha_7 + d_1}
\frac{2^{-\ell}}{k^2} \\
& = \lambda^{n + \frac{3}{2}-N} \frac{(2^{-\ell})^{\alpha_5-\alpha_2 - 2 \alpha_7 + 1 + d_1-n}}{k^{\alpha_ 4 -2\alpha_7+ d_1-n + 2}} 
 \le \Big( \frac{\lambda}{2^\ell k} \Big)^{n + \frac{3}{2}-N} \Big(\frac{2^{\ell}}{k}\Big)^{\frac 12}.
\end{align*}
Suppose $j = 2$. In that case, we choose $N$ and $M$ sufficiently large. Since $\alpha_ 4 -2\alpha_7-2\alpha_8\ge -N-M$ and $\alpha_5-\alpha_2-2\alpha_7-\alpha_8\ge -N-M$, we obtain the bound
\begin{align*} 
& \lambda^{n + \frac{3}{2}-N-M}
\frac{(2^{\ell})^{\alpha_2} (2^\ell k)^{n} (2^{-\ell})^{\alpha_5} }{k^{\alpha_ 4 } }
\Big(\frac{2^{-\ell}}{k}\Big)^{-2\alpha_7 + d_1}
\Big(\frac{2^{-\ell}}{k^2}\Big)^{-\alpha_8 + 1}\\
& = \lambda^{n + \frac{3}{2}-N-M} \frac{(2^{-\ell})^{\alpha_5-\alpha_2-2\alpha_7-\alpha_8 +1+ d_1-n}}{k^{\alpha_ 4 -2\alpha_7-2\alpha_8 + d_1-n + 2}} 
\le \Big( \frac{\lambda}{2^\ell k} \Big)^{n + \frac{3}{2}-N-M} \Big(\frac{2^{\ell}}{k}\Big)^{\frac 12}.
\end{align*}
Suppose $j = 3$. In that case, we choose $N = 0$ and $M$ sufficiently large. By Lemma~\ref{lem:decay-away-2}, we have $\alpha_7 = 0$, $\alpha_ 4 -2\alpha_8\ge -M$ and $\alpha_5-\alpha_2-\alpha_8\ge -M$. Thus, we obtain the bound 
\begin{align*} 
& \lambda^{n + \frac{3}{2}-M}
\frac{(2^{\ell})^{\alpha_2} (2^\ell k)^{n} (2^{-\ell})^{\alpha_5} }{k^{\alpha_ 4 } }
\Big(\frac{2^{-\ell}}{k}\Big)^{d_1} 
\Big(\frac{2^{-\ell}}{k^2}\Big)^{-\alpha_8 + 1} \\
& = \lambda^{n + \frac{3}{2}-M} \frac{(2^{-\ell})^{\alpha_5-\alpha_2-\alpha_8 +  1 +d_1-n}}{k^{\alpha_ 4 -2\alpha_8 + d_1-n + 2}} 
\le \Big( \frac{\lambda}{2^\ell k} \Big)^{n + \frac{3}{2}-M} \Big(\frac{2^{\ell}}{k}\Big)^{\frac 12}.
\end{align*}
This finishes the proof of Proposition~\ref{prop:decay-away}.
\end{proof}

\section{The rank $d_1-1$ cone condition  on  Heisenberg  groups}\label{cinecurv}

Recall from Proposition~\ref{lem:decomp}  that the convolution kernels $\sK_{\lambda,t}^{0}(x, u)$ and $\sK_{\lambda,t}^{k,\ell}(x, u)$ are oscillatory integrals, with common phase
\[
\Phi(x,u,t,\sigma,\tau)
=\sigma\big(t^2-|x|^2 g(\tau)+4u\tau\big),
\quad \text{with } g(\tau)=\tau\cot\tau,
\]
where $\sigma\sim 1$ and $\tau\in \R.$
Thus, the phase $\Psi$ of the integral kernel of the corresponding convolution operators is given by
\begin{align*}
\Psi(x,u, y,v,t,\sigma, \tau)
& :=\Phi((y,v)^{-1}(x,u),t,\sigma,\tau) \\
& \phantom{:}= \Phi(x-y,u-v+\tfrac 1 2 \langle Jx,y\rangle,t,\sigma,\tau).
\end{align*}
Thus, 
$$
\Psi(x,u, y,v,t,\sigma, \tau)=\sigma\Big( t^2-|x-y|^2 g(\tau) +4\big (u-v+\tfrac 1 2 \langle Jx,y\rangle\big)\tau\Big).
$$
Let 
$$A_\Psi:=\{(x,u, y,v,t,\sigma, \tau): \Psi'_{\sigma,\tau}(x,u, y,v,t,\sigma, \tau)=0,\, \sigma\sim 1\},$$
and set
$$
\Gamma:=\{(x,u,t,\Psi'_x,\Psi'_u,\Psi'_t): (x,u, y,v,t,\sigma, \tau)\in A_{\Psi}\}.
$$
For $(x,u,t)$ fixed, with, say, $t>0,$  let $\Gamma_{(x,u,t)}$ be  the corresponding  set 
\[
\Gamma_{(x,u,t)}:=\{(\Psi'_x,\Psi'_u,\Psi'_t): (x,u, y,v,t,\sigma, \tau)\in A_{\Psi}\}.
\]
We want to show the following.

\begin{lemma}\label{cine}
Let $(x,u)\in\bbH_n$ and $t>0$. Then $\Gamma_{(x,u,t)}$ is an open part of a conic surface which has $d_1-1=2n-1$ non-vanishing principal curvatures at every point.
\end{lemma}

\begin{proof}
In view of the left-invariance of the sub-Laplacian $L$, it suffices to prove this at the point $(x,u)=(0,0)$. We compute the critical points of $\Psi$ with respect to the frequency variables $\sigma,\tau$ at $(x,u)=(0,0)$. Since $\sigma>0$, these are given by the solutions to the equations
\begin{align}\label{crit1}
t^2-|y|^2 g(\tau) -4v\tau & = 0,\\
   -|y|^2 g'(\tau) -4v & = 0. \label{crit2}
\end{align}
Moreover,
\[
\Psi'_x = 2\sigma\big(g(\tau)y-\tau Jy\big),\qquad
\Psi'_u = 4\sigma \tau,\qquad
\Psi'_t = 2\sigma t.
\]
By \eqref{crit2}, $4v=-|y|^2 g'(\tau),$ so by \eqref{crit1} and \eqref{htau},
\[
0=t^2-|y|^2 \big(g(\tau)-\tau g'(\tau)\big)=t^2-|y|^2 h(\tau) =t^2-|y|^2\Big(\frac {\tau}{\sin \tau}\Big)^2.
\]
Hence, using polar coordinates $y=\rho\,\omega$, where $\rho:=|y|$ and $\omega\in S^{d_1-1}$, we obtain
\begin{equation}\label{eq:crit-y}
y=t\,\Big|\frac{\sin\tau}{\tau}\Big|\,\omega.
\end{equation}
This equation should be solved as a smooth function of $\rho.$  Since the mapping $\tau \mapsto (\sin \tau)/\tau$ is by far not injective, we can  find  such solutions $\tau=q_{\tau_*}(\rho)$ only locally near a fixed point $\tau_*$ (with the exception of a discrete set of points $\tau_*$ at which the derivative of   $\tau \mapsto (\sin \tau)/\tau$ vanishes).

Then, \[ 4v=-|y|^2g'(\tau) =-t^2\Big(\frac{\sin\tau}{\tau}\Big)^2g'(\tau). \] 
Thus, if  $(\sin \tau)/\tau>0,$ say, using \eqref{eq:crit-y}, we obtain
\[
\Psi'_x=2\sigma t\left( \cos(\tau) \, \omega -  \sin(\tau) \, J\omega\right)
=2\sigma t \, e^{-\tau J}\omega.
\]
So this local part of $\Gamma_{(0,0,t)}$ is of the form 
\begin{align*}
\Gamma^{\tau_*}_{(0,0,t)}&= \Big\{2\sigma \Big(t\, e^{-q_{\tau_*}(|y|)J}\frac y{|y|},2q_{\tau_*}(|y|),t\Big): a_*<|y|<b_*\Big\}\\
&= \Big\{2\sigma \Big(t\frac y{|y|},2q_{\tau_*}(|y|),t\Big): a_*<|y|<b_*\Big\},
\end{align*}
where $a_*,b_*$ depend on $\tau_*.$
In particular, $\Gamma^{\tau_*}_{(0,0,t)}$ is an open subset of the cone
\begin{equation}\label{conetilde}
\tilde \Gamma:=\big\{\sigma \big(t \omega,\tau,t\big): \sigma>0, \tau\in\R, \omega\in S^{d_1-1}\big\}\subset \R^{d}\times \R.
\end{equation}
The   section  of this cone at $\sigma=1$ is given by $tS^{d_1-1}\times \R\times\{t\},$ so that it has exactly $d_1-1$ non-vanishing principal curvatures at every point.
\end{proof}

\bibliographystyle{amsplain}

\end{document}